\documentclass[a4paper,12pt]{article}

\usepackage[austrian,american]{babel}

\usepackage[intlimits,leqno]{amsmath}
\usepackage{amssymb}
\usepackage{amsthm}
\usepackage{graphicx,color}

\theoremstyle{plain}
\newtheorem{*theo}{*Theorem}
\newtheorem{theo}{Theorem}

\newtheorem{coro}{Corollary}
\newtheorem{lemm}{Lemma}

\newtheorem{rema}{Remark}
\newtheorem{exam}{Example}
\newtheorem{defi}{Definition}

\renewcommand{\i}{{\rm i}}
\newcommand\R{{\mathbb{R}}}
\newcommand\N{{\mathbb{N}}}
\newcommand\C{{\mathbb{C}}}

\newcommand\Z{{\mathbb{Z}}}
\newcommand{\F}{\mathcal{F}}

\newcommand{\x}{\mathbf{x}}

\renewcommand{\k}{\mathbf{k}}
\newcommand{\y}{\mathbf{y}}

\newcommand\e{{\mathbf{e}}}

\newcommand\vd{{\mathbf{d}}}
\newcommand\vO{{\mathbf{0}}}

\usepackage{makeidx}
\usepackage{epsfig}

\usepackage{latexsym}

\usepackage{pstricks,pst-node,amsfonts}

\renewcommand{\d}{{\rm d}}

\newcommand{\supp}{\mbox{supp}}

\begin{document}

\title{Causal diffusion and its backwards diffusion problem}

\author{ Richard Kowar\\
Department of Mathematics, University of Innsbruck, \\
Technikerstrasse 21a/2, A-6020,Innsbruck, Austria
}

\maketitle

\begin{abstract}
In this article we consider the backwards diffusion problem for a \emph{causal} diffusion 
model developed in~\cite{Kow11}. Here causality means that the speed of propagation of the 
concentration is finite. 
For the investigation of this inverse problem, we derive an 
analytic representation of the Green function of the causal diffusion model in the 
$\mathbf{k}-t-$domain  (wave vector-time domain). 
We perform a theoretical and numerical comparison between standard (and noncausal) diffusion 
with our diffusion model in the  $\mathbf{k}-t-$domain and in the $\mathbf{x}-t-$domain. 
Moreover, we prove that the backwards diffusion problem of the causal diffusion model is ill-posed, 
but not exponentially ill-posed. 
In contrast to the classical backwards diffusion problem, the forward operator 
of the causal direct problem is not compact if the space dimension is $1$. 
The paper is concluded with numerical simulations of the backwards diffusion problem 
via the Landweber method. 
\end{abstract}

\section{Introduction}

The standard model of the backwards diffusion problem is a case example of an \emph{exponentially 
ill-posed} problem and it is related to several problems in tomography and image processing. 
For example, such problems have been studied in the  
articles~\cite{EngRun95,ElaIsa97,Do98,IsaKin00,MarSch01,GrKlLu03,Ahm07,HeHuMc08,WeChSuLi10,Add11}  
and books~\cite{Isa90,EngHanNeu96,Kir96,Isa98,Wei98a,NatWue01,Lo01,GuiMorRya04,Sa06,ScGrGrHaLe09,
ScGrGrHaLe09} to name but a few. 

This article starts over the backwards diffusion problem by replacing the \emph{standard} 
diffusion equation by a \emph{causal} diffusion model. Recently, such 
a model has been developed and studied in~\cite{Kow11}. 
By causality we understand that a characteristic 
feature of a process like an \emph{interface} or a \emph{front} must propagate 
with a finite speed $c$. This means that if a ``point concentration'' is added to a solution in point 
$\x=\vO$, then the concentration is zero outside the ball $\overline{B_{c\,T}(\vO)}$ after the 
time period $T$. Here it is not relevant whether this interface or front is 
visible. 
Although a causal behavior is naturally demanded for problems involving hyperbolic 
equations, it is usually disregarded for problems involving parabolic equations. 
It seem to the author that the modeling of causal equations is quite difficult 
and unfortunately it is considered as insignificant.  
Indeed, if a direct problem is smoothing or damping, then after a sufficiently 
long time period it does not matter if the exact or the perturbed model is used. 
However, the situation is quite different for the respective inverse and \emph{ill-posed} problem 
for which data and \emph{modeling errors} have a \emph{strong impact on the solution}. 
Hence, although our causal diffusion model yields similar numerical values as the standard 
diffusion model (compare Fig.~\ref{fig:ex01b} and Fig.~\ref{fig:ex01c}), it seems evident that causal 
diffusion is of practical interest for inverse problems related to diffusion. 

Apart from this fact, the modeling and investigation of causal mathematical equations 
is interesting from the pure mathematical point of view.  

The goal of this paper is to investigate to what extent a \emph{causal} diffusion model influences 
the respective backwards diffusion problem. We show that this inverse problem is ill-posed, 
but not exponentially ill-posed. 
For this purpose we derive basic properties of the causal diffusion model (cf. 
Section~\ref{sec-cdiff} and the appendix) which  can be summaries as follows.  
If $v$ denotes the distribution of a substance diffusing with constant speed 
$c$ and initial concentration $u$, then we have the following analytic 
representation\footnote{We will see that causal diffusion is determined 
by the speed of diffusion $c$, a time period $\tau$ and the space dimension $N$. Here we 
assume $c=1$ and $\tau=1$.} 
\begin{equation}\label{vprop}
\begin{aligned}
   \hat v(\cdot,m+s) 
         =  (2\,\pi)^{-N/2} \,\Upsilon_N(|\cdot|)^m\, 
                            \Upsilon_N(|\cdot|\,s)\, \hat u  
\end{aligned}
\end{equation}
for $m\in\N_0$ and $s\in (0,1]$, where $\hat v$ denotes the Fourier transform of $v$ with 
respect to $\x\in\R^N$ and $\Upsilon_N$ is the solution of 
\begin{equation*}
\begin{aligned}
    \Upsilon_N''(t) + \frac{(N-1)}{t}\,\Upsilon_N'(t) + \Upsilon_N(t) = 0 
   \qquad\quad t>0\,, 
\end{aligned}
\end{equation*}
with initial conditions $\Upsilon_N(0+) = 1$ and $\Upsilon_N'(0+)$. 
For example, for $N=1,\,2,\,3$ we have 
\begin{equation}\label{Upsilon123}
\begin{aligned}
    \Upsilon_1(t) = \cos(t),\qquad   
    \Upsilon_2(t) = \mbox{J}_0(t)  \qquad\mbox{and}\qquad 
    \Upsilon_3(t) = \mbox{sinc}(t)\,,
\end{aligned}
\end{equation} 
where $\mbox{J}_0$ denotes the \emph{Bessel function} of first kind and order zero (cf. appendix).  
We note that from property~(\ref{vprop}) and $\supp(\check \Upsilon (\cdot,m+s)) = B_{m+s}(\vO)$ 
causality can be infered (cf.~(\ref{causaGcond2})). Here $\check \Upsilon (\cdot,s)$ denotes the 
inverse Fourier transform of $\Upsilon(|\cdot|\,s)$. 
Several important properties of the causal diffusion model are derived from these properties 
which are in strong contrast to the standard diffusion model (cf. Section~\ref{sec-cdiff} 
and~\ref{sec-comp}). 

The respective backwards diffusion problem corresponds to the solution of the Fredholm integral 
equation of the first kind
\begin{equation*}\label{invprob}
\begin{aligned}
   F_T(u) = w  \qquad\quad \mbox{for given data $w$,}
\end{aligned}
\end{equation*} 
where the forward operator is defined by $F_T(u):=v(\cdot,T)$ with $v$ as in~(\ref{vprop}) and 
$T>0$ denotes the data acquisition time. We show (for appropriate spaces) that the 
forward operator is injective and that it is compact (cf. Section~\ref{sec-invprobprop})
\begin{itemize}
\item [1)] if $N=2$ and $T>2\,\tau$ and  

\item [2)] if $N\geq 3$ and $T>\tau$. 

\end{itemize}
Here $N$ denotes the space dimension and $t$ the time. 
We note that the envelope of the Fourier transform of $F_T(u)$ does not 
decrease exponentially fast. In this sense the inverse problem is \emph{not} exponentially 
ill-posed. 
Furthermore, numerical simulations of the backwards diffusion problem are performed (cf. 
Section~\ref{sec-Sim}), which confirm our theoretical results. \\

The paper is organized as follows: In Section~\ref{sec-cdiff} we present our causal model 
of diffusion and derive those properties of diffusion that are needed for this paper. 
For the convenience of the reader we put the technical part that is relevant for 
Sections~\ref{sec-cdiff} and~\ref{sec-comp} in the appendix. 
Comparisons between the standard diffusion model and our causal diffusion model are performed 
in Section~\ref{sec-comp}. 
The theoretical and numerical aspects of the backwards diffusion problem are investigated in 
Sections~\ref{sec-invprobprop} and~\ref{sec-Sim}. Numerical simulations of the inverse problem 
via the Landweber method are presented at the end of Section~\ref{sec-Sim}.

\section{Causal diffusion and its properties}
\label{sec-cdiff}

We now define \emph{causal diffusion} for the case of a \emph{constant} speed 
$c\in (0,\infty)$. For the more general case we refer to~\cite{Kow11}.

\begin{defi}\label{defCD}
Let $c,\tau\in (0,\infty)$, $\d \sigma(\x')$ denote the Lebesgue surface measure on 
$\R^N$ and $|S_{R}(\vO)|$ denote the surface area of the sphere $S_R(\vO)$. 
Diffusion with a constant speed $c$ is defined by 
\begin{equation}\label{defcausdiff10}
\begin{aligned}
   v_{c,\tau}(\x,t) 
      = \int_{S_{R(t)}(\x)} \frac{ v_{c,\tau}\left(\x',\tau_{n(t)}\right)}{|S_{R(t)}(\vO)|} 
                    \,\d \sigma(\x')\,
           \quad\mbox{with}\quad v_{c,\tau}|_{t=0}=u\,,
\end{aligned}
\end{equation}
where $\tau_{n(t)}:=n(t)\,\tau$ and 
$$
  n(t)\in\N_0 \quad \mbox{such that } \quad t\in (n(t)\,\tau,(n(t)+1)\,\tau]\,,
$$
and $R(t):=c\,(t-n(t)\,\tau)$. If $u(\x)=\delta(\x)$, then we call $G_{c,\tau}:=v_{c,\tau}$ 
the Green function of diffusion.  Here $\delta(\x)$ denotes the 
\emph{delta distribution}\footnote{Our notation of the delta distribution is specified 
at the beginning of the Appendix.} on $\R^N$. 
\end{defi}

Let $c>0$, $\tau>0$ and $v_{c,\tau}(\x,t)$ be defined as in Definition~\ref{defCD}. 
It follows from induction (cf. Lemma~2 in~\cite{Kow11}) that the \emph{forward operator}  
\begin{equation}\label{FTv1}
\begin{aligned}
   F_T :L^1(\R^N)\to L^1(\R^N),\, u\mapsto v_{c,\tau}(\x,T)   
    \qquad\quad \mbox{($T>0$ fixed)}
\end{aligned}
\end{equation}
is well-defined and that $\|u\|_{L^1} = \|v_{c,\tau}\|_{L^1}$, 
i.e. causal diffusion satisfies the \emph{conservation law of mass}. 
In order to analyse the properties of the forward operator in 
Section~\ref{sec-invprobprop}, we derive the Fourier representation of the Green 
function $G_{c,\tau}$ of causal diffusion. 
In this paper $\hat f(\k)$ and $\F\{f\}(\k)$ denote the Fourier transform of 
$\x\in\R^N\mapsto f(\x)$. Our definition of the Fourier transform and the respective 
Convolution Theorem are formulated at the beginning of the Appendix.

\begin{rema}
The reader may object that the above definition of causal diffusion is not derived from 
first principles and that it  does not look very similar to standard diffusion.  
Because of property $\|u\|_{L^1} = \|v_{c,\tau}\|_{L^1}$, it follows that $v_{c,\tau}$ satisfies a 
continuity equation and, as shown in the introduction of~\cite{Kow11}, $v_{c,\tau}$ satisfies 
approximately Fick's law. Hence it is reasonable that the causal diffusion model follows from 
first principles under the side condition of causality. The derivation of our model 
from microscopic equations is intended to be carried out in the future.

Moreover, because standard diffusion satisfies a strongly continuous semigroup property 
with respect to time and, as shown in Theorem~\ref{theo:02} below, our causal diffusion model 
satisfies a discrete semigroup property with respect to time, it is evident that both models 
yield similar numerical results under appropriate conditions (compare Fig.~\ref{fig:ex01b} 
and Fig.~\ref{fig:ex01c}). 
\end{rema}

\begin{theo}\label{theo:01}
Let $G_{c,\tau}$ and $\Upsilon_N$ be defined as in Definition~\ref{defCD} 
and~(\ref{defUpsilon}) (cf. Appendix), respectively. Then 
\begin{equation}\label{defGcks}
\begin{aligned}
   \hat G_{c,\tau}(\k,s) 
                  = \frac{\Upsilon_N(|\k|\,c\,s)}{(2\,\pi)^{N/2}}
        \qquad  \mbox{for}  \qquad \k\in\R^N,\,s\in (0,\tau] 
\end{aligned}
\end{equation} 
and $\mu_s(A) :=  \int_A G_{c,\tau}(\x,s)\,\d\x$ defines a positive measure on the 
Borel sets. Moreover,~(\ref{Upsilon123}) and~(\ref{I}) (cf. Appendix) hold.  
\end{theo}

\begin{proof}
We note that $s\in (0,\tau]$ implies $n(s)=0$ and $R(s)=c\,\tau$. Moreover, we have 
\begin{equation*}
\begin{aligned}
  &\int_{S_R(\x)}  f(\x')\,\d \,\sigma(\x') \equiv \int_{S_1(\vO)} f(\x+R\,\y)\,R^{N-1}\, 
                  \d\, \sigma(\y) \,,\\
  &|S_R(\vO)|=|S_1(\vO)|\,R^{N-1}\,.
\end{aligned}
\end{equation*}
From these facts and Definition~\ref{defCD} with $u(\x)=\delta(\x)$, it follows 
that 
\begin{equation*}
\begin{aligned}
   G_{c,\tau}(\x,s)
      = \int_{S_{R(s)}(\x)} \frac{ \delta(\x')}{|S_{R(s)}(\vO)|} \,\d \sigma(\x')
      = \int_{S_1(\vO)} \frac{ \delta(\x+R(s)\,\y)}{|S_1(\vO)|} \,\d \sigma(\y)\,. 
\end{aligned}
\end{equation*}
$\mu_s(A)$ is a positive measure, since $G_{c,\tau}(\x,s)$ is a positive distribution. 

To determine the Fourier transform of $G_{c,\tau}$, we use the following series 
representation derived in Lemma~\ref{lemm:hatGc} (cf. Appendix): 
\begin{equation*}
\begin{aligned}
    (2\,\pi)^{N/2}\, \F\left\{
         \int_{S_1(\vO)} \frac{\delta(\cdot+c\,s\,\y)}{|S_1(\vO)|}\,\d \sigma(\y)
          \right\}(\k)
            = \sum_{j=0}^\infty (-1)^j\cdot a_{2\,j}\cdot (|\k|\,c\,s)^{2\,j} 
\end{aligned}
\end{equation*}
for $s\in (0,\tau]$ with $a_0=1$ and 
\begin{equation*}
\begin{aligned}
    a_{2\,j} = \frac{1}{(2\,j)!}\,\frac{1\cdot 3\cdot 5 \cdots (2\,j-1)}
                                      {N\cdot (N+2)\cdot(N+4) \cdots (N+2\,j-2)}
\qquad \mbox{for $j\in\N$\,.}
\end{aligned}
\end{equation*} 
In Theorem~\ref{theo:Upsilon2} it is shown that $\Upsilon_1(t)=\cos(t)$, 
$\Upsilon_2(t)=\mbox{J}_0(|\k|\,c\,s)$ and~(\ref{I}) holds. $\Upsilon_3(t)= \mbox{sinc}(t)$ 
can be concluded from Theorem~\ref{theo:Upsilon2}, too, or alternatively from  
\begin{equation*}
\begin{aligned}
   a_{2\,j}
        = \frac{1}{(2\,j)!}\,\frac{1\,\cdot 3\cdot 5 \cdots (2\,j-1)}
                                  {3\cdot 5\cdot 7\cdots (2\,j+1)} 
        = \frac{1}{(2\,j+1)!}\,.
\end{aligned}
\end{equation*}
This concludes the proof. 
\end{proof}

The following theorem together with Theorem~\ref{theo:01} provides us with a 
complete description of causal diffusion and a mean to compare causal and 
standard diffusion in the $\k-t-$space (cf. Subsection~\ref{subsec-Comp}).

\begin{theo}\label{theo:02}
Let $v_{c,\tau}$, $u$ and $G_{c,\tau}$ be defined as in Definition~\ref{defCD}. 
Moreover, let $S_\tau$ denote the space convolution operator with kernel $G(\x,\tau)$, 
i.e. $S_\tau\,u:=G(\cdot,\tau)*_\x u$ for every $u\in L^1(\R^n)$. Then we have 
\begin{equation*}
\begin{aligned}
   v_{c,\tau}(\cdot,t) = S_\tau^m \,S_s\,u  
   \qquad \mbox{for}\qquad t = \tau_m+s,\,s\in (0,\tau]\,, 
\end{aligned}
\end{equation*}
which is equivalent to 
\begin{equation}\label{vUpsilon}
\begin{aligned}
   \hat v_{c,\tau}(\cdot,t) 
         =  (2\,\pi)^{-N/2} \,\Upsilon_N(|\cdot|\,c\,\tau)^m\, 
                            \Upsilon_N(|\cdot|\,c\,s)\, \hat u  \,.
\end{aligned}
\end{equation}
\end{theo}

\begin{proof}
The claim follows by induction. 
Let $m=0$. Then $t = s \in(0,\tau]$ and from Theorem~\ref{theo:01}, we get 
\begin{equation*}
\begin{aligned}
   (G_{c,\tau}(\cdot,s) *_\x u)(\x) 
        &= \int_{\R^N} \int_{S_1(\vO)} \frac{\delta(\x'+c\,s\,\y)}{|S_1(\vO)|}\,u(\x-\x') 
                            \,\d\x' \,\d \sigma(\y) \\
        &= \int_{S_1(\vO)} \frac{u(\x+c\,s\,\y)}{|S_1(\vO)|} \,\d \sigma(\y) 
         = v(\x,t)\,. 
\end{aligned}
\end{equation*}
Now we assume the induction assumption 
$$
   v(\cdot,r + \tau_{m-1}) = S_\tau^{m-1} \,S_r\,u 
          \qquad \mbox{for}\qquad  
   r\in(0,\tau]\,.
$$
Let $t = s + \tau_m$ with $s\in(0,\tau]$. From the induction assumption 
with $r:=\tau$, Theorem~\ref{theo:01} and Definition~\ref{defCD}, we infer 
\begin{equation*}
\begin{aligned}
   & S_\tau^m \,S_s\,u(\x) 
         = S_s \,S_\tau^m\,u(\x) 
         = S_s\, v(\x,\tau_m) \\
         & \qquad = \int_{\R^N} \int_{S_1(\vO)} \frac{\delta(\x'+c\,s\,\y)}{|S_1(\vO)|}\, 
                                   v(\x-\x',\tau_m) 
                       \,\d\x' \,\d \sigma(\y) \\
         & \qquad = \int_{S_1(\vO)} \frac{v(\x+c\,s\,\y,\tau_m)}{|S_1(\vO)|} 
                       \,\d \sigma(\y) 
         =  v(\x,s+\tau_m )\,,
\end{aligned}
\end{equation*}
since $n(t)=m$. 

Finally, the representation~(\ref{vUpsilon}) follows from Theorem~\ref{theo:01} and 
the Convolution Theorem~(\ref{convth}) (cf. Appendix). 
This concludes the proof. 
\end{proof}

\begin{rema}\label{rema:discSG}
According to Theorem~\ref{theo:02} the family of operators $\{S_{\tau_m}\,|\,m\in\N_0\}$ 
is a \emph{discrete semigroup}, i.e. $S_0$ is the identity and 
$$
     S_{\tau_m+\tau_n} = S_{\tau_m}\,S_{\tau_n}  \qquad\mbox{for}\qquad m,\,n\in\N_0\,.
$$
Formally, the limit $\tau\to 0$ leads to a continuous semigroup 
$\{S_t\,|\,t\geq 0\}$. Indeed, in Subsection~\ref{subsec-SAIOC} we show for the case 
$N=2$ that the limit $\tau\to 0$ under the side condition 
$\frac{c^2\,\tau}{2\,N}=const.$ yields the standard diffusion model and thus 
$\lim_{\tau\to 0}\{S_{\tau_m}\,|\,m\in\N_0\}$ is a strongly continuous semigroup 
on $[0,\infty)$. 
A consequence of this limit process is that $c\to \infty$, i.e. this limit diffusion 
process is not causal. Because of these facts, we denote the Green function of 
\emph{standard diffusion} by $G_{\infty,0}$. 
\end{rema}

\begin{coro}\label{coro} 
The diffusion model defined by Definition~\ref{defCD} satisfies causality condition 
\begin{equation}\label{causaGcond2}
\begin{aligned}
   \supp (G_{c,\tau}) \subseteq \{(\x,t)\in\R^N\times [0,\infty)\,|\,|\x|\leq c\,t\} \,.
\end{aligned}
\end{equation} 
\end{coro}

\begin{proof}
We recall that the space convolution of two distributions with compact support $A$ and 
$B$ is well-defined and its support lies in $A+B:=\cup_{\x\in A}\{\x\}+B$.  
Moreover, we note that 
$$
B_{R_1}(\vO)+B_{R_2}(\vO) = B_{R_1+R_2}(\vO)\,,
$$  
where $B_R(\vO)$ denotes the open ball of radius $R$ and center $\vO$. 

Let $t=s+\tau_m$ with $m\in\N_0$ and $s\in (0,\tau]$. From Definition~\ref{defCD}, it 
follows that $\supp (G_{c,\tau}(\x,r))\subseteq B_{c\,r}(\vO)$ for $r\in (0,\tau]$ and 
thus by Theorem~\ref{theo:02} the support of $G_{c,\tau}(\x,t)$ lies in 
$B_{c\,s+c\,m\,\tau}(\vO) = B_{c\,t}(\vO)$, which proves the claim. 
\end{proof}

The following corollary provides us with an alternative definition of causal diffusion 
and it can be used to compare causal and standard diffusion in the $\x-t-$space 
(cf. Subsection~\ref{subsec-SAIOC}).

\begin{coro}\label{coro:waveeq}
Let $\tau_m$ for $m\in\N_0$ and $v_{c,\tau}$ be defined as in Definition~\ref{defCD} with 
$u\in L^1(\R^N)$. The function 
$$
      v(\cdot,s) := v_{c,\tau}(\cdot,\tau_m+s)  \qquad \mbox{for} \qquad s\in (0,\tau]
$$
solves the wave equation 
\begin{equation}\label{seqofeq1}
\begin{aligned}
   &\frac{\partial^2 v}{\partial s^2} 
         + \frac{(N-1)}{s}\,\frac{\partial v}{\partial s} 
    - c^2\,\nabla^2 v = 0  \qquad \mbox{on} \qquad (0,\tau]
\end{aligned}
\end{equation}
with initial conditions\footnote{It can be shown that 
$\frac{\partial v_{c,\tau}}{\partial t}(\cdot,\tau_m-) \not = 0 = 
\frac{\partial v_{c,\tau}}{\partial t}(\cdot,\tau_m+)$, i.e. 
$v_{c,\tau}$ is not continuous at the set of time instants $\{\tau_m\,|\,m\in\N\}$ 
where the semigroup property holds. This is in strong contrast to standard diffusion.}
\begin{equation}\label{initcond}
\begin{aligned}
   &v(\cdot,0) = v_{c,\tau}(\cdot,\tau_m)\,\qquad\mbox{and}\qquad
       \frac{\partial v}{\partial s}(\cdot,0+) = 0  \,.
\end{aligned}
\end{equation}
Here $v_{c,\tau}(\cdot,0)$ ($m=0$) is understood as the initial distribution 
$u\in L^1(\R^N)$. 
\end{coro}

\begin{proof}
From Corollary~\ref{coro:Upsilon} (cf. Appendix), it follows that 
\begin{equation*}
\begin{aligned}
   &\frac{\partial^2 \Upsilon(|\k|\,c\,s)}{\partial (|\k|\,c\,s)^2} 
    + \frac{(N-1)}{|\k|\,c\,s}\,\frac{\partial \Upsilon(|\k|\,c\,s)}{\partial (|\k|\,c\,s)} 
    + c^2\,\Upsilon(|\k|\,c\,s) = 0  \,,
\end{aligned}
\end{equation*}
where $\k$ and $c$ are fixed. From this, and 
$$
   \hat v(\k,s) 
         =  \Upsilon_N(|\k|\,c\,\tau)^m\, \Upsilon_N(|\k|\,c\,s)\, \hat u(\k)\,, 
$$
we obtain 
\begin{equation*}
\begin{aligned}
   &\frac{\partial^2 \hat v}{\partial s^2} 
         + \frac{(N-1)}{s}\,\frac{\partial \hat v}{\partial s} 
    - c^2\,\i^2\,|\k|^2\,\hat v = 0 \,.
\end{aligned}
\end{equation*}
But this is equivalent to equation~(\ref{seqofeq1}). 
The first initial condition follows from $\Upsilon_N(0+)=1$ (cf. 
Corollary~\ref{coro:Upsilon}) and 
$\hat v_{c,\tau}(\k,\tau_m+) =\hat u(\k) \,\Upsilon_N(|\k|\,c\,\tau)^m$:
\begin{equation*}
\begin{aligned}
       \hat v(\k,0+) 
        = \hat u(\k) \,\Upsilon_N(|\k|\,c\,\tau)^m\,  
               \Upsilon_N (|\k|\,c\,0+)
        = \hat v_{c,\tau}(\k,\tau_m+) \,. 
\end{aligned}
\end{equation*}
Finally, the second initial condition follows from  $\Upsilon'(0+)=0$ (cf. 
Corollary~\ref{coro:Upsilon}): 
\begin{equation*}
\begin{aligned}
       \frac{\partial \hat v}{\partial s}(\k,0+) 
        = \hat u(\k) \,\Upsilon_N(|\k|\,c\,\tau)^m\,  
               \frac{\partial \Upsilon_N}{\partial s} (|\k|\,c\,s)|_{s=0+}
        = 0  \,.
\end{aligned}
\end{equation*}
This concludes the proof. 
\end{proof}

\section{Standard diffusion versus causal diffusion}
\label{sec-comp}

In the following we compare standard and causal diffusion in the $\k-t-$space and 
the $\x-t-$space. As explained in Remark~\ref{rema:discSG}, we denote the Green 
function of standard diffusion by $G_{\infty,0}$, since $c=\infty$ and $\tau=0$. 
Similarly we use the notation 
$$
     v_{\infty,0} := G_{\infty,0} *_\x u  \qquad\mbox{for}\qquad u\in L^1(\R^N)\,.
$$

\subsection{Comparison in the $\k-t-$space}
\label{subsec-Comp}

First we recall the definition of the Green function $G_{\infty,0}$ of standard  
diffusion and establish the \emph{link relation}
\begin{equation}\label{defD0}
   D_0 = \frac{c^2\,\tau}{2\,N}   \, 
\end{equation}
between the diffusivity $D_0$ of standard diffusion and the 
parameters $c$ and $\tau$ of causal diffusion. Then we compare the Green function of 
both processes 
in the $\k-t-$space. 

It is well-known that the Green function of standard diffusion reads as follows 
(cf. e.g.~\cite{Harris79,Fet80,Wei98a,Eva99,Kow11})
\begin{equation*}%\label{Ginftyxt}
\begin{aligned}
   G_{\infty,0} (\x,t) := (4\,\pi\,D_0\,t)^{-N/2} \, 
                 \exp{\left(-\frac{|\x|^2}{4\,D_0\,t}\right)} 
  \qquad (\x\in\R^N,\,t>0)\,
\end{aligned}
\end{equation*}
and that its Fourier transform with respect to $\x$ is given by  
\begin{equation}\label{Ginftykt}
\begin{aligned}
   \hat G_{\infty,0}(\k,t) = (2\,\pi)^{-N/2} \, \exp{\left(-D_0\,|\k|^2\,t\right)}  
  \qquad (\k\in\R^N,\,t>0)\,.
\end{aligned}
\end{equation}
For sufficiently small $|\k|$ (and $t=\tau$) we have
$$
    \hat G_{\infty,0}(\k,\tau) 
         \approx (2\,\pi)^{-N/2}\,\left(1 - D_0\,|\k|^2\,\tau\right)\,.
$$
For causal diffusion we get a similar approximation from Theorem~\ref{theo:01} 
together with Lemma~\ref{lemm:hatGc} (cf. Appendix), namely 
\begin{equation*}
\begin{aligned}
   \hat G_{c,\tau}(\k,\tau)
       \approx (2\,\pi)^{-N/2}\,\left(1 - \frac{|\k|^2\,c^2\,\tau^2}{2\,N}\right)\,.
\end{aligned}
\end{equation*} 
Comparison of these first order approximations yields the link relation~(\ref{defD0}).

\begin{rema} 
Because of 
$$
    2^a = e^{a\,\log 2}  \qquad\mbox{for}\qquad a\in\R\,,
$$
the function 
\begin{equation}\label{Gsq}
\begin{aligned}
   G_\#(\x,t) := (4\,\pi\,D_0\,t)^{-N/2} \, 
            2^{\left(-\frac{|\x|^2}{4\,D_0\,t}\right)} 
  \qquad (\x\in\R^N,\,t>0)\,.
\end{aligned}
\end{equation}
satisfies the standard diffusion equation with diffusion constant $D_\#:=D_0/\log(2)$, 
i.e.  
\begin{equation}\label{hatGsq}
\begin{aligned}
   \hat G_\#(\k,t) = (2\,\pi)^{-N/2} \, \exp{\left(-D_\#\,|\k|^2\,t\right)}  \,.
\end{aligned}
\end{equation}
We use this (perturbed) diffusion model as a third reference model.  
\end{rema}

For the rest of this subsection we focus on the case $N=3$ for which we have
(cf. Theorems~\ref{theo:01} and~\ref{theo:02}):
\begin{equation}\label{hatGc3d}
\begin{aligned}
   \hat G_{c,\tau}(\k,\tau_m+s) 
      &= (2\,\pi)^{-3/2}\,\mbox{sinc}^m (|\k|\,c\,\tau)\,\mbox{sinc} (|\k|\,c\,s)\,
\qquad s\in (0,\tau]\,.
\end{aligned}
\end{equation}
We see that the function 
$\k\mapsto \hat G_{c,\tau}(\k,t)$ ($t>0$) is not $C^\infty$,  
since the necessary condition 
\begin{equation*}
\begin{aligned}
   \exists C>0 \,\forall m\in\N\,\forall \k\in\R^N:\quad    | \hat G_{c,\tau}(\k,t)| 
      \leq C\,(1+|\k|)^{-m}   
\end{aligned}
\end{equation*}
does not hold (cf. Paley-Wiener-Schwartz Theorem in~\cite{Hoe03}). 
Moreover, it is easy to see from~(\ref{hatGc3d}) that 
\begin{itemize}
\item [1)]  $\k\mapsto\hat G_{c,\tau}(\k,t)$ has a discrete and infinite set of zeros, 

\item [2)]  $\k\mapsto\hat G_{c,\tau}(\k,t_1)$ and $\k\mapsto\hat G_{c,\tau}(\k,t_2)$ 
             have the same zeros if and only if $(t_1-t_2)/\tau\in\N_0$, 

\item [3)]  some of the zeros move with speed $c$ during the time intervals 
            $(\tau_{m-1},\tau_m]$ ($m\in\N$) such that at the time instants 
            $t=\tau_{m-1}$ and $t=\tau_m$ the same set $Z$ of zeros occur. 
            The set $Z$ is given by $\{\k\in\R^3\,|\,\hat G_{c,\tau}(\k,\tau)=0\}$.

\end{itemize} 
This behaviour is in contrast to standard diffusion, since 
$\k\mapsto\hat G_{\infty,0}(\k,t)$ is $C^\infty$ and has no zeros at all. However, 
since $G_{c,\tau}(\cdot,t)$ has compact support for each $t>0$, the Paley-Wiener-Schwartz 
Theorem implies a behaviour of such type for causal diffusion. 

%We conclude this subsection with a numerical example. 
Now we perform a numerical comparison. 

\begin{exam}
Let $N=3$, $c=1$, $\tau=1$ and $D_0$ be defined as in~(\ref{defD0}). 
For these parameters Fig.~\ref{fig:image02} shows a numerical comparison 
of the Green function of causal diffusion~(\ref{hatGc3d}) with the 
Green functions~(\ref{Ginftykt}) and~(\ref{hatGsq}) of standard diffusion. 
This and further numerical experiments indicate that 
\begin{itemize}
\item [i)]   $\k\mapsto\hat G_\#(\k,t)$ is closer to $\k\mapsto\hat G_{\infty,0}(\k,t)$  
             than $\k\mapsto\hat G_{c,\tau}(\k,t)$, 

\item [ii)]  if $t$ is large, then $\hat G_\#(\k,t)$, $\hat G_{\infty,0}(\k,t)$ and 
             $\hat G_{c,\tau}(\k,t)$ are very small for large $|\k|$,  

\item [iii)] $\k\mapsto\hat G_{c,\tau}(\k,t)$ has a discrete set of zeros, in particular it 
             is not monotone. If $n$ is even, then $\k\mapsto\hat G_{c,\tau}(\k,t)$ oscillates 
             around zero.

\end{itemize}
This behavior indicates that modeling errors are a serious issue for the backwards 
diffusion problem.  
\end{exam}

\subsection{Comparison in the $\x-t-$space}
\label{subsec-SAIOC}

In the following we demonstrate that for an appropiate parameter set a discretization 
of the standard diffusion equation can yield a similar results as the causal 
diffusion model introduced in Definition~\ref{defCD}. 

In order to keep the following formulas and equations short, 
we focus on the two dimensional case. Consider the diffusion 
of an image with size of pixel $(\Delta x)^2$  and size of time step 
$\Delta t:=\tau$. We use the notion 
$$
    v_{i,j}^m:=v(i\,\Delta x, j\,\Delta x, \tau_m) 
  \qquad\quad \mbox{for $i,j\in\Z$ and $m\in\N_0$.}
$$  
If the length of an image pixel $\Delta x$ satisfies (cf. Definition~\ref{defCD})
$$
         R(\tau) = c\,\tau \equiv \Delta x\,,
$$
then we can use the (rough) approximation 
%(cf. Fig.~\ref{fig:image01}) 
$$
  \int_{|\x-\y|=R(\tau)} \frac{f(\x')}{|S_1(\vO)|} \,d\sigma(\y) 
      \approx (f_{i+1,j} +f_{i-1,j} +f_{i,j+1} +f_{i,j-1})/4\,. 
$$
With this discretization the causal diffusion model~(\ref{defcausdiff10}) is equivalent to 
\begin{equation}\label{feuler}
\begin{aligned}
  \frac{{\blue v_{i,j}^{n+1}}-v_{i,j}^n}{\tau}
    = \frac{\Delta x^2}{4\,\tau}\,
      \left[\frac{{\blue v_{i+1,j}^n} - 2\,v_{i,j}^n   + {\blue v_{i-1,j}^n}}{\Delta x^2} 
     + \frac{{\blue v_{i,j+1}^n} - 2\,v_{i,j}^n   + {\blue v_{i,j-1}^n}}{\Delta x^2} \right]\,,
\end{aligned}
\end{equation}
which is the \emph{Forward Euler method of the classical diffusion equation}. 
The classical diffusion equation can be obtained for $\tau\to 0$ under the side 
condition $\Delta x^2/(4\,\tau)=const$, i.e. the diffusivity corresponds to 
$D_0=\Delta x^2/(2\,N\,\tau)$ with $N=2$. 
Carring out this limit process yields $c=\infty$, i.e. the diffusion speed can be 
interpreted as infinite. In particular, this shows that the discrete semigroup 
$\{S_{\tau_m}\,|\,m\in\N_0\}$ (cf. Theorem~\ref{theo:02}) converges to 
the strongly continuous semigroup of the standard diffusion equation.

\begin{rema}
Similarly, the discretization of the wave equation~(\ref{seqofeq1}) for $N=2$ 
yields the \emph{Forward Euler method of the classical diffusion equation}, since 
\begin{equation*}
\begin{aligned}
    \partial_t^2 v + \frac{c}{\tau}\, \partial_t v 
     \equiv      \frac{v_{i,j}^{n+1}-\not 2\,v_{i,j}^n + \not{v_{i,j}^{n-1}}}{\tau^2} 
                      + \frac{\not v_{i,j}^{n} - \not v_{i,j}^{n-1}}{\tau^2}
     \equiv \frac{\partial_t v}{\tau}\,. 
\end{aligned}
\end{equation*} 
Here the CFL condition is satisfied for the discretization $\Delta t:=\tau$ and 
$\Delta x:= c\,\tau$ (cf.~\cite{CouFriLev67}). However, to obtain the Forward 
Euler method we have to neglected the second condition in~(\ref{initcond}), i.e. 
\begin{equation*}
\begin{aligned}
       \frac{\partial v}{\partial t}(\cdot,\tau_m+) = 0  
\qquad\mbox{for all}\qquad m\in\N \,.
\end{aligned}
\end{equation*} 
\end{rema}

\vspace{0.5cm}
The following numerical example indicates that for \emph{sufficiently large time $t$} the 
forward Euler method (with fine space discretization) can be considered as a noncausal 
approximation of the causal diffusion model. 

\begin{exam}\label{exam:A}
Let N = 2, $c = 6.3\cdot 10^{-3}\,m/s$, $R:=10^{-3}\,m$ and $\tau:=R/c$. Here we use 
the notation $R:=R(\tau)$. To the parameters $c$ and $\tau$ (with $N=2$) of causal 
diffusion corresponds the diffusion constant 
$D_0 := c^2\,\tau/4=1.575\cdot 10^{-6}\,m^2/s$ 
of standard diffusion. The initial mass distribution is shown in Fig.~\ref{fig:ex01}. 
To calculate $v_{c,\tau}$ defined as in Definition~\ref{defCD}, the circles with radius 
$R$ were discretized by $65$ points (cf. Fig.~\ref{fig:ex01}). 
The noncausal distribution $v_{\infty,0}$ was calculated via the forward Euler method 
for the standard diffusion equation. To guaranteed the convergence of this scheme, 
the discretization was chosen as $\Delta x := R/10$ and 
$\Delta t := \frac{\Delta x^2}{2\,N\,D_0}$ such that 
$$
    \frac{\Delta x^2}{2\,N\,\Delta t} \geq D_0 \,
$$
holds. 
A time sequence of $v_{c,\tau}(\cdot,t)$ and $v_{\infty,0}(\cdot,t)$ for the time instants 
$t=\frac{\tau}{3},\,\frac{2\,\tau}{3},\,\tau,\,\ldots,\,2\,\tau$ is visualized in 
Fig.~\ref{fig:ex01b} and Fig.~\ref{fig:ex01c}, respectively. 
As expected each distribution $v_{\infty,0}(\cdot,t)$ is very smooth, in contrast to the 
distribution $v_{c,\tau}(\cdot,t)$ of causal diffusion. No edge or corner appears in 
the case of standard diffusion. 
Although it is not visible, in contrast to $v_{c,\tau}(\cdot,t)$, the support of 
$v_{\infty,0}(\cdot,t)$ does not lie within the image. 
For this example, $v_{\infty,0}(\cdot,t)$ and $v_{c,\tau}(\cdot,t)$ are very similar 
after a time period of about $t=3\,\tau$. Again, we note that this means that 
modeling errors for the backward diffusion problem is an issue. 
\end{exam}

\section{Basic properties of the forward operator}
\label{sec-invprobprop}

The calculation of a diffusing substance over the time period $T$ with initial 
concentration $u\in L^1(\R^N)$ corresponds to the evaluation of the forward 
operator~(\ref{FTv1}). 
We define this as the direct problem and consider the estimation of the initial 
concentration $u$ from appropriate data $w$. That is to say the solution of the 
Fredholm integral equation of the first kind 
\begin{equation}\label{invprob}
\begin{aligned}
   F_T(u) = w  \qquad\quad \mbox{for given data $w$.}
\end{aligned}
\end{equation} 
This inverse problem requires the knowledge of $c$, $\tau$ and $T$. 
In this section we investigate the properties of the forward operator and in the 
subsequent section we discuss and perform numerical simulations of the inverse problem. 
We use the notation: 

\begin{defi}
a) Let $T>0$ and $\Omega_0$ be an open subset of $B_r(\vO)$ ($r>0$). Then we define 
$\Omega_T := B_{r+c\,T}(\vO)$. Here $B_r(\vO)$ denotes the open ball with center 
$\vO$ and radius $r$. \\
b) $L_c^2(\R^N)$ is defined as the space of $L^2-$functions with compact support 
in $\R^N$.
\end{defi}

\begin{theo}\label{theo:inj}
Let $T>0$ and $G_{c,\tau}$ be as in Definition~\ref{defCD}. 
The sets of zeros of $\hat G_{c,\tau}(\cdot,T)$ is discrete and countably infinite, and 
the operator $F_T:L_c^1(\R^N)\to L_c^1(\R^N)$ is injective. 
\end{theo}

\begin{proof}
a) For $t=\tau_n+s$ with $n\in\N_0$ and $s\in (0,\tau]$, we have 
$\hat G_{c,\tau}(\k,\tau_n+s)= \hat G_{c,\tau}(\k,\tau)^n\,\hat G_{c,\tau}(\cdot,s)$. 
Hence  it is sufficient to show that the sets of zeros of $\hat G_{c,\tau}(\cdot,s)$ 
is discrete and countably infinite. Assume that the function $f:\k\mapsto \hat G_{c,\tau}(\cdot,s)$ 
vanishes on a non-empty $M$ with an accumulation point $\k_0\in M$. 
Because $f$ has compact support, it can be extended to an analytic function 
$f_{ext}:\C^N\to\C^N$ such that $f_{ext}(\k)=f(\k)$ for $\k\in \R^N$ 
(Paley-Wiener Theorem). Since $f_{ext}(\k)=0$ for $\k\in M$ and $\k_0\in M$ is an 
accumulation point, $f_{ext}$ is the zero function. Thus $M$ must be discrete. (This 
fact can also be concluded from Theorem~\ref{theo:Upsilon2}.) That the set of zeros of 
$\hat G_{c,\tau}(\cdot,s)$ is countably infinite follows from Theorem~\ref{theo:Upsilon2} 
and the fact that $\Upsilon_1(s)=\cos(s)$ and $\Upsilon_2(s)=\mbox{J}_0(s)$ have  
countably infinite zeros.  \\
b) For the injectivity of $F_T$. Since $u$ and $G_{c,\tau}(\cdot,T)$ have compact 
support, $\hat u$ and $\hat G_{c,\tau}(\cdot,T)$ exist and the Convolution Theorem 
holds (cf. Theorem~7.1.15 in~\cite{Hoe03}). Hence 
$$
     \F\{F_T(u)\} = \hat G_{c,\tau}(\k,T)\,\hat u \,
$$ 
which implies that
$$
   F_T(u) = 0 \qquad\Rightarrow\qquad   u = 0
$$
is equivalent to
$$
   \hat G_{c,\tau}(\cdot,T)\,\hat u = 0 \qquad\Rightarrow\qquad   \hat u = 0\,.
$$ 
From $\hat G_{c,\tau}(\cdot,T)\,\hat u = 0$ and part a) of the proof we infer that 
$\hat u$ vanishes on a non-empty open set $M$. 
Because $u$ has compact support, the Paley-Wiener Theorem implies that $\hat u$ 
can be extended to an analytic function $\hat u_{ext}$ on $\C^N$ satisfying 
$\hat u_{ext}(\k)=0$ for $\k\in M$. Therefore $\hat u_{ext}$ is the zero function 
and consequently $u$ vanishes. This proves that $F_T$ is injective. 
\end{proof}

\begin{theo}\label{theo:FTsa}
The operator $F_T:L_c^2(\R^N)\to L_c^2(\R^N)$ is positive, linear and self-adjoint. 
\end{theo}

\begin{proof}
First we show that $F_T:L_c^2(\R^N)\to L_c^2(\R^N)$ is well-defined. Because $L_c^2(\R^N)$ 
is a subspace of a Hilbert space, it is a Hilbert space, too.   
If $u\in L_c^2(\R^N)$, then $u\in L_c^1(\R^N)$ and thus $F_T(u)\in L^1(\R^N)$. 
Since $G_{c,\tau}$ and $u$ have compact support, their convolution exist and it 
has compact support (cf. Theorem~7.1.15 in~\cite{Hoe03}). Hence 
we obtain $F_T(L_c^2(\R^N))\subseteq L_c^1(\R^N)$. 
According to Parseval's formula and 
$$
    |(2\,\pi)^{N/2}\,\hat G(\k,t)|\leq C   \qquad \mbox{for some constant $C$}\,, 
$$ 
which follows from Theorem~\ref{theo:Upsilon2} (cf. Appendix) and Theorem~\ref{theo:01}, we have 
\begin{equation*}
\begin{aligned}
    \|F_T(u)\|_{L^2}^2 
         &=  (2\,\pi)^{N}\|\F\{F_T(u)\}\|_{L^2}^2 
         =  (2\,\pi)^{N}\,C\,\int_{\R^N} |\hat G(\k,t)\,\hat u(\k)|^2 \,\d\k\\
         &\leq C\,\|u\|_{L^2}^2\, <\infty \,,
\end{aligned}
\end{equation*}
i.e. $F_T(u)\in L_c^2(\R^N)$. Hence the operator is well-defined. 

The positivity and linearity of the operator $F_T$ follows at once from 
Definition~\ref{defCD} and Theorem~\ref{theo:02}, respectively. 

Since $G_{c,\tau}(\x-\x',T)=G_{c,\tau}(\x'-\x,T)$, it follows that
\begin{equation*}
\begin{aligned}
    \langle F_T(u),w\rangle_{L^2} 
          = \int_{\R^N} \int_{\R^N} G_{c,\tau}(\x-\x',T)\,u(\x')\,w(\x)\,\d\,\x'\,\d\,\x
          = \langle u, F_T(w)\rangle_{L^2} \, 
\end{aligned}
\end{equation*}
for $w\in L_c^2(\R^N)$ and thus $F_T$ is self-adjoint. 
This concludes the proof. 
\end{proof}

\begin{theo}
If $N=1$ and $T>0$, then $G_{c,\tau}(\cdot,T)$ is a discrete and positive 
measure and the operator $F_T:L^2(\Omega_0)\to L^2(\Omega_T)$ is not compact. 
\end{theo}

\begin{proof}
Without loss of generality we set $c=1$. According to Theorem~\ref{theo:01} we have 
for $s\in (0,\tau]$: 
$$
     G_{c,\tau} (x,s) 
        = \F^{-1}\{\cos(k\,s)\}(x) 
        = \frac{1}{2}\,[\delta(x-s) + \delta(x+s)]\,,
$$
which implies that $G_{c,\tau} (x,T)$ is a convolution of positive distributions with 
singular support. Therefore $G_{c,\tau}(\cdot,T)$ corresponds to a discrete and 
positive measure. 
That $F_T$ is not compact follows from the fact that
$$
   F_T = (R_\tau+L_\tau)^m\,(R_s+L_s)  \qquad \mbox{for} \qquad T=\tau_m+s\,,
$$
where  $R_s,\,L_s:L^2(\R)\to L^2(\R)$ are noncompact operators defined by 
$$
    R_s(u):=u(\cdot-s)  \qquad \mbox{and} \qquad 
    L_s(u):=u(\cdot+s)\,.
$$
\end{proof}

\begin{theo}
If $N=2$ and $T>2\,\tau$, then $G_{c,\tau}(\cdot,T)\in L_c^2(\R^2)$ and the 
operator $F_T:L^2(\Omega_0)\to L^2(\Omega_T)$ is compact. 
\end{theo}

\begin{proof}
Without loss of generality we set $c=1$. Let $T=\tau_m + s$ with $m\geq 2$ and $s\in (0,\tau]$. 
From Theorem~\ref{theo:01} together with $|\mbox{J}_0(r)|\leq 1$ and the asymptotic 
behaviour~(\ref{asymJ0}) of $\mbox{J}_0$ (cf. appendix), 
we get for $N=2$:
\begin{equation*}
\begin{aligned}
    \|\hat G_{c,\tau}(\cdot,T)\|_{L^2(\R^2)}^2 
       &= 2\,\pi\,\int_0^\infty \mbox{J}_0^{2\,m}(r\,\tau)\,
              \mbox{J}_0^2(r\,s)\,r  \,  \d\,r \\
       &\leq   A
             + \frac{2^3}{\pi\,\tau\,s}\,
                \int_M^\infty \frac{1}{r^2} \,  \d\,r \,,
\end{aligned}
\end{equation*}
where 
$$
   A := 2\,\pi\,\int_0^{M} \mbox{J}_0^{2\,m}(r\,\tau)\,
              \mbox{J}_0^2(r\,s)\,r \,  \d\,r <\infty\,
$$
and $M>0$. 
Because of $\int_M^\infty 1/r^2\,\d\,r = 1/M$, we arrive at 
\begin{equation*}
\begin{aligned}
    \|\hat G_{c,\tau}(\cdot,T)\|_{L^2(\R^2)}^2 
        &=   A
             + \frac{2^3}{\pi\,\tau\,s}\,\frac{1}{M} 
         <\infty\,,
\end{aligned}
\end{equation*}
i.e. $\hat G_{c,\tau}(\cdot,T)$ lies in $L^2(\R^2)$. Consequently, 
$G_{c,\tau}(\cdot,T)$ lies in $L_c^2(\R^2)$. 
The compactness of the operator $F_T$ (for $N=2$ and $T>2\,\tau$) follows 
from Theorem 8.15 in~\cite{Alt00}.
\end{proof}

\begin{theo}
If $N\geq 3$ and $T>\tau$, then $G_{c,\tau}(\cdot,T)\in L_c^2(\R^3)$ and the 
operator $F_T:L^2(\Omega_0)\to L^2(\Omega_T)$ is compact.
\end{theo}

\begin{proof}
Without loss of generality we set $c=1$. Let $T=\tau_m + s$ with $m\in\N$ and $s\in (0,\tau]$. 
From Theorem~\ref{theo:01} and the estimation~(\ref{estUpsilon}) in Theorem~\ref{theo:Upsilon2}, 
it follows for $N\geq3$: 
\begin{equation*}
\begin{aligned}
    \|\hat G_{c,\tau}(\cdot,T)\|_{L^2(\R^2)}^2 
       &= |S_1(\vO)|\,\int_0^\infty \Upsilon_N^{2\,m}(r\,s)\,\Upsilon_N^2(r\,s)
                                      \,r^{N-1}  \,  \d\,r \\
       &\leq |S_1(\vO)|\,\left( A 
                  + C_N^{2\,(m+1)}\,\int_M^\infty \frac{r^{N-1}}{(r\,s)^{(m+1)\,(N-1)}}
                                      \,  \d\,r \right)\,,
\end{aligned}
\end{equation*}
where 
$$
   A :=  |S_1(\vO)|\,\int_0^{M}  \Upsilon_N^{2\,m}(r\,s)\,\Upsilon_N^2(r\,s)
                                      \,r^{N-1} \d\,r <\infty\,
$$
and $M>0$ is sufficiently large. 
Because of $\int_M^\infty r^{-(N-1)}\,\d\,r = -\frac{1}{(N-2)}\,M^{-(N-2)}$, we end up with 
\begin{equation*}
\begin{aligned}
    \|\hat G_{c,\tau}(\cdot,T)\|_{L^2(\R^2)}^2 
        &=    |S_1(\vO)|\,\left( A
                +\frac{C_N^{2\,(m+1)}}{s^{(m+1)\,(N-1)}\,(N-2)\,M^{N-2}} 
                \right)<\infty\,,
\end{aligned}
\end{equation*}
i.e. $\hat G_{c,\tau}(\cdot,T)$ lies in $L^2(\R^2)$. As a consequence, 
$G_{c,\tau}(\cdot,T)$ lies in $L_c^2(\R^2)$. 
The compactness of the operator $F_T$ (for $N=2$ and $T>2\,\tau$) follows 
from Theorem 8.15 in~\cite{Alt00}. 
\end{proof}

From the Paley-Wiener-Schwartz Theorem (cf.~\cite{Hoe03}), it follows that the Moore-Penrose 
inverse $F^\dagger$ is uniquely defined by 
$$
F_T^\dagger := (F_\tau^\dagger)^m \, F_s^\dagger   \qquad\qquad 
 (\mbox{$T=\tau_m+s$, $s\in (0,\tau]$})
$$
with 
\begin{equation}\label{relFdagger}
\begin{aligned}
  \F\{F_s^\dagger(w)\}(\k)
     := \,\frac{\hat w(\k) }{\hat G(\k,s)}\, \chi_{\Omega_0(s)}(\k) 
  \qquad\mbox{for}\qquad  \k\in\Omega_0(s)\,.
\end{aligned}
\end{equation}
Therefore if the data lies in
$$
   \mathcal{R}(F_T)
       := \left\{w\in L_c^2(\R^N)\,\left|\, 
          \frac{\hat w}{\hat G(\cdot,\tau)^m\,\hat G(\cdot,s)} 
              \in L^2(\R^N) \right. \right\}\,,
$$ 
then the initial concentration $u$ can be estimated in principle. 
In contrast to standard diffusion $\hat G(\k,s)$ has countably infinite and discrete zeros  
(cf. Theorem~\ref{theo:inj}). 
Hence it follows:

\begin{coro}
A necessary condition for $w\in \mathcal{R}(F_T)$ is that $\hat w$ has a zero of order 
$\geq m$ at $k_*$ if $\hat G_{c,\tau}(\cdot,T)$ has a zero of order $m$ at $k_*$. 
\end{coro}

According to Theorem~\ref{theo:Upsilon2} for $N\in\N$ we have  
$$
        \Upsilon_N(t) \asymp t^{(N-1)/2}  \qquad\mbox{for}\qquad t\to\infty\,
$$
and thus the envelope of $\k\mapsto\hat G_{c,\tau}(\k,T)$ 
decreases as
$$
        \k\mapsto   a_T\,|\k|^{(\lfloor T/\tau\rfloor+1)\,(N-1)/2}\,
        \qquad\mbox{for}\qquad |\k|\to\infty\,,
$$
where 
$$
      a_T := (c\,\tau)^{\lfloor T/\tau\rfloor\,(N-1)/2}\,
         \left( c\,( T-\lfloor T/\tau\rfloor \,\tau ) \right)^{(N-1)/2}\,.
$$
Here $\lfloor a\rfloor$ denotes the largest integer $\leq a$ (and   
$m\equiv\lfloor T/\tau\rfloor$, $s\equiv T-\tau\,\lfloor T/\tau\rfloor$).
Hence we get: 

\begin{coro}
If $N=2$ and $T>2\,\tau$ or $N\geq 3$ and $T>\tau$, then the inverse problem~(\ref{invprob}) 
is ill-posed, but \emph{not} exponentially ill-posed.
\end{coro}

We end this section with a remark about the \emph{technique of time reversal}.

\begin{rema}\label{rema:timereversal}
For  the special case $T\in (0,\tau]$, it follows from Corollary~\ref{coro:waveeq} that the 
inverse problem~(\ref{invprob}) depends continuously on the data if the additional data 
$w_2:=\frac{\partial F_T(u)}{\partial t}(\cdot,T)$ is known. 
More precisely, the solution can be calculated by 
\begin{equation*}
\begin{aligned}
   \hat u(\k) 
     =   \chi_{A}(\k)\, \frac{\hat w(\k)}{\Upsilon_N(|\k|\,c\,T)} 
       + \chi_{B}(\k)\, \frac{\hat w_2(\k)}{|\k|\,c\,\Upsilon_N'(|\k|\,c\,T)} \,,
\end{aligned}
\end{equation*} 
where $\chi_A$ is defined as in~(\ref{char}) and $\{A,\,B\}$ is a covering of $\R^N$ such that 
$A\cap B=\emptyset$ and $\Upsilon_N(|\cdot|\,T)$ and 
$\Upsilon_N'(|\cdot|\,T)$ do not vanish on $A$ and $B$, respectively. Here we have used the fact 
that the zeros of $\Upsilon_N$ and $\Upsilon_N'$ are of order one which follows from 
Corollary~\ref{coro:Upsilon} and Theorem~\ref{theo:Upsilon2} in the appendix. 
\end{rema}

\section{Simulation of the inverse problem}
\label{sec-Sim}

\subsection{Simulation of data via a particle method}
\label{subsec-SimData}

In order to avoid an \emph{inverse crime} we calculate the synthetic data for the 
inverse problem by a \emph{particle method} (cf.~\cite{GrKnZuCa04}). 
One of the advantages of a particle method (as long as no mass flows over the 
boundary) is that the total mass is conserved. 
For simplicity we focus on the $2D-$case and drop the subscripts $c$ and $\tau$ 
in $G_{c,\tau}$ and $v_{c,\tau}$. 

\subsubsection*{The particle method}

The initial distribution $u$ is approximated by an \emph{image}, i.e. a piecewise 
constant function with quadratic   pixels of length $\Delta x$. At time instant 
$t=\tau_{n-1}$ ($n\in\N$) the mass concentrated in a pixel separates in $M$ parts 
and each part propagates on a stright line with constant speed $c$ in a randomly 
chosen direction $\vd$ during the time period $\tau$. Here the directions are  
chosen with equal probability out of the set 
$$
   \{ A(\varphi)\,\e_1\,|\,\varphi=0,\,\pi/M,\,\ldots,\,(M-1)\,\pi/M \}\,,
$$
where $\e_1:=(1,0)^T$ and $A(\varphi)$ denotes the matrix that rotates the argument 
about the angle $\varphi$ in positive direction. 
This kind of data simulation allows that more than one "particle'' go in the 
same direction such that a special type of noise is included in the simulated  
data.  
To each image pixel is then associated the number of all particles that lie 
within the pixel multiplied by $1/M$.

\subsubsection*{Noise}

In order to avoid an inverse crime we perturbed the length of the radius 
$R(\tau)=c\,\tau$ by $\pm 0.25\%$ of its original length (uniformly distributed 
perturbation). In addition, uniformly distributed $L^2-$noise with positive 
mean value were added to the simulated data. As noise level we have chosen 
$\delta=0.005$ ($0.5\%$).

\subsubsection*{Convergence of the particle method}

In the following we denote by $G[M](\x,t)$ the simulated distribution with 
initial distribution 
\begin{equation*}
\begin{aligned}
   \delta[M](\x) 
        := \left\{
           \begin{array}{ll}
              1 & \mbox{if $\max(|x|,|y|)<\frac{\Delta x}{2}$}\\
              0 & \mbox{elsewhere}
           \end{array}
           \right.  \,.
\end{aligned}
\end{equation*}
From analysis it is known that 
\begin{equation}\label{convprop}
\begin{aligned}
     \delta[M](\x)\stackrel{M\to\infty}{\longrightarrow} \delta(\x) 
\quad\mbox{and}\quad 
     G[M](\x,t)\stackrel{M\to\infty}{\longrightarrow} G(\x,t)
\quad \mbox{in}\quad \mathcal{D}'(\R^2)\,.
\end{aligned}
\end{equation}
Here $G$ denotes the Green function of causal diffusion (cf. 
Definition~\ref{defCD}) and $\mathcal{D}'(\R^2)$ denotes the space 
of distributions on $\R^2$. We now show that the algorithm described above for 
an initial distribution $u$ provides us with an approximate solution of $F_T(u)$, 
were $F_T$ denotes the forward operator~(\ref{FTv1}).

\begin{theo}
Let $u\in L_c^1(\R^2)$ and $v(\cdot,t):=G(\cdot,t)*_\x u$. For 
$$
   v[M](\cdot,t) := G[M](\cdot,t)*_\x u  \qquad(t>0)\,,
$$
it follows that
$$
    v[M](\cdot,t) \stackrel{M\to\infty}{\longrightarrow} v(\cdot,t) 
   \qquad \mbox{in}\qquad L^1(\R^2)\,.
$$
\end{theo}

\begin{proof}
Since the space $C_0^\infty(\R^2)$ is dense in $L^1(\R^2)$, we assume
without loss of generality that $u\in C_0^\infty(\R^2)$. We have 
\begin{equation*}
\begin{aligned}
  \| v[M](\cdot,t)-v(\cdot,t)\|_{L^1}
  = \int_{\R^2} \left| f[M](\x) \right| \d\x  
\end{aligned}
\end{equation*}
with 
\begin{equation}\label{f[M]}
\begin{aligned}
 f[M](\x) = \int_{\R^2} [G[M](\x',t)-G(\x',t)]\, u(\x-\x')\, \d\x'\,. 
\end{aligned}
\end{equation}
The function $f[M]$ is an element of $C_0^\infty(\R^2)$, since $G[M](\cdot,t)-G(\cdot,t)$ 
has compact support and $u\in C_0^\infty(\R^2)$ (cf. Proposition~32.1.1 
in~\cite{GasWit99}). 
From ~(\ref{convprop}) together with $u\in C_0^\infty(\R^2)$, it follows that the 
right hand side of~(\ref{f[M]}) converges pointwise and uniformly to zero on compact sets. 
This together with the fact that $f[M]$ has compact support implies 
$$
     \|v[M](\cdot,t)-v(\cdot,t)\|_{L^1}\stackrel{M\to\infty}{\longrightarrow} 0\,.
$$
As was to be shown. 
\end{proof}

\subsection{Numerical solution of the backwards diffusion problem}
\label{subsec-SimInv}

For solving the inverse problem we use the \emph{Landweber method} 
(cf. e.g.~\cite{EngHanNeu96,Kir96,Isa98,Lo01}).
Since $F_T:L_c^2(\R^N)\to L_c^2(\R^N)$ is a positive, linear and self-adjoint operator 
(cf. Theorem~\ref{theo:FTsa}) the Landweber method reads as 
follows\footnote{For simplicity, we write $u_n$, $w^\delta$ instead of $u_n[M]$, 
$w^\delta[M]$.}
\begin{equation*}
\begin{aligned}
     u_{n+1} = P\{u_n -\omega\,F_T\,[F_T(u_n)- w^\delta]\}\,,
\end{aligned}
\end{equation*}
where $\omega$ denotes the \emph{relaxation parameter}, $w^\delta$ denotes the noisy data 
and $P$ denotes the orthogonal projection onto  
$$
    \mathcal{R}(P)=\{u\in L^2\,|\,u\geq 0\}\,.
$$ 
The use of the projection operator guarantees that the solution is a positive (mass) 
distribution. As parameter choice rule we use the 
\emph{discrepancy principle}, i.e. the iteration is stopped as soon as 
\begin{equation*}
\begin{aligned}
     \|F_T(u_{n+1})-w^\delta\|_{L^2} < \eta\,\delta \qquad\quad (\eta\geq 2)
\end{aligned}
\end{equation*}
is true.  The relaxation parameter was chosen as
$$
   \omega := \frac{1}{4}\,\frac{\|F_T(u_n)- w^\delta\|_{L^2}^2}
                               {\|F_T\,[F_T(u_n)- w^\delta]\|_{L^2}^2}\,.
$$

In order to avoid an inverse crime, the data $w^\delta$ is calculated by the particle 
method ($M=65$) described above and the calculation of the Forward operator $F_T$ in each 
iteration step is performed by integrals over circles. Each circle is discretized 
by $50$ points. 

We now present two simulations of the backwards diffusion problem for $T=\tau$ and 
$T=3\,\tau$, respectively.

\begin{exam}\label{exam:invpr}
Consider the initial distribution shown in Fig.~\ref{fig:6}.  
This image consists of $682^2$ quadratic pixels of length $\Delta x:=1/681$.  
As characteristic parameters of causal diffusion we have chosen $c=1$ and 
$\tau=8\,\Delta x/c$. Hence the characteristic radius $R(\tau)$ is $8$ 
times $\Delta x$. As described above,  the length of the radius $R(\tau)$ was 
randomly perturbed by $\pm 0.25\%$ of its original length. 
The data acquisition is performed at time $T=\tau$ and $0.5\%$ 
(uniformly distributed) $L^2-$noise was added to the simulated data.     
The numerical results are visualized in Fig.~\ref{fig:6} and Fig.~\ref{fig:7}. 
As expected, the estimation of large structures is much better than for smaller 
ones. Since the data acquisition is performed at a quite early time the 
estimation works well.\footnote{Cf. Remark~\ref{rema:timereversal}.} 
The Discrepancy principle stops optimally for $\eta=9.4$ after $6$ steps. 
\end{exam}

\begin{exam}
We consider the inverse problem from Example~\ref{exam:invpr} again, but for the  
later data acquisition time $T=3\,\tau$. The Discrepancy principle stops optimally 
for $\eta=5.9$ after $6$ steps. For this situation the forward operator is compact. 
As Fig.~\ref{fig:8} shows it is not possible to restore the edges of the question 
mark, since the data are to much ``smooth''.  This result reflects the ill-posedness 
of the problem. 
\end{exam}

\section{Appendix}

\subsection*{The delta distribution}

We use the following notation for the \emph{delta distributions}. Let $N\in\N$ and 
$\x\in\R^N$. Then $\delta(\x)$ satisfies 
\begin{equation*}
\begin{aligned}
   \int_{\R^N} f(\x)\,\delta(\x-\x_0)\,\d\,\x = f(\x_0)  
  \qquad\mbox{for}\qquad f\in C_c(\R^N)\,.
\end{aligned}
\end{equation*}
Here $C_c(\R^N)$ denotes the set of continuous funtions with compact support. Since 
$\delta(\x)$ has compact support, $C_c(\R^N)$ can be replaced by $C(\R^N)$.
In this notation the dirac measure $\mu_\delta$ (cf.~\cite{La93}) reads as follows
$$
    \mu_\delta(A)  
        =  \int_{\R^N} \chi_{A}(\x) \delta(\x)\,\d\,\x  
        = \left\{
           \begin{array}{ll}
              1 & \mbox{if $\vO\in A$}\\
              0 & \mbox{elsewhere}
           \end{array}
           \right.  \,,
$$
where 
\begin{equation}\label{char}
 \mbox{$\chi_{A}(\x)$ denotes the characteristic function of the set $A\subseteq \R^N$.}
\end{equation} 
In case $N=1$ we use the notation $\delta(x)$ instead 
of $\delta(\x)$.

\subsection*{The Fourier transform}

We use the following notation for the Fourier transformation: 
\begin{equation*}
\begin{aligned}
   &\hat f(\k)   := \F\{f\}(\k)     
       := (2\,\pi)^{-N/2}\, \int_{\R^N} e^{\i\,\k\cdot\x} f(\x) \,\d\x \\
   &\check g(\x) := \F^{-1}\{g\}(\x) 
       :=  (2\,\pi)^{-N/2}\, \int_{\R^N} e^{-\i\,\k\cdot\x} g(\k) \,\d\k\, 
\end{aligned}
\end{equation*}
for $f,\,g\in L^1(\R^N)$. Here $\k\in\R^N$ is called the wave vector. In this notation the 
convolution theorem reads as follows
\begin{equation}\label{convth}
\begin{aligned}
  \F\{f_1\}\,\F\{f_2\} 
     = (2\,\pi)^{-N/2}\,\F\{f_1 *_\x f_2\} \qquad f_1,\,f_2\in L^1(\R^N)\,. 
\end{aligned}
\end{equation}

\subsection*{Special functions}

We define the function $\mbox{sinc}:\R\to\R$ as the continuous extension of 
$x\in\R\backslash \{0\}\mapsto sin(x)/x$ and recall that the \emph{Bessel function} 
of first kind and order zero has the series representation (cf.~\cite{Heu91})
\begin{equation}\label{J0}
    \mbox{J}_0(x) 
        =  \sum_{j=0}^\infty (-1)^j\,\left(\frac{x^j}{2^j\, j!}\right)^2 \,.
\end{equation}

In order to derive an analytic representation of the Fourier transform of the 
Green function of causal diffusion, we need the following two lemmata. 

\begin{lemm}\label{lemm:03}
For $N\in\N$ with $N>1$ and $j\in\N$, let  
\begin{equation*}
\begin{aligned}
    I(N,j) := \int_{-\pi/2}^{\pi/2} \sin^j(\varphi)\cos^{N-2}(\varphi) \,\d\varphi \,.
 \end{aligned}
\end{equation*}
If $j$ is odd, then $I(N,j)=0$ and if $j$ is even, then
\begin{equation}\label{INj}
\begin{aligned}
  \frac{I(N,j)}{I(N,0)} 
   = \frac{1\,\cdot 3\cdot 5 \cdots (j-1)}{N\cdot (N+2)\cdot (N+4)\cdots (N+j-2)}  \, \,.
\end{aligned}
\end{equation}
\end{lemm}

\begin{proof}
If $j$ is odd, then $\sin^j(\varphi)\cos^{N-2}(\varphi)$ is an odd function and thus 
$I(N,j)$ vanishes. 

Now let $j$ be even.  We perform a proof by induction. 
\begin{itemize}
\item [i)]   Let $j=2$. Integration by Parts yields 
             \begin{equation*}
             \begin{aligned}
               I(N,2) 
                   = \frac{1}{N}\, \int_{-\pi/2}^{\pi/2} \cos^{N-2}(\varphi) \,\d\varphi
                   = \frac{I(N,0)}{N}\,.
             \end{aligned}
             \end{equation*}

\item [ii)]  We now assume the induction assumption for $j=m-2$, i.e. 
             \begin{equation*}
             \begin{aligned}
                I(N,m-2) 
                  = \frac{1\,\cdot 3\cdot 5 \cdots (m-3)}{N\cdot (N+2)\cdot (N+4)\cdots (N+m-4)}  
                                \, I(N,0)\,
             \end{aligned}
             \end{equation*}
             and prove~(\ref{INj}) for $j=m$. 
             Integration by Parts yields 
             $$
               I(N,m) = \frac{m-1}{m+N-2}\, I(N,m-2)\,.
             $$
             Employing the induction assumption to this result leads to~(\ref{INj}) 
             with $j=m$. This concludes the proof.
\end{itemize}
\end{proof}

\begin{lemm}\label{lemm:hatGc}
Let 
\begin{equation}\label{defUpsilon}
\begin{aligned}
    \Upsilon_N(t) 
            :=  \sum_{j=0}^\infty (-1)^j\cdot a_{2\,j}\cdot t^{2\,j} 
        \qquad \mbox{for} \qquad t\in (0,\infty) 
\end{aligned}
\end{equation}
with $a_0=1$ and 
\begin{equation*}
\begin{aligned}
    a_{2\,j} = \frac{1}{(2\,j)!}\,\frac{1\cdot 3\cdot 5 \cdots (2\,j-1)}
                                      {N\cdot (N+2)\cdot(N+4) \cdots (N+2\,j-2)}
\qquad\quad (j\in\N)\,.
\end{aligned}
\end{equation*} 
The series~(\ref{defUpsilon}) is absolutely convergent and 
\begin{equation*}
\begin{aligned}
         \int_{S_1(\vO)} \frac{\delta(\x+c\,s\,\y)}{|S_1(\vO)|}\,\d \sigma(\y)
            =  \frac{\F^{-1}\{ \Upsilon_N(|\cdot|\,c\,s) \}(\x)}{(2\,\pi)^{N/2}}
        \qquad \mbox{for} \qquad \x\in\R^N,\, s\in (0,\tau]\,.
\end{aligned}
\end{equation*}
\end{lemm}

\begin{proof}
That the series representation~(\ref{defUpsilon}) converges absolutely follows at once from  
the Quotient Criterion. 

Let $\x,\,\k\in\R^N$. From  
$$
   \F\{\delta(\x + a_1\,\y)\}(\k) 
      = (2\,\pi)^{-N/2}\, e^{-\i\,a_1\,(\k\cdot\y)}
      \qquad\quad \mbox{($a_1>0$ constant)}
$$  
and 
\begin{equation*}
\begin{aligned}
    \int_{S_1(\vO)} e^{a_2\,(\k\cdot \y)} \d \sigma(\y)  
       =  \int_{S_1(\vO)} e^{a_2\,|\k|\,(\e_1\cdot \y)} \d \sigma(\y) 
   \qquad\quad \mbox{($a_2\in\C$ constant)}, 
\end{aligned}
\end{equation*}
it follows that
\begin{equation*}
\begin{aligned}
   \hat g(\k,s):=
 \F\left\{
         \int_{S_1(\vO)} \frac{\delta(\cdot+c\,s\,\y)}{|S_1(\vO)|}\,\d \sigma(\y)
          \right\}(\k)
    = \int_{S_1(\vO)} \frac{e^{-\i\, (\e_1\cdot \y) \, |\k|\,c\,s}}{(2\,\pi)^{N/2}\,|S_1(\vO)|}
                      \,\d \sigma(\y)   
\end{aligned}
\end{equation*}
for $s\in [0,\tau]$ and $\k\in\R^N$. 
Instead of $\e_1$ we can also use anyone in $\{\e_2,\,\e_3\,,\ldots,\,\e_N\}$. 
Expanding the exponential function yields 
\begin{equation}\label{repGc01}
\begin{aligned}
   (2\,\pi)^{N/2}\,\hat g(\k,s) 
    =  \sum_{j=0}^\infty (-1)^j\,\frac{(|\k|\,c\,s)^{2\,j}}{(2\,j)!}\,d_{2\,j} \,.
\end{aligned}
\end{equation}
with 
\begin{equation*}
\begin{aligned}
   d_j  :=  \int_{S_1(\vO)}  \frac{(\e_1\cdot \y)^j }{|S_1(\vO)|}\,\d \sigma(\y)     
  \qquad \mbox{for}\qquad j\in\N\,.
\end{aligned}
\end{equation*}
We see at once that $d_j=0$ if $j$ is odd and $d_0=1$.  
For the convenience of the reader, we consider the cases $N=1$ and $N>1$ separately. 
\begin{itemize}
\item [a)] For $N=1$ we have $\e_1\equiv 1$
      $$
        \int_{S_1(\vO)} \d \sigma(\y) \equiv \int_\R (\delta(y-1) + \delta(y+1))\,\d\,y
        \qquad\mbox{and}\qquad  |S_1(\vO)|=2 \,,
      $$
      and thus
      \begin{equation*}
      \begin{aligned}
          d_j = \frac{1^j + (-1)^j }{2} 
           = \left\{
           \begin{array}{ll}
              1 & \mbox{if $j$ is even}\\
              0 & \mbox{if $j$ is odd}
           \end{array}
           \right.  \,.
      \end{aligned}
      \end{equation*}
      Inserting this into the series representation yields
      \begin{equation*}
      \begin{aligned}
         (2\,\pi)^{N/2}\,\hat g(\k,s) 
             =  \sum_{j=0}^\infty (-1)^j\,\frac{(|k|\,c\,s)^{2\,j}}{(2\,j)!} \,. 
      \end{aligned}
      \end{equation*}

\item [b)] Let $N>1$. For the derivation of the series representation we use the following 
           $N-$dimensional orthogonal coordinate system (cf.~\cite{StWi93})
      $$
       (r,\varphi_1,\,\ldots\,,\varphi_{N-1}) 
         \in [0,\infty)\times (-\pi,\pi)\times (-\pi/2,\pi/2)^{N-2} 
     $$ 
     defined by 
     \begin{equation*}
     \begin{aligned}
        &x_N    = r\,\sin(\varphi_{N-1})\,,  \\
        &x_{N-1} = r\,\cos(\varphi_{N-1})\,\sin(\varphi_{N-2})\,,  \\
        &x_{N-2} = r\,\cos(\varphi_{N-1})\,\cos(\varphi_{N-2})\,\sin(\varphi_{N-3})\,,  \\
        &x_{N-3} = r\,\cos(\varphi_{N-1})\,\cos(\varphi_{N-2})\,\cos(\varphi_{N-3})
                                   \,\sin(\varphi_{N-4})\,,  \\
      & \qquad \qquad\qquad\qquad \vdots \\ 
%        &x_3    = r\,\cos(\varphi_{N-1}) \,\cdots\,\cos(\varphi_3)\,\cos(\varphi_{N-2})\,, \\
        &x_2    = r\,\cos(\varphi_{N-1}) \,\cdots\,\cos(\varphi_2)\,\sin(\varphi_1)\,,  \\
        &x_1    = r\,\cos(\varphi_{N-1}) \,\cdots\,\cos(\varphi_2)\,\cos(\varphi_1)\,,  
     \end{aligned}
     \end{equation*}
     with surface measure 
     \begin{equation*}
     \begin{aligned}
        \d\sigma  
         = r^{N-1}\,\d\varphi_1\, \prod_{l=2}^{N-1} \cos^{l-1}(\varphi_l) \,\d\varphi_l \,.
     \end{aligned}
     \end{equation*}
     Since $\e_N\cdot \e_r =  \sin(\varphi_{N-1})$ and 
     \begin{equation*}
     \begin{aligned}
      |S_1(\vO)|   
       = \int_{S_1(\vO)} \d\varphi_1\, \prod_{l=2}^{N-1} \cos^{l-1}(\varphi_l) \,\d\varphi_l \,,
     \end{aligned}
     \end{equation*}
     we obtain
     \begin{equation*}
     \begin{aligned}
       d_j  
       &= \int_{S_1(\vO)} \d\varphi_1\, \prod_{l=2}^{N-2} \cos^{l-1}(\varphi_l) \,\d\varphi_l 
              \,\frac{\sin^j(\varphi_{N-1})}{|S_1(\vO)|}\,\cos^{N-2}(\varphi_{N-1}) 
                       \,\d\varphi_{N-1}\\
      &= \frac{I(N,j)}{I(N,0)}
     \end{aligned}
     \end{equation*}
     with $I(N,j)$ defined as in Lemma~\ref{lemm:03}. 
     From this and Lemma~\ref{lemm:03}, we obtain $d_{2\,j-1}=0$, $d_0=1$ and 
     \begin{equation*}
     \begin{aligned}
       d_{2\,j}
         = \frac{1\,\cdot 3\cdot 5 \cdots (2\,j-1)}{N\cdot (N+2)\cdot (N+4)\cdots (N+2\,j-2)} 
          \qquad \mbox{for}\qquad j>0\,.
     \end{aligned}
     \end{equation*}
     Inserting this into the series~(\ref{repGc01}) yields the claimed series representation. 
\end{itemize}
This concludes the proof. 
\end{proof}

The following corollary follows form Lemma~\ref{lemm:hatGc}.

\begin{coro}\label{coro:Upsilon}
For $N\in\N$. The function $\Upsilon_N$ defined as in~(\ref{defUpsilon}) satisfies the 
problem 
\begin{equation*}
\begin{aligned}
    \Upsilon_N''(t) + \frac{(N-1)}{t}\,\Upsilon_N'(t) + \Upsilon_N(t) = 0 
   \qquad\quad t>0\,, 
\end{aligned}
\end{equation*}
with initial conditions
\begin{equation*}
\begin{aligned}
    \Upsilon_N(0+) = 1  \qquad\mbox{and}\qquad  \Upsilon_N'(0+) = 0\,.
\end{aligned}
\end{equation*}
Here $t=0$ is a \emph{regular singular point}\footnote{Cf. e.g.~\cite{Heu91}.} 
of the ordinary differential equation. 
\end{coro}

The following theorem enables us to specify the space Fourier transfrom of the Green 
function of causal diffusion for every dimension $N$ and to prove some compactness 
results for the forward operator of causal diffusion. 

\begin{theo}\label{theo:Upsilon2}
Let $N\in\N$ with $N\geq 3$ and $t>0$. The function $\Upsilon_N$ defined as in~(\ref{defUpsilon}) 
satisfies 
\begin{equation}\label{I}
\begin{aligned}
    \Upsilon_N(t) = -\frac{(N-2)}{t}\,\Upsilon_{N-2}'(t)   
\end{aligned}
\end{equation}
with 
\begin{equation*}
\begin{aligned}
    \Upsilon_1(t) = \cos(t)
\qquad\mbox{and}\qquad 
    \Upsilon_2(t) = \mbox{J}_0(t)  \,.
\end{aligned}
\end{equation*} 
Here $\mbox{J}_0$ denotes the \emph{Bessel function} of first kind and order zero. 
Moreover, we have
\begin{equation}\label{estUpsilon}
\begin{aligned}
   |\Upsilon_N(t)| \leq C_N\,t^{-(N-1)/2}  \qquad\mbox{for sufficiently large $t$}
\end{aligned}
\end{equation} 
and some constant $C_N>0$. 
\end{theo}

\begin{proof}
The relation between $\Upsilon_N$ and $\Upsilon_{N-2}'$ follows at once from the series 
representation~(\ref{defUpsilon}). 
Moreover, 
\begin{itemize}
\item [a)] if $N=1$, then 
           \begin{equation*}
           \begin{aligned}
                  a_{2\,j}
                       = \frac{1}{(2\,j)!}\,\frac{1\,\cdot 3\cdot 5 \cdots (2\,j-1)}
                                                 {1\cdot 3\cdot 5\cdots (2\,j-1)} 
                       = \frac{1}{(2\,j)!}\,
           \end{aligned}
           \end{equation*}
           and thus $\Upsilon(t)=\cos(t)$ and 

\item [b)] if $N=2$, then 
           \begin{equation*}
           \begin{aligned}
                  a_{2\,j}
                     = \frac{1}{(2\,j)!}\,\frac{1\,\cdot 3\cdot 5 \cdots (2\,j-1)}
                                                 {2\cdot 4\cdot 6\cdots (2\,j)} 
                     = \frac{1}{[2\cdot 4\cdot 6\cdots (2\,j)]^2} 
                     =  \frac{1}{[2^j\, (j!)]^2}\,,
           \end{aligned}
           \end{equation*}
           which implies $\Upsilon(t) = \mbox{J}_0(t)$ 
           (cf.~(\ref{J0}) in the Appendix and~\cite{Heu91}). 

\end{itemize} 
In order to proof the estimation we use 
\begin{equation}\label{II}
\begin{aligned}
    \frac{(t^{N-2}\,\Upsilon_N(t))'}{(N-2)\,t^{N-3}} = \Upsilon_{N-2}(t)  \,,
\end{aligned}
\end{equation}
which follows from the series representation~(\ref{defUpsilon}). We perform a proof by 
induction. Since $cos(t)$ is bounded and the Bessel function $\mbox{J}_0(t)$ satisfies 
the asymptotic behavior (cf.~\cite{BroSem79})
\begin{equation}\label{asymJ0}
   \mbox{J}_0(t) \asymp \sqrt{\frac{2}{\pi\,t}}\,\cos\left(t-\frac{\pi}{4}\right)
           \qquad\mbox{for}\qquad t\to\infty\,,
\end{equation}
the estimation holds for $N=1$ and $N=2$. We assume that the estimation~(\ref{estUpsilon}) 
holds and prove 
\begin{equation*}
\begin{aligned}
   |\Upsilon_{N+2}(t)| \leq C_{N+2}\,t^{-(N+1)/2}  \qquad\mbox{for sufficiently large $t$.}
\end{aligned}
\end{equation*} 
From~(\ref{I}) and~(\ref{II}) we get
\begin{equation*}
\begin{aligned}
   |\Upsilon_{N+2}(t)| 
       &=    \left| \frac{N}{t}\,\Upsilon_N'(t) \right| 
        \leq \frac{N\,(N-2)}{t^2}\, ( |\Upsilon_N(t)| + |\Upsilon_{N-2}(t)|) \\
       &\leq \frac{N\,(N-2)\,(C_N+C_{N-2})}{t^{(N+1)/2}}\,,
\end{aligned}
\end{equation*}
which proves the claim.  
\end{proof}

% --- sinc  ---
\begin{figure}[!ht]
\begin{center}
\includegraphics[height=4.5cm,angle=0]{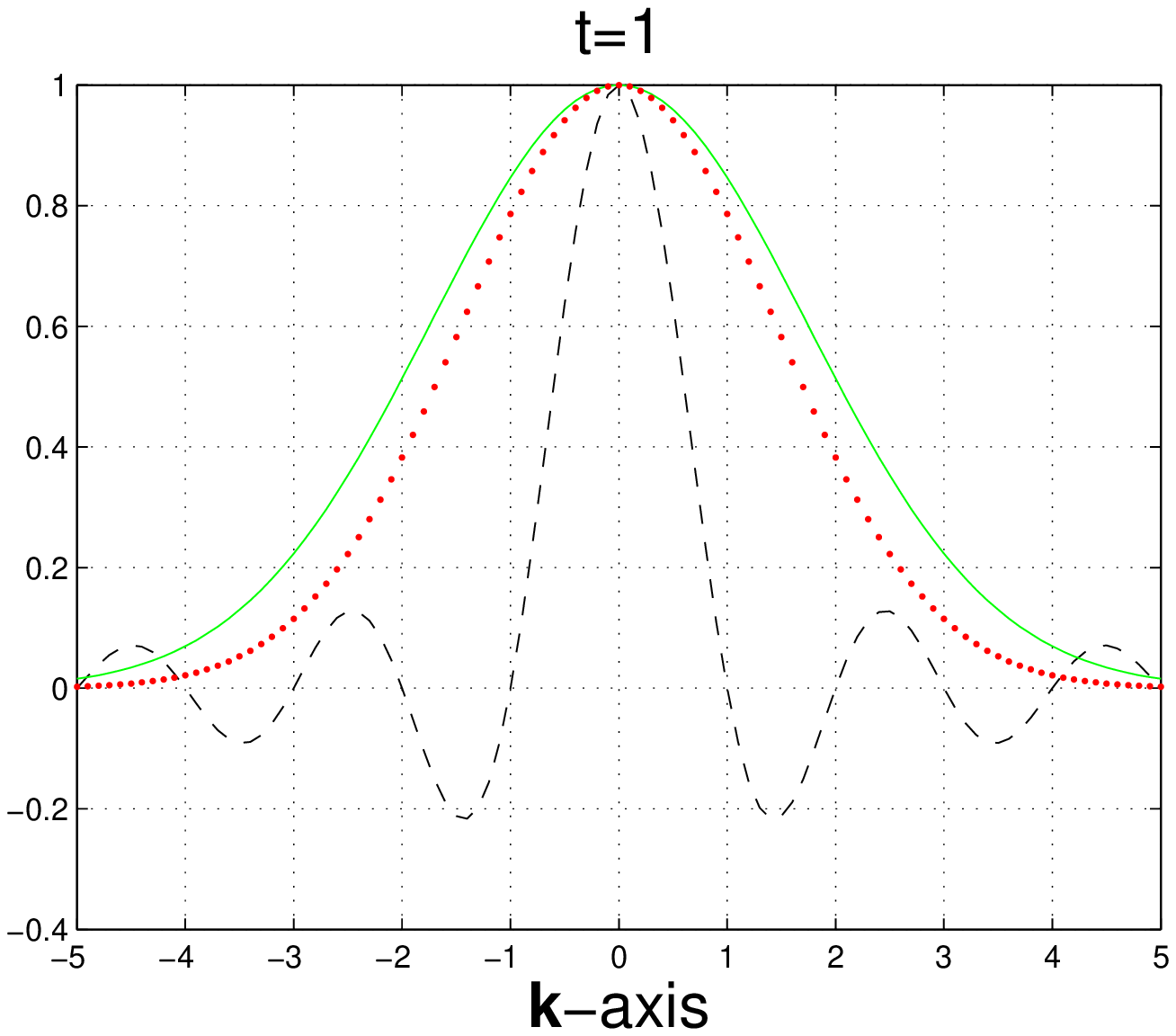}
\includegraphics[height=4.5cm,angle=0]{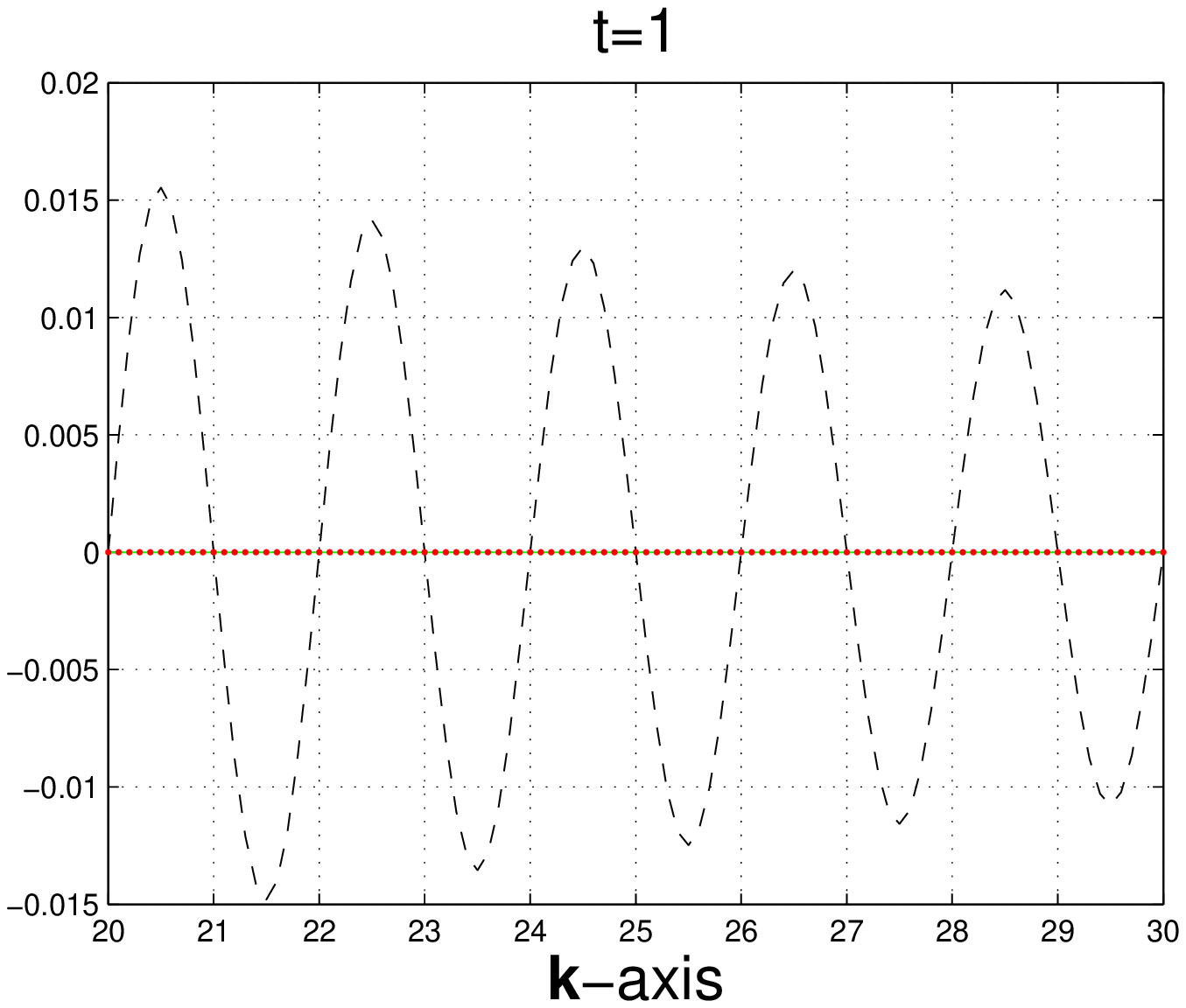}\\
\includegraphics[height=4.5cm,angle=0]{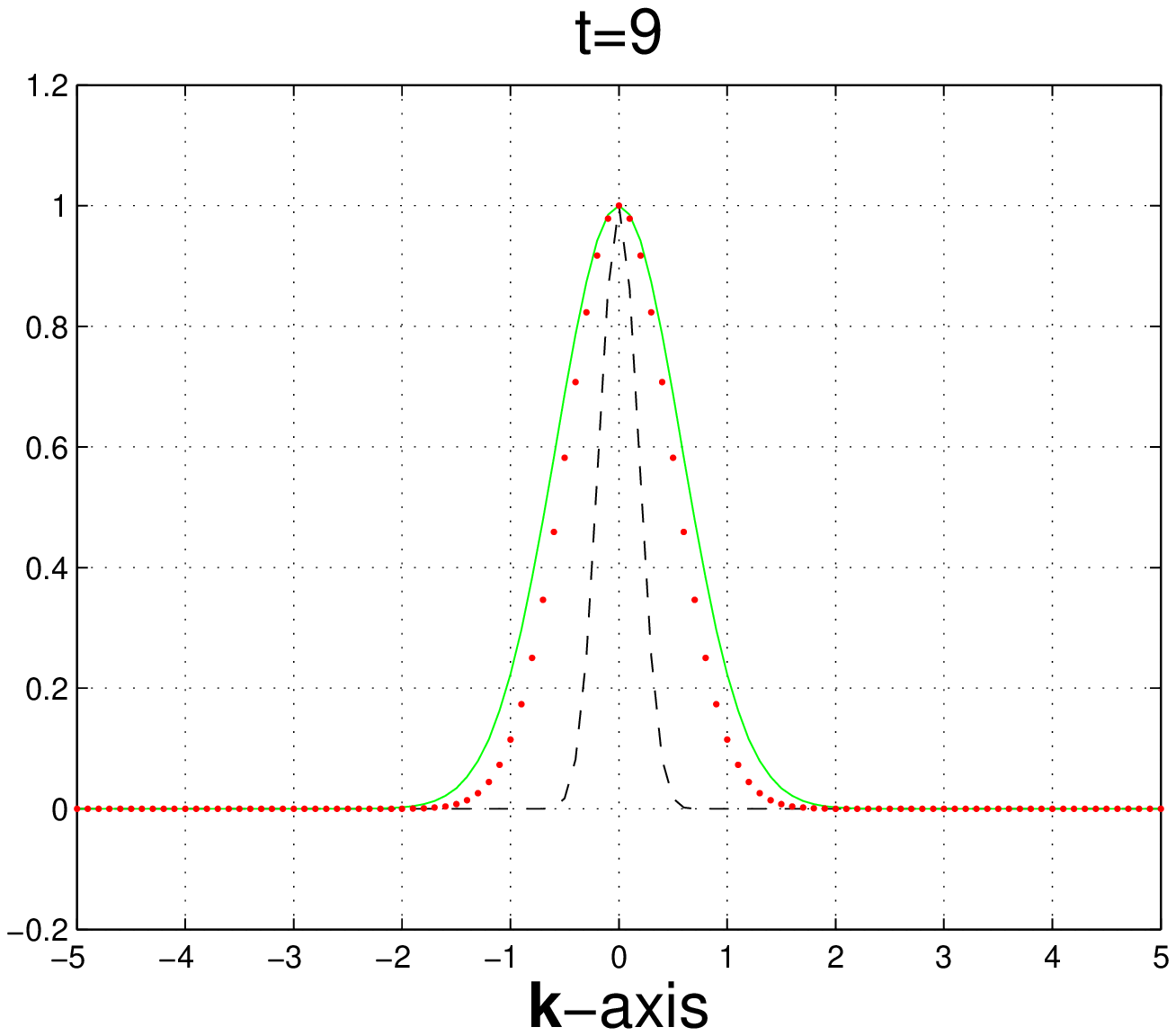}
\includegraphics[height=4.5cm,angle=0]{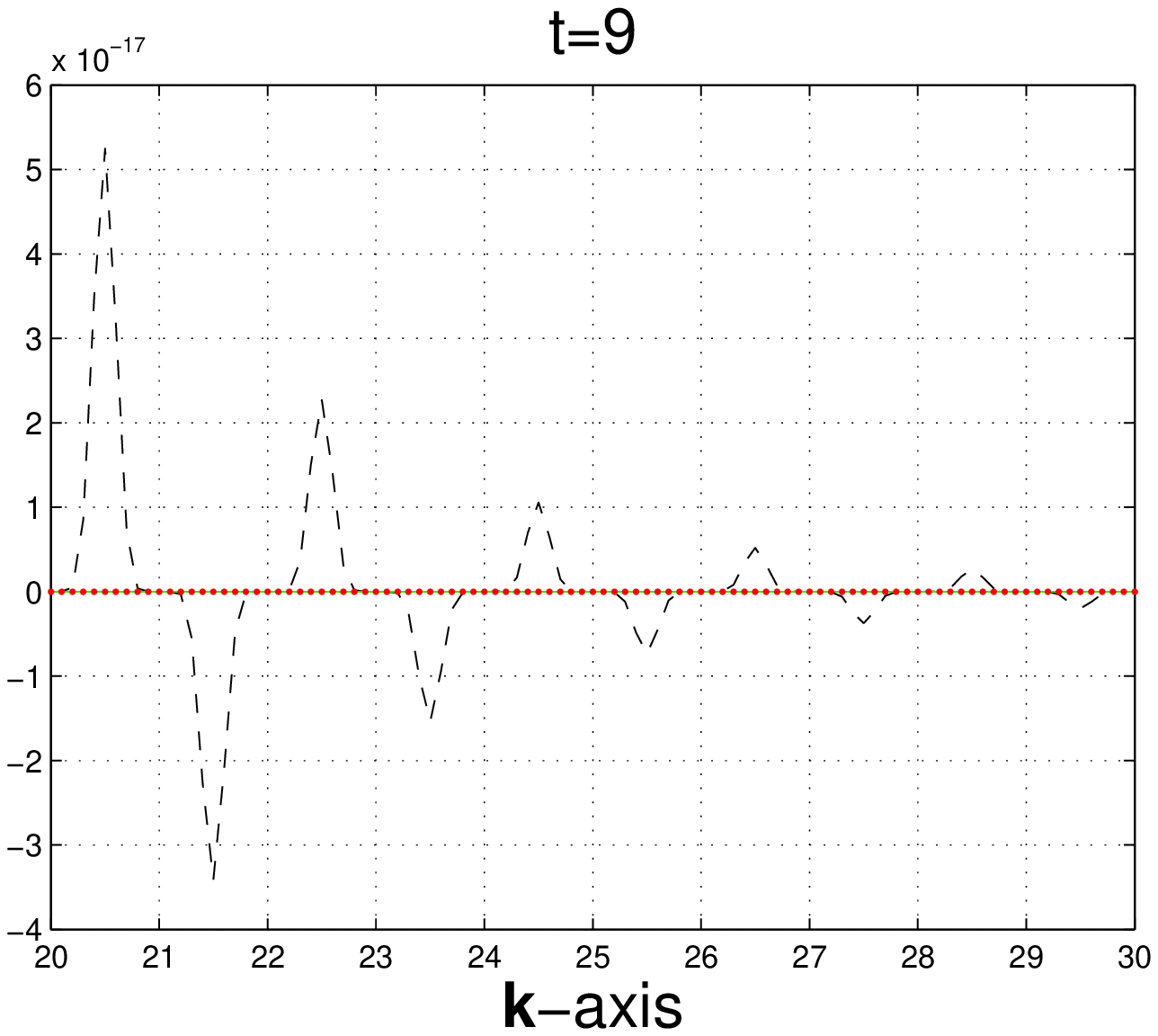}
\end{center}
\caption{Comparison of 
         $|\k|\mapsto (2\,\pi)^{3/2}\,\hat G_\#(\k,t)$ (dotted red line),
         $|\k|\mapsto (2\,\pi)^{3/2}\,\hat G_{\infty,0}(\k,t)$ (solid green line) and 
         $|\k|\mapsto (2\,\pi)^{3/2}\,\hat G_{c,\tau}(\k,t)$ (dashed black line)\,. 
         Here $c=1$, $\tau=1$ and $t\in\{\tau,\,9\,\tau\}$.}
\label{fig:image02} 
\end{figure}

% --- 2D scale space 0) ---
\begin{figure}[!ht]
\begin{center}
\includegraphics[height=5.0cm,angle=0]{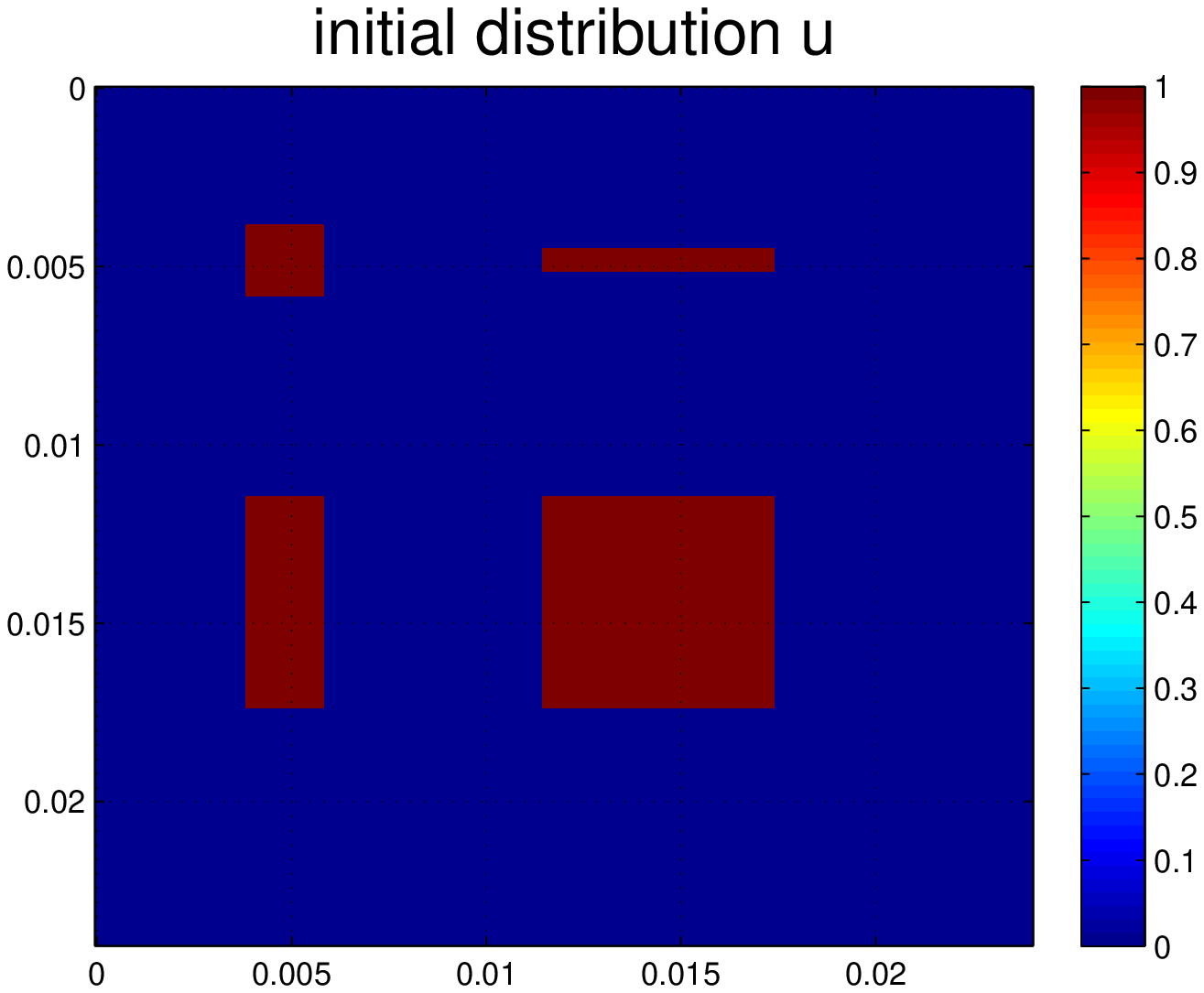}
\includegraphics[height=4.7cm,angle=0]{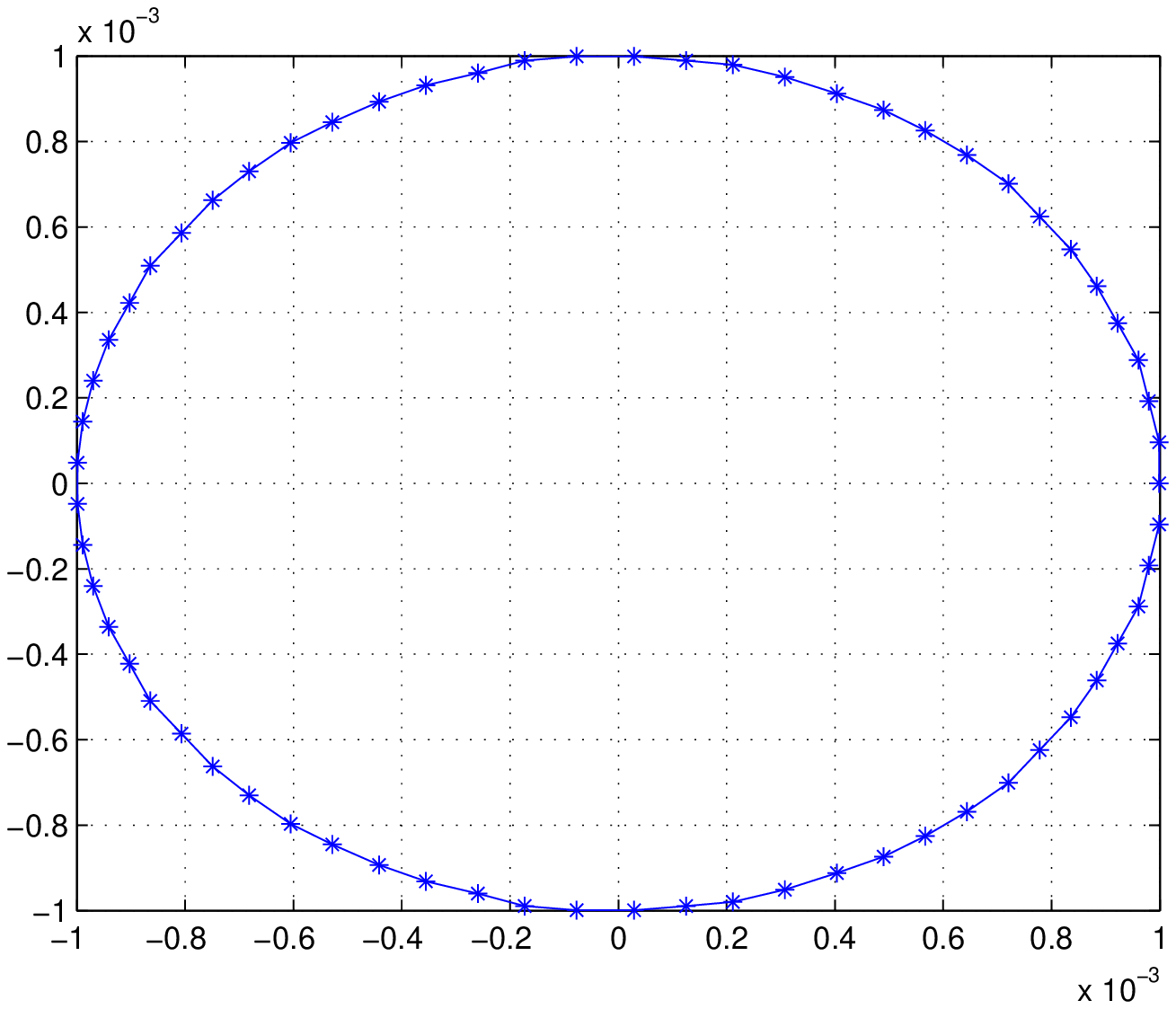}\\
\end{center}
\caption{The left picture shows the initial distribution $u$ and the right picture 
visualizes the discretization of the circle of radius $R=R(\tau)$ that is used to 
evaluate the circle integrals in Definition~\ref{defCD}.} 
\label{fig:ex01}
\end{figure}

% --- 2D scale space 1) ---
\begin{figure}[!ht]
\begin{center}
\includegraphics[height=5cm,angle=0]{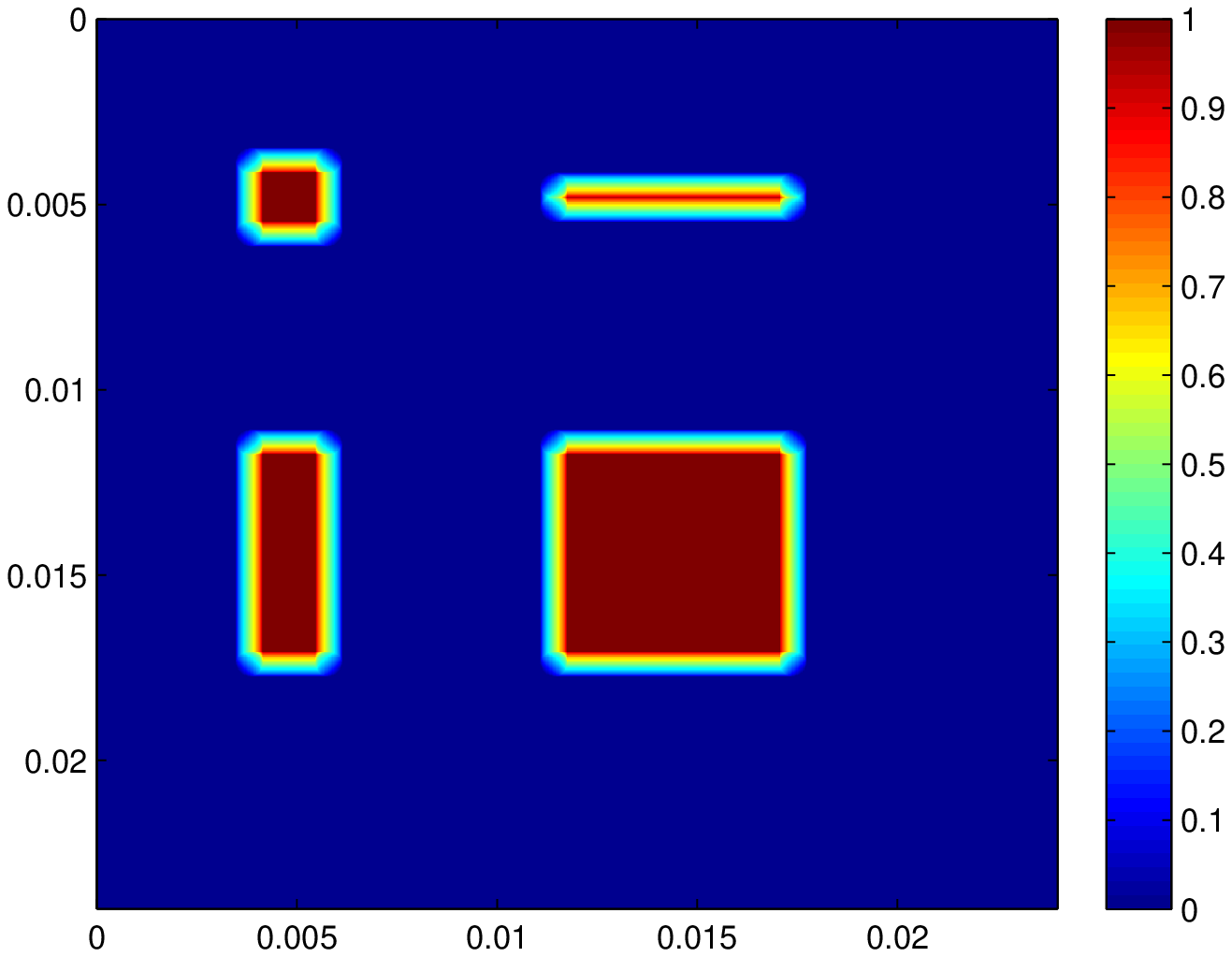}
\includegraphics[height=5cm,angle=0]{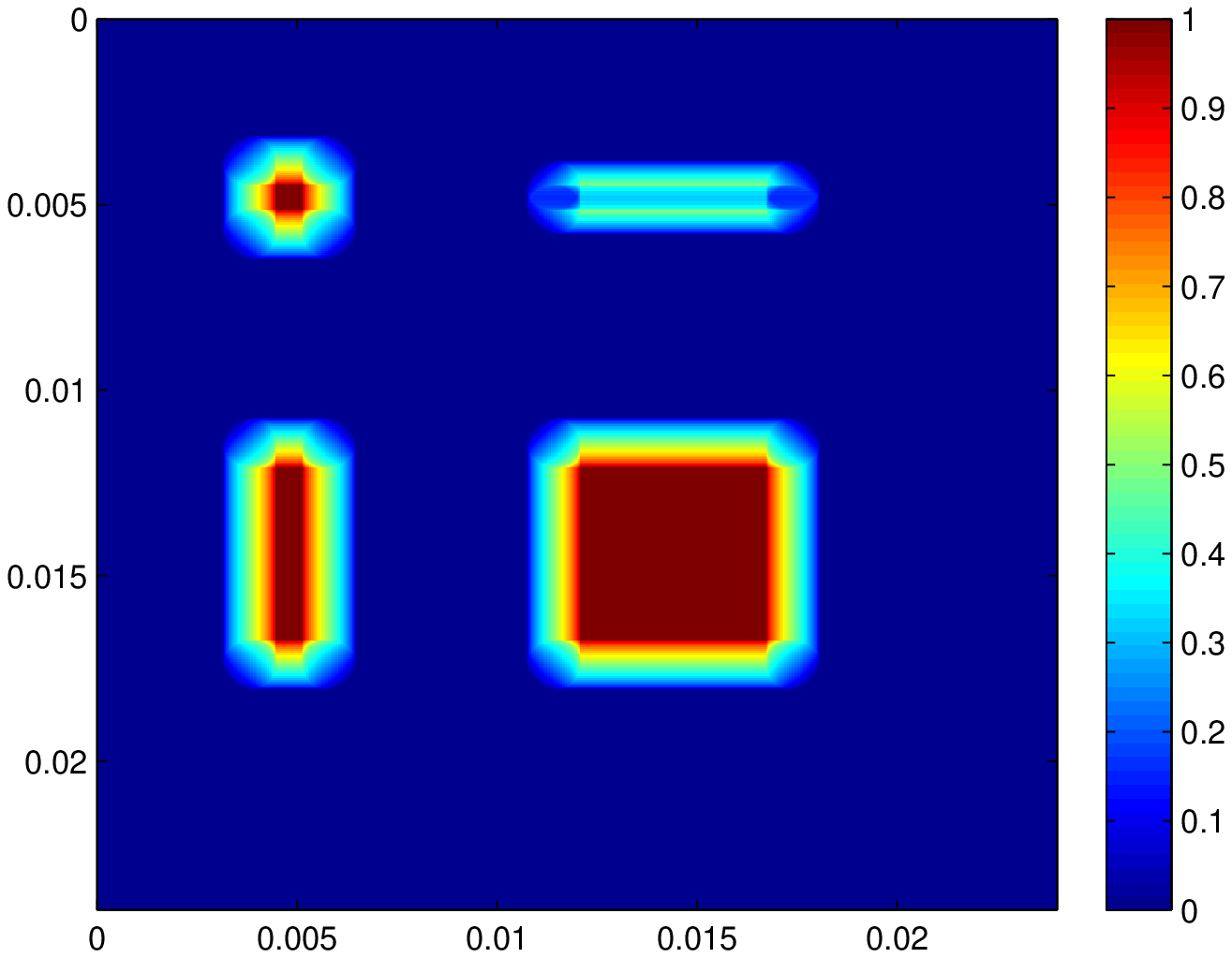}\\
\includegraphics[height=5cm,angle=0]{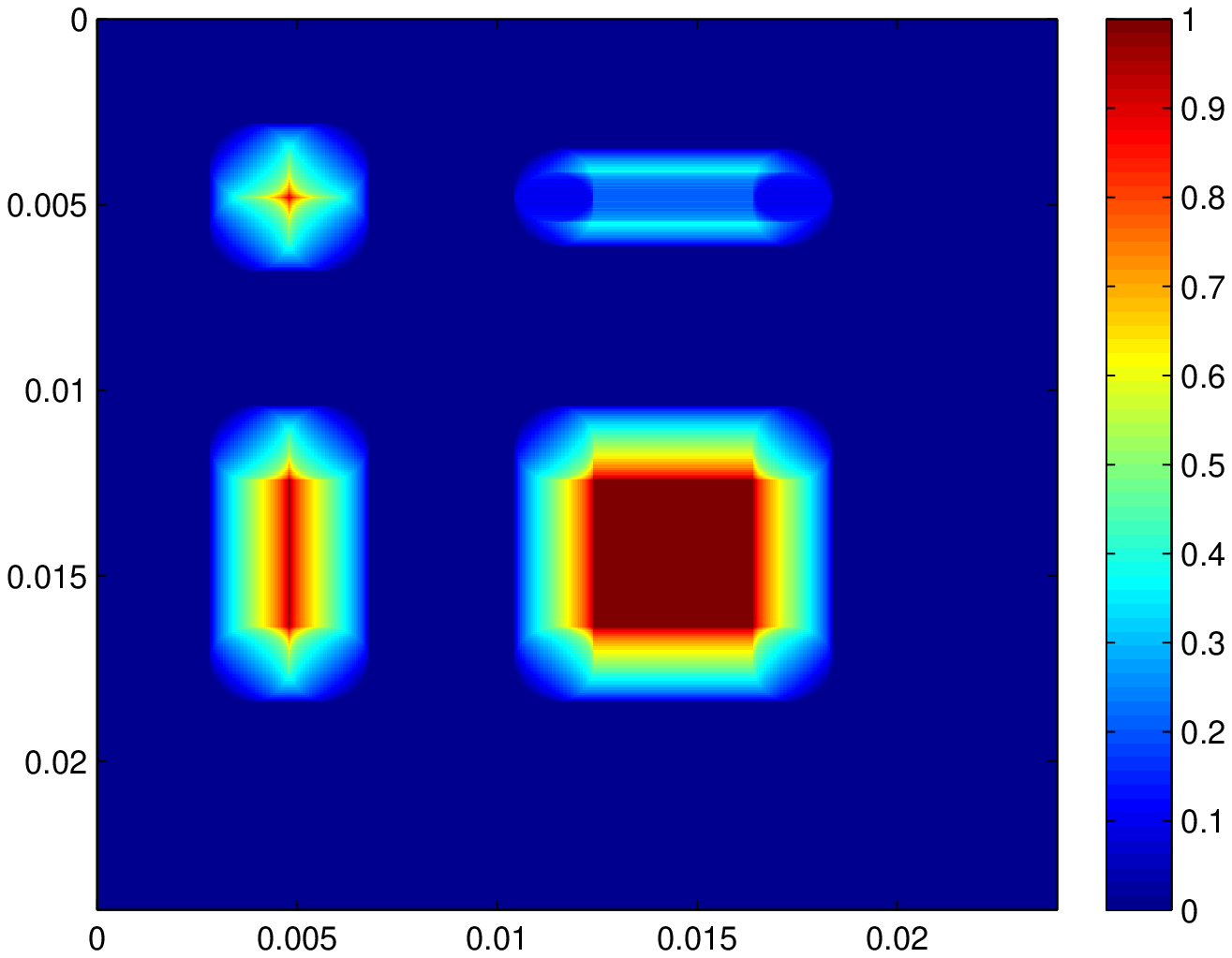}
\includegraphics[height=5cm,angle=0]{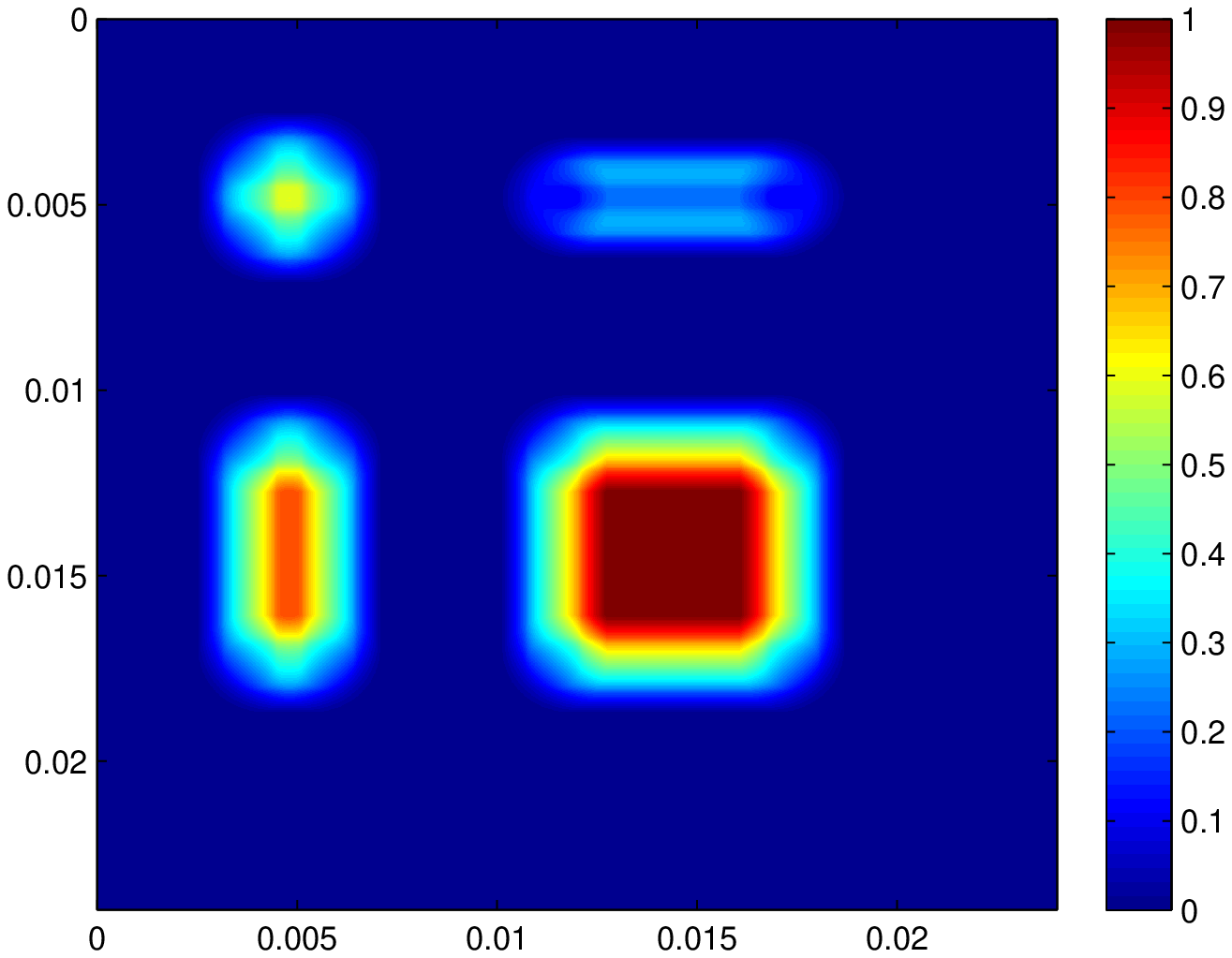}\\
\includegraphics[height=5cm,angle=0]{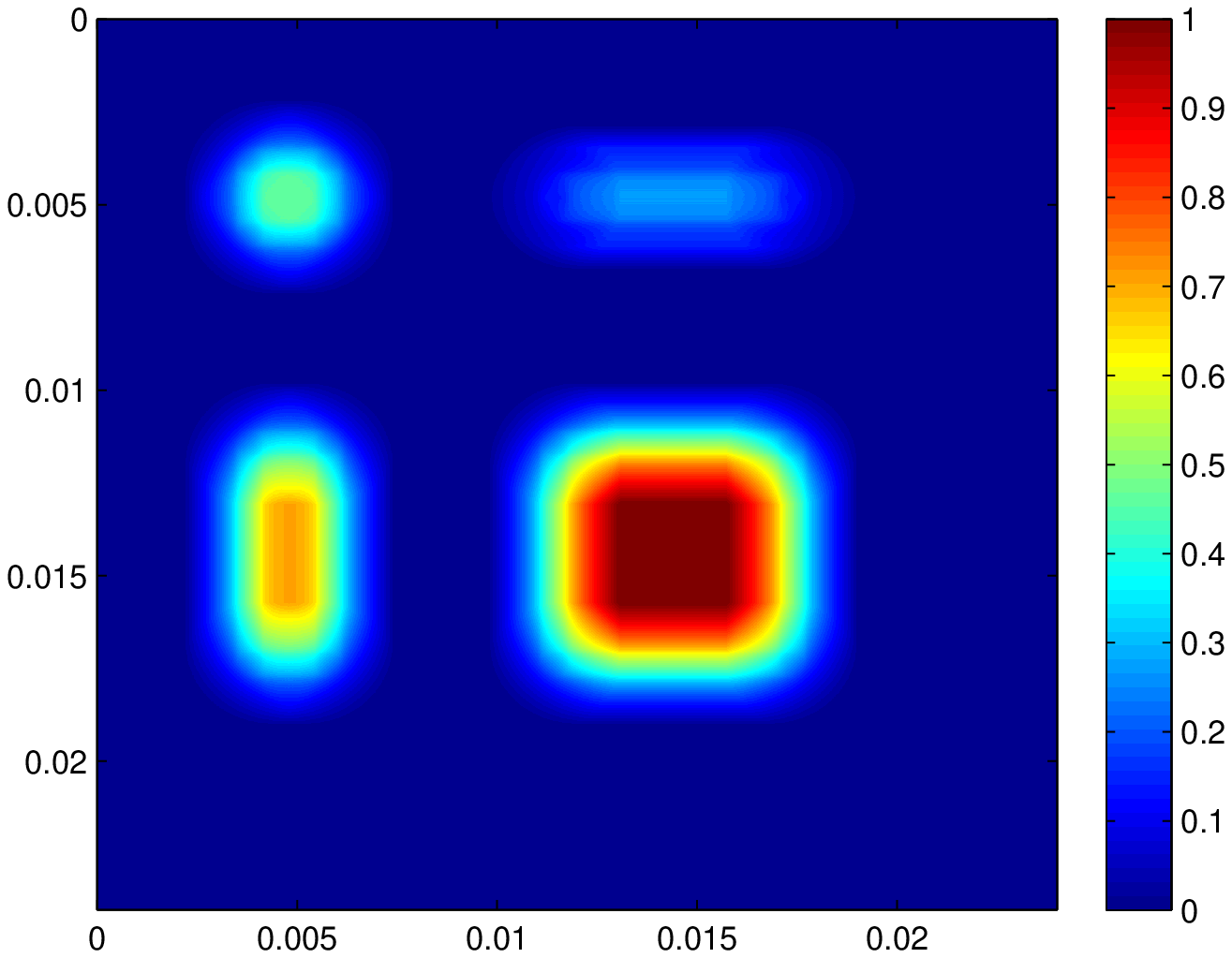}
\includegraphics[height=5cm,angle=0]{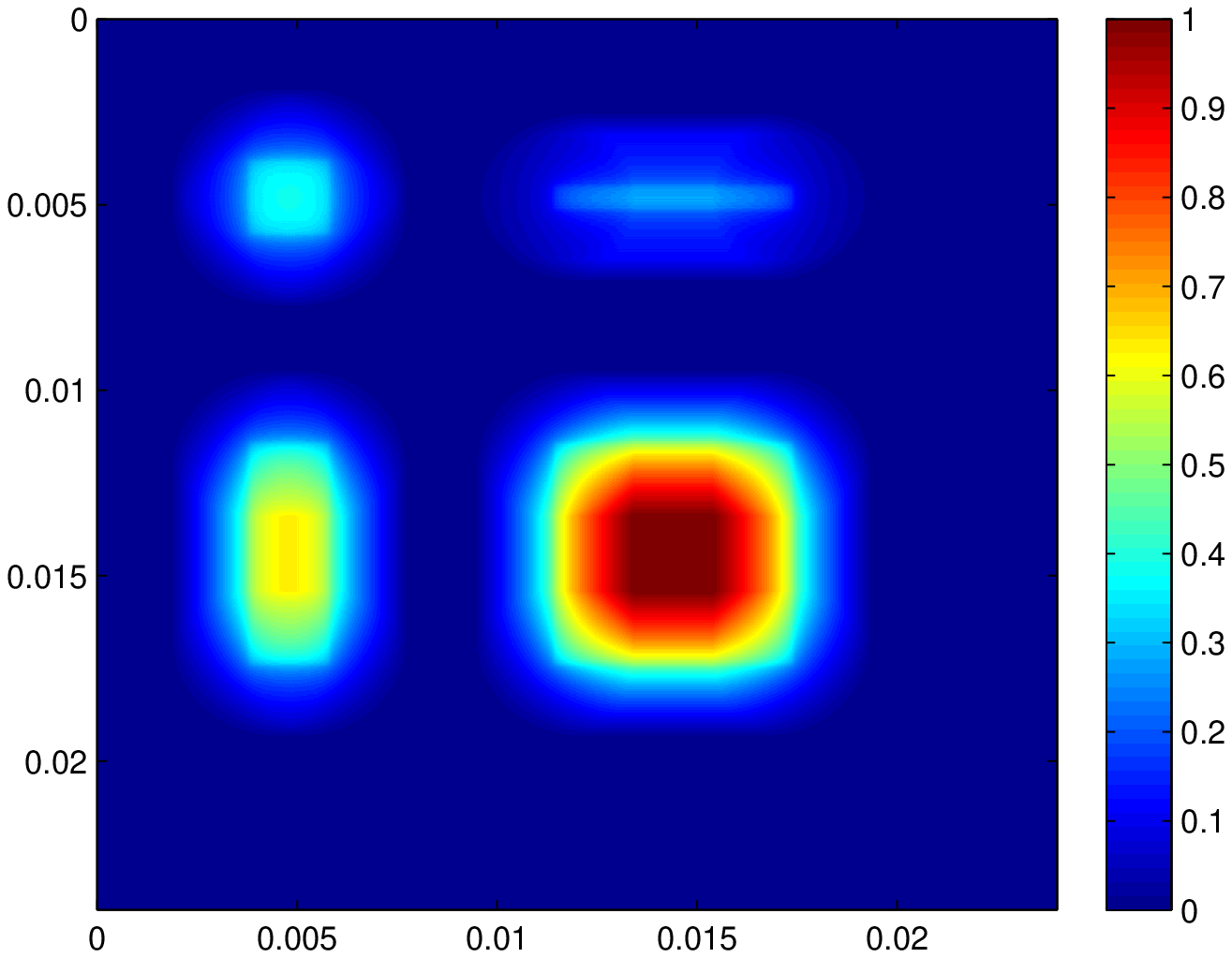}
\end{center}
\caption{Visualization of causal diffusion for the time sequence 
         $(\frac{\tau}{3},\,\frac{2\,\tau}{3},\,\tau,\,\ldots,\,2\,\tau)$ with 
         discretization $\Delta x=9.6\cdot 10^{-6}$ and $\Delta t=\tau=R/c$.  
         The initial distribution is shown in Fig.~\ref{fig:ex01}. 
}
\label{fig:ex01b}
\end{figure}

% --- 2D scale space 2) ---
\begin{figure}[!ht]
\begin{center}
\includegraphics[height=5cm,angle=0]{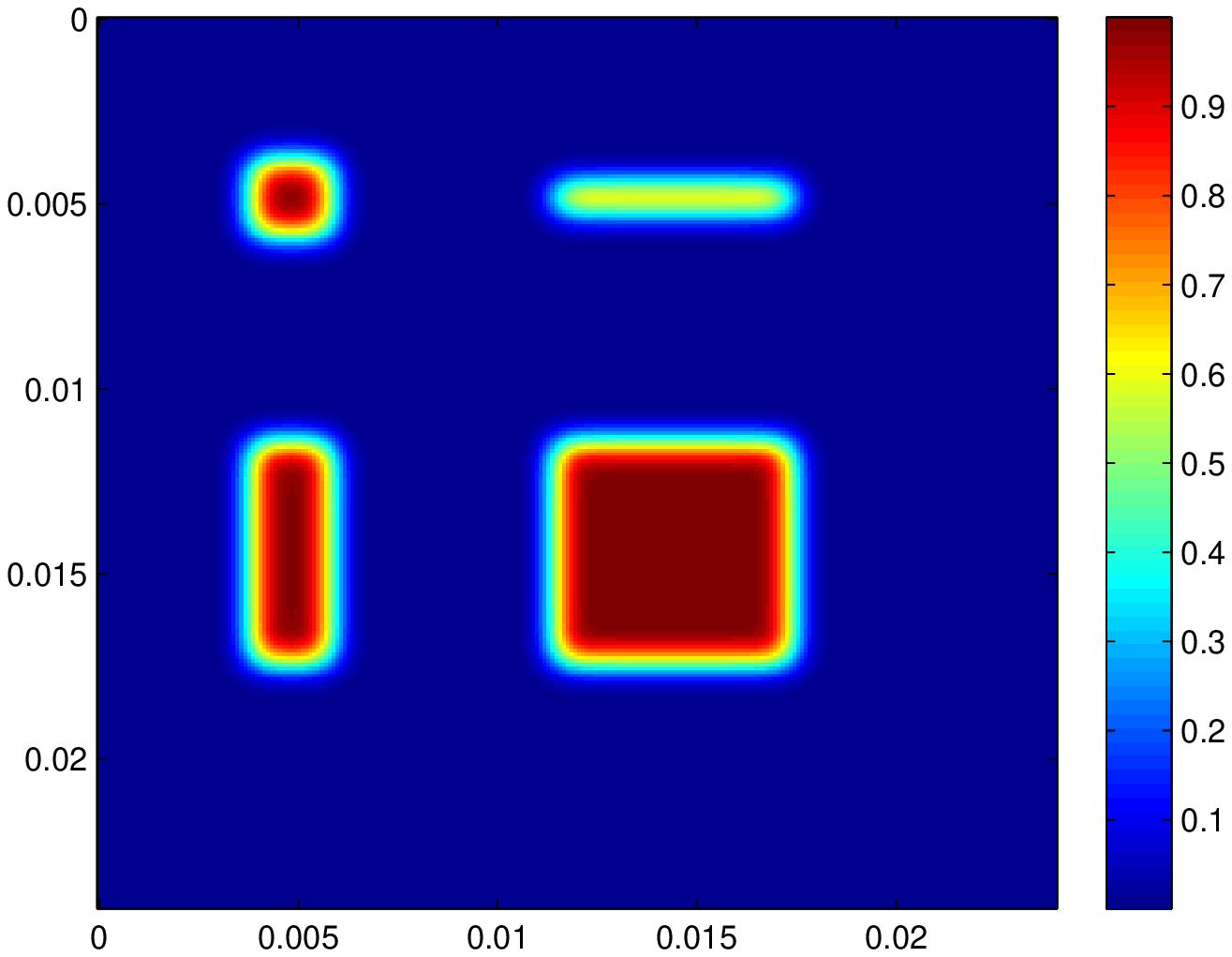}
\includegraphics[height=5cm,angle=0]{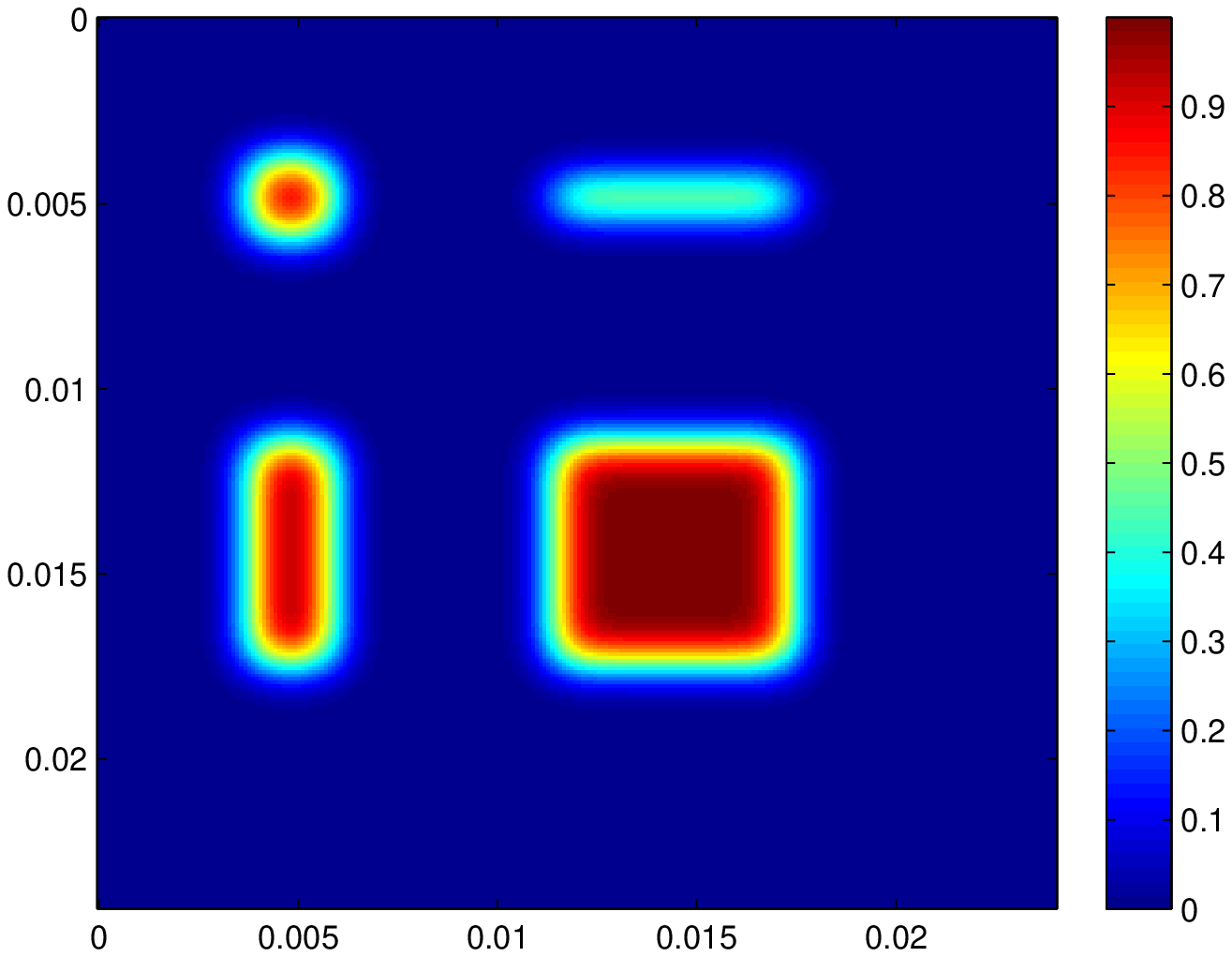}\\
\includegraphics[height=5cm,angle=0]{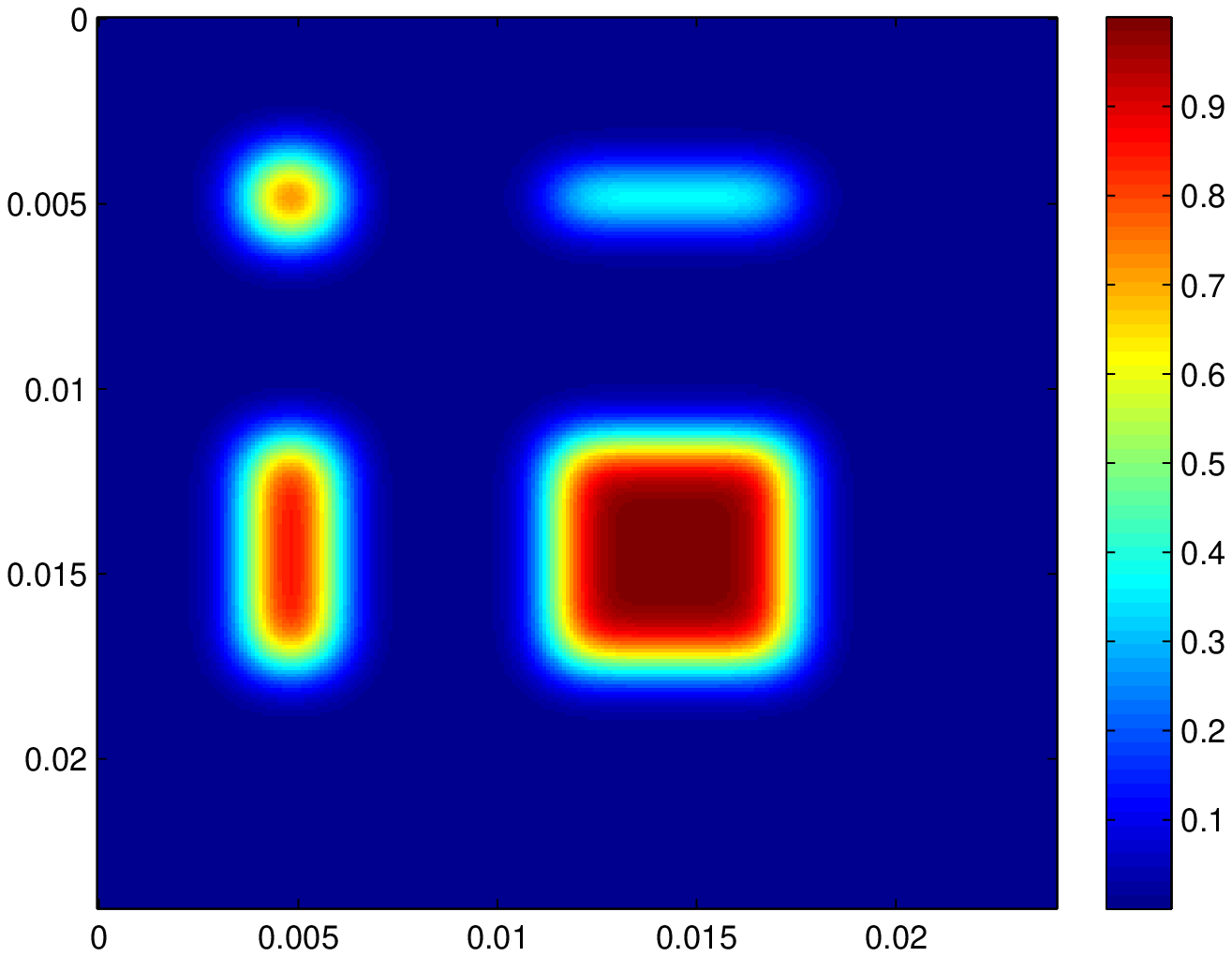}
\includegraphics[height=5cm,angle=0]{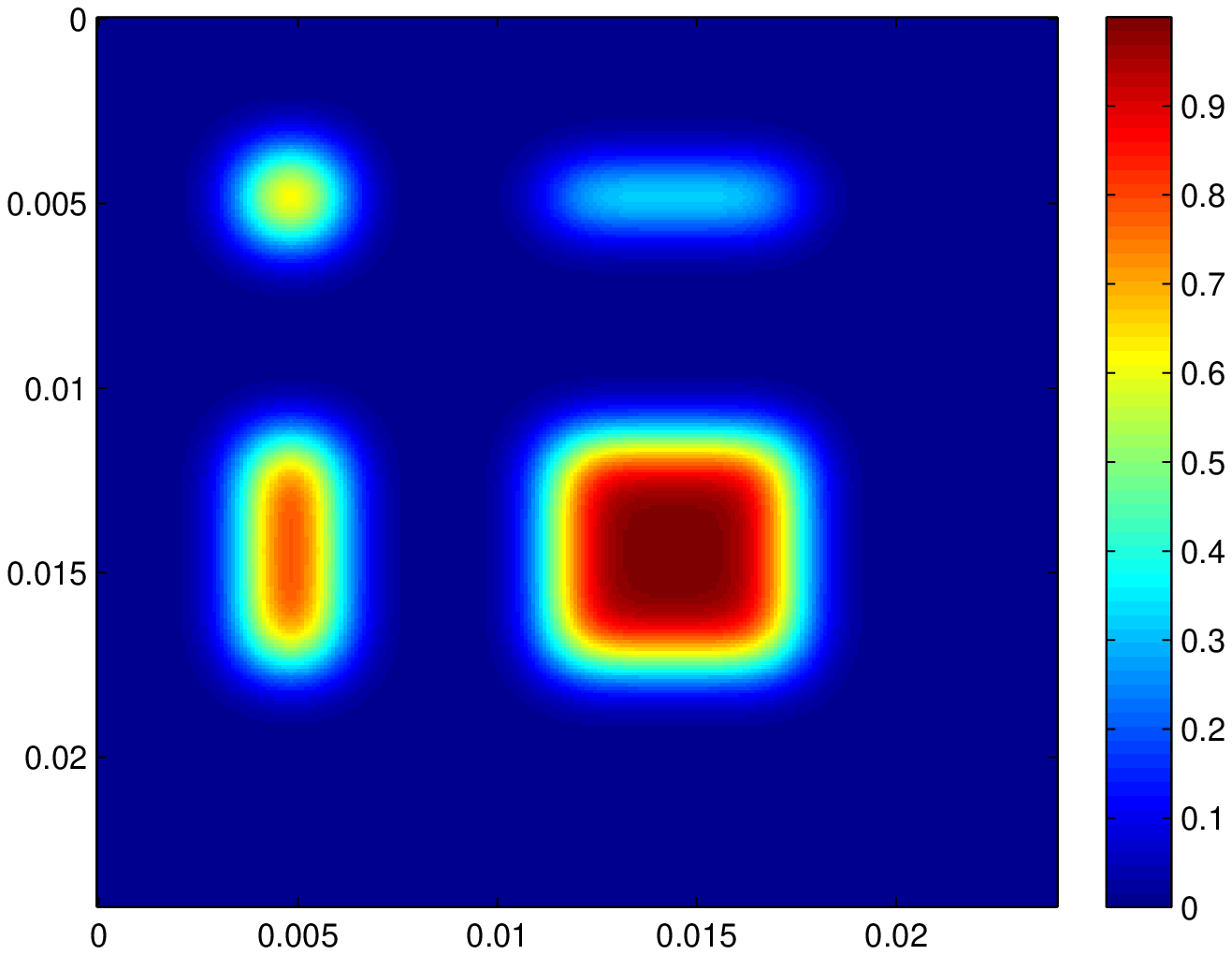}\\
\includegraphics[height=5cm,angle=0]{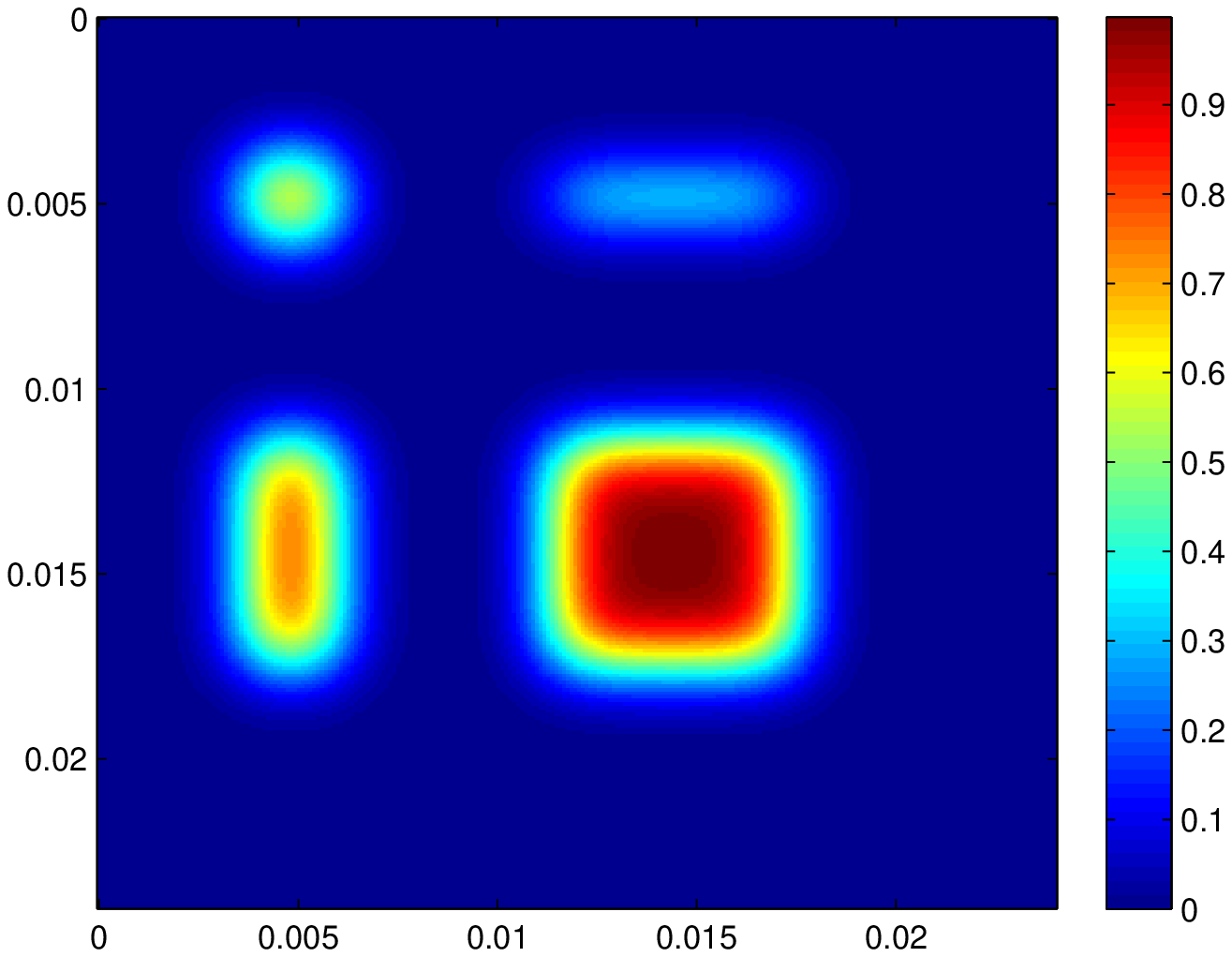}
\includegraphics[height=5cm,angle=0]{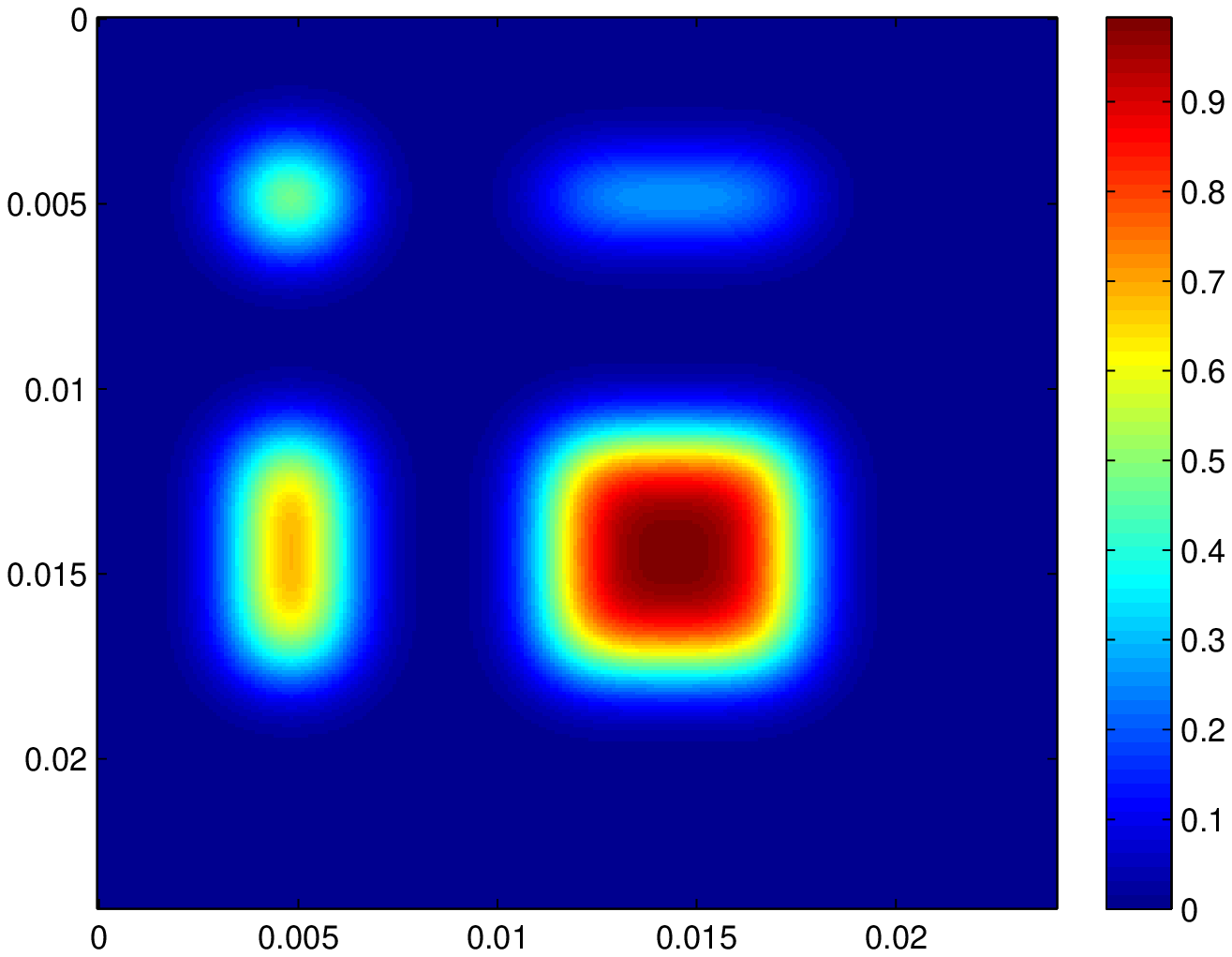} 
\end{center}
\caption{Visualization of standard diffusion for the time sequence 
         $(\frac{\tau}{3},\,\frac{2\,\tau}{3},\,\tau,\,\ldots,\,2\,\tau)$ with  
         discretization $\Delta x := R/10$ and 
         $\Delta t := \frac{\Delta x^2}{2\,N\,D_0}$. 
         The initial distribution is shown in Fig.~\ref{fig:ex01}. 
}
\label{fig:ex01c}
\end{figure}

% --- Inverse pr. 1) ---
\begin{figure}[!ht]
\begin{center}
\includegraphics[height=5cm,angle=0]{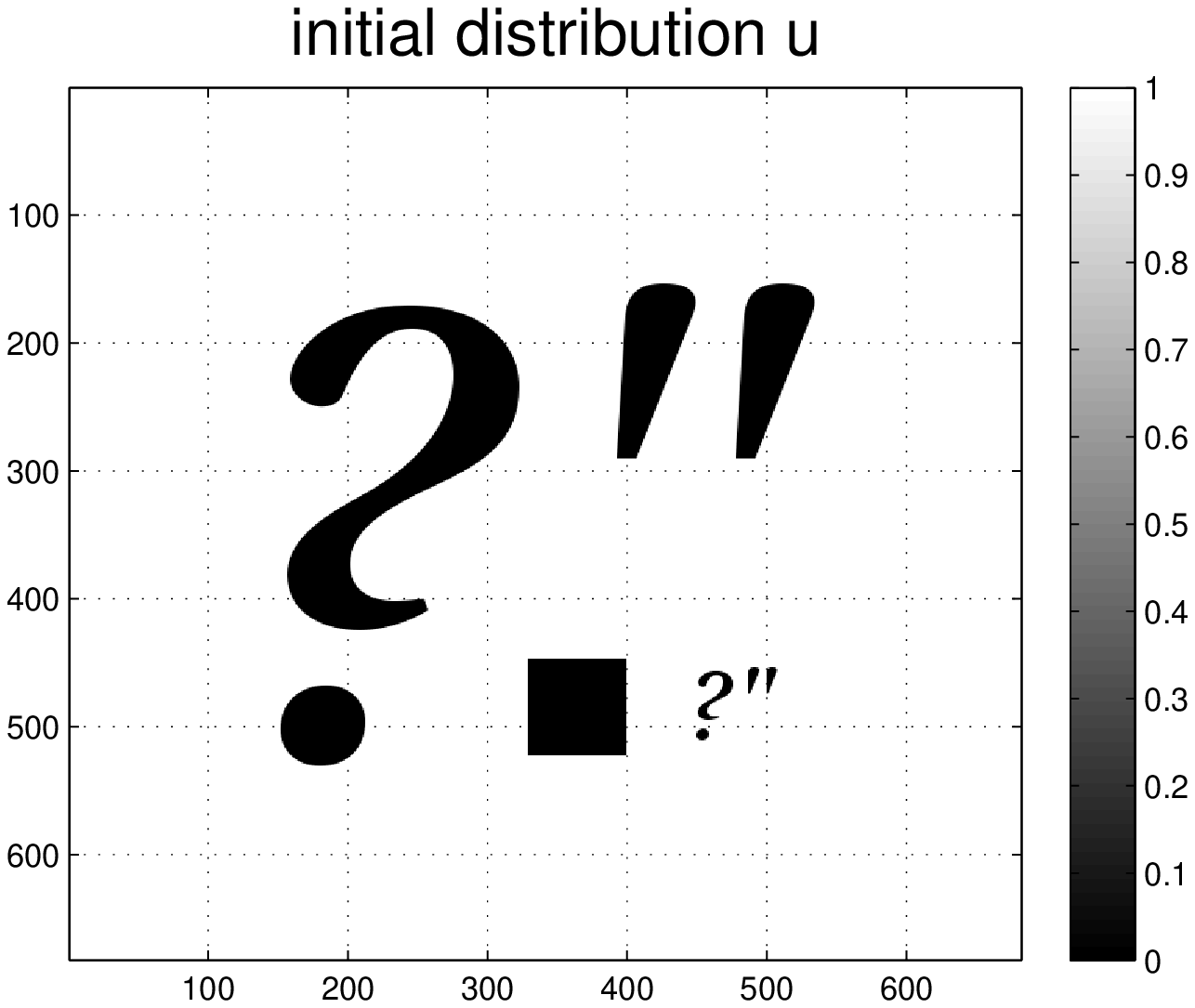}
\includegraphics[height=5cm,angle=0]{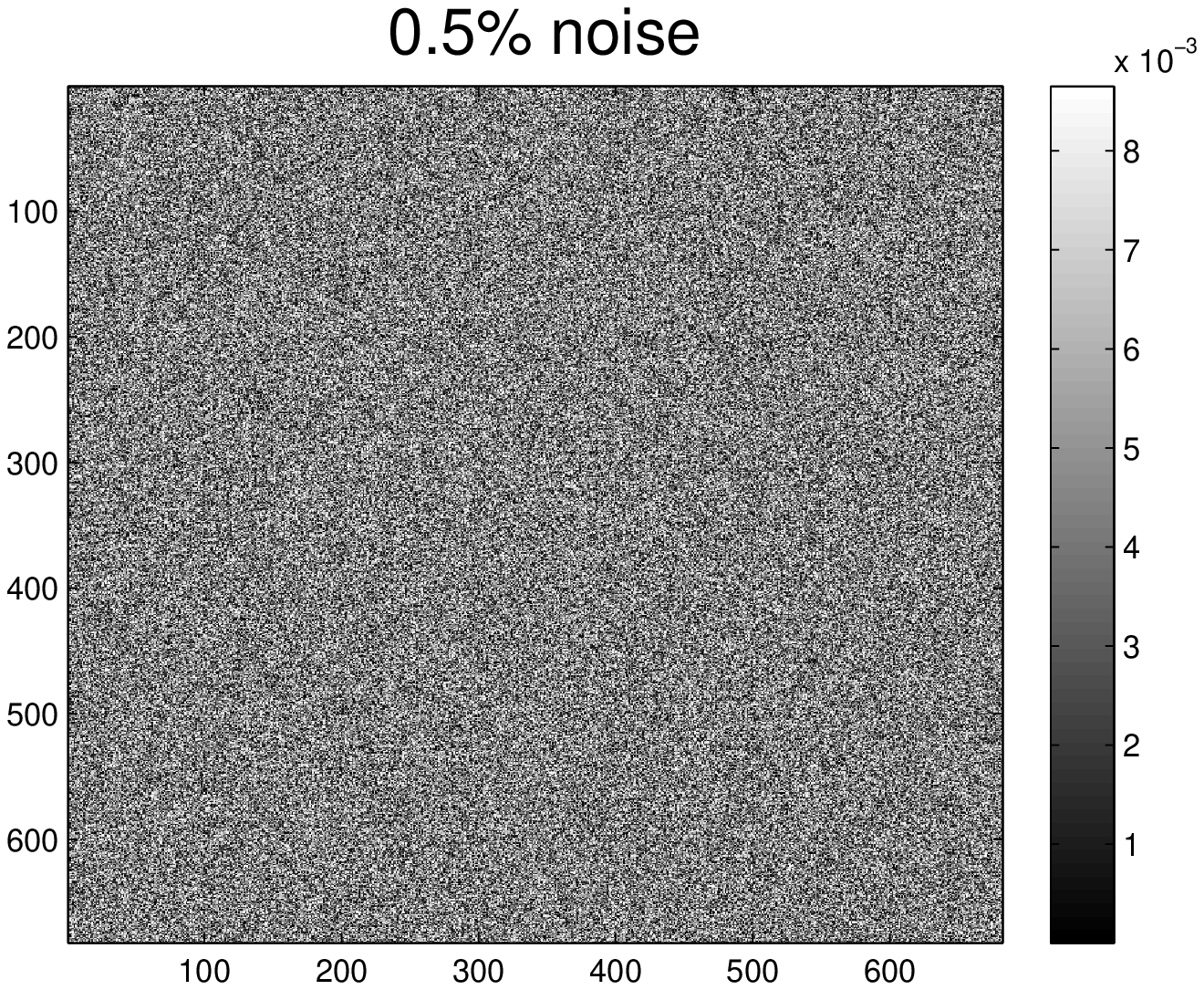}\\
\includegraphics[height=5cm,angle=0]{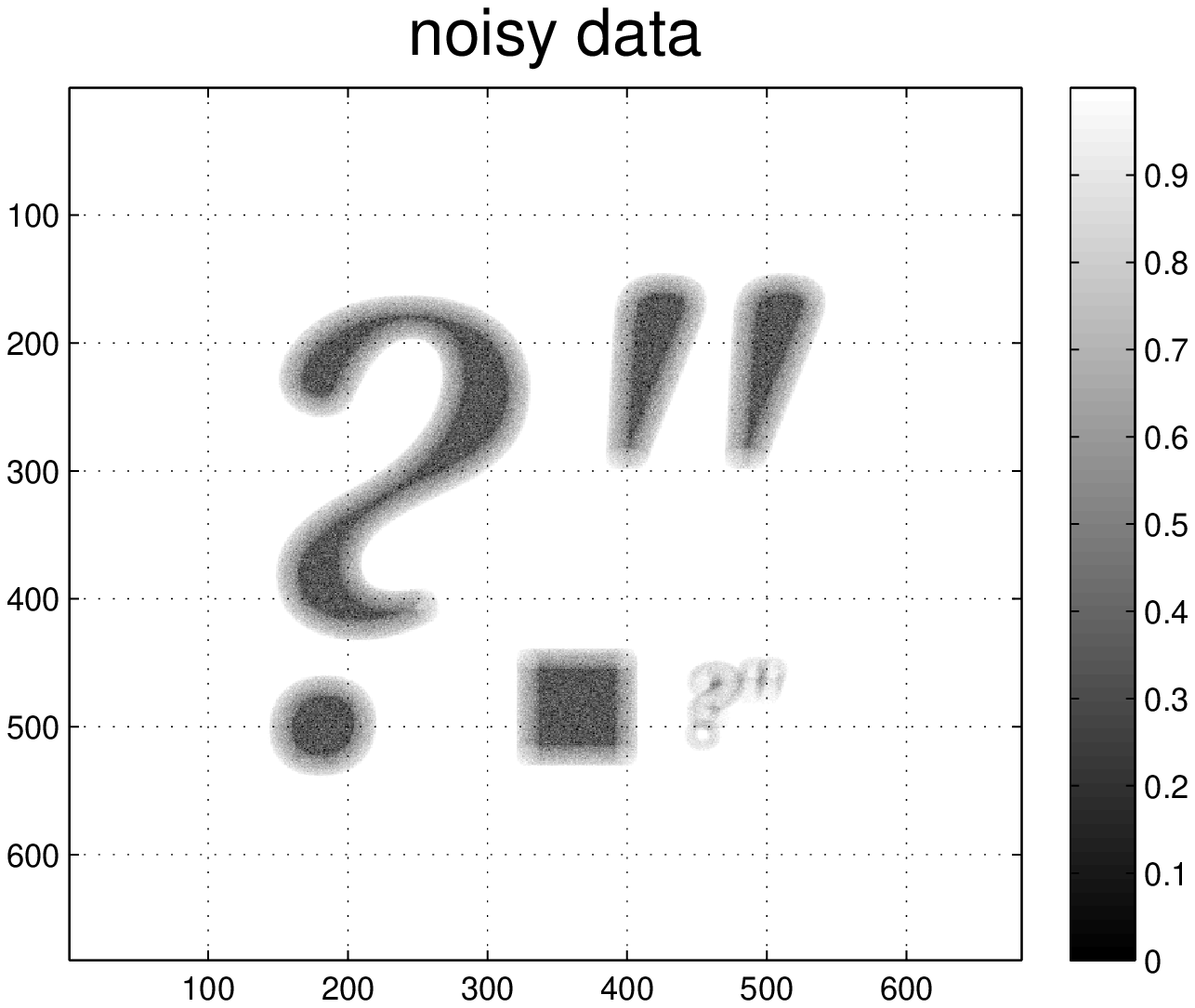}
\includegraphics[height=5cm,angle=0]{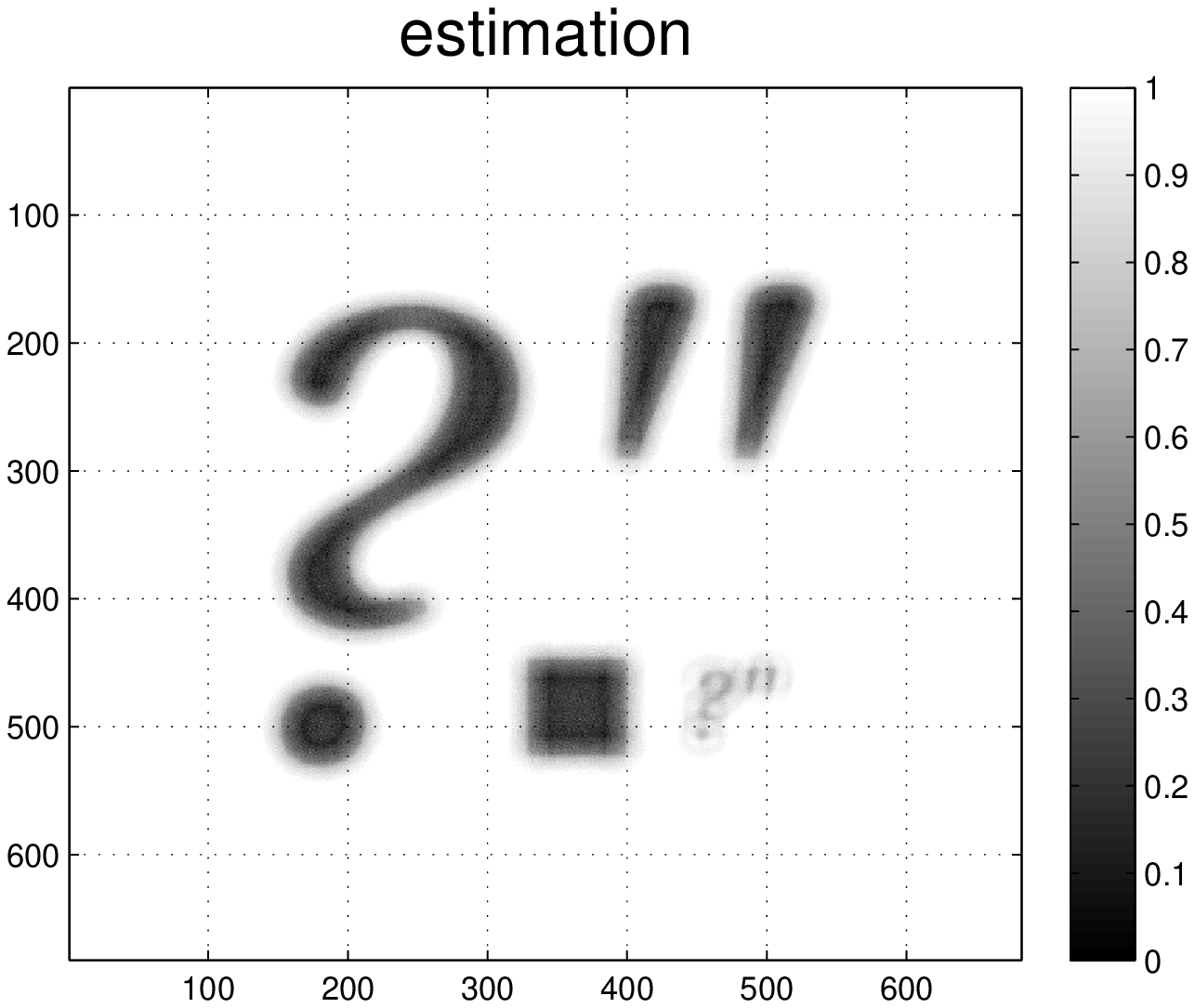}\\
\includegraphics[height=5cm,angle=0]{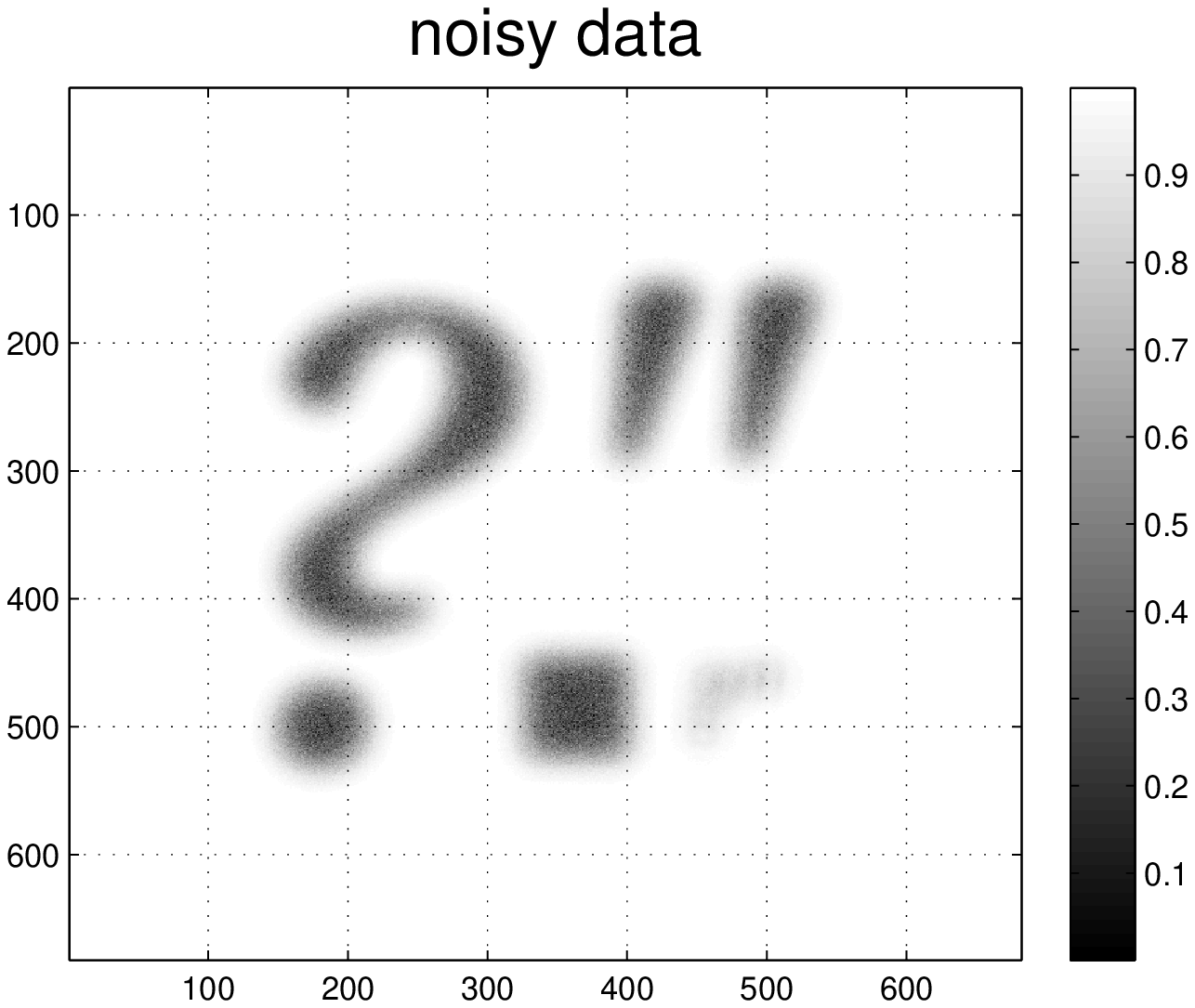}
\includegraphics[height=5cm,angle=0]{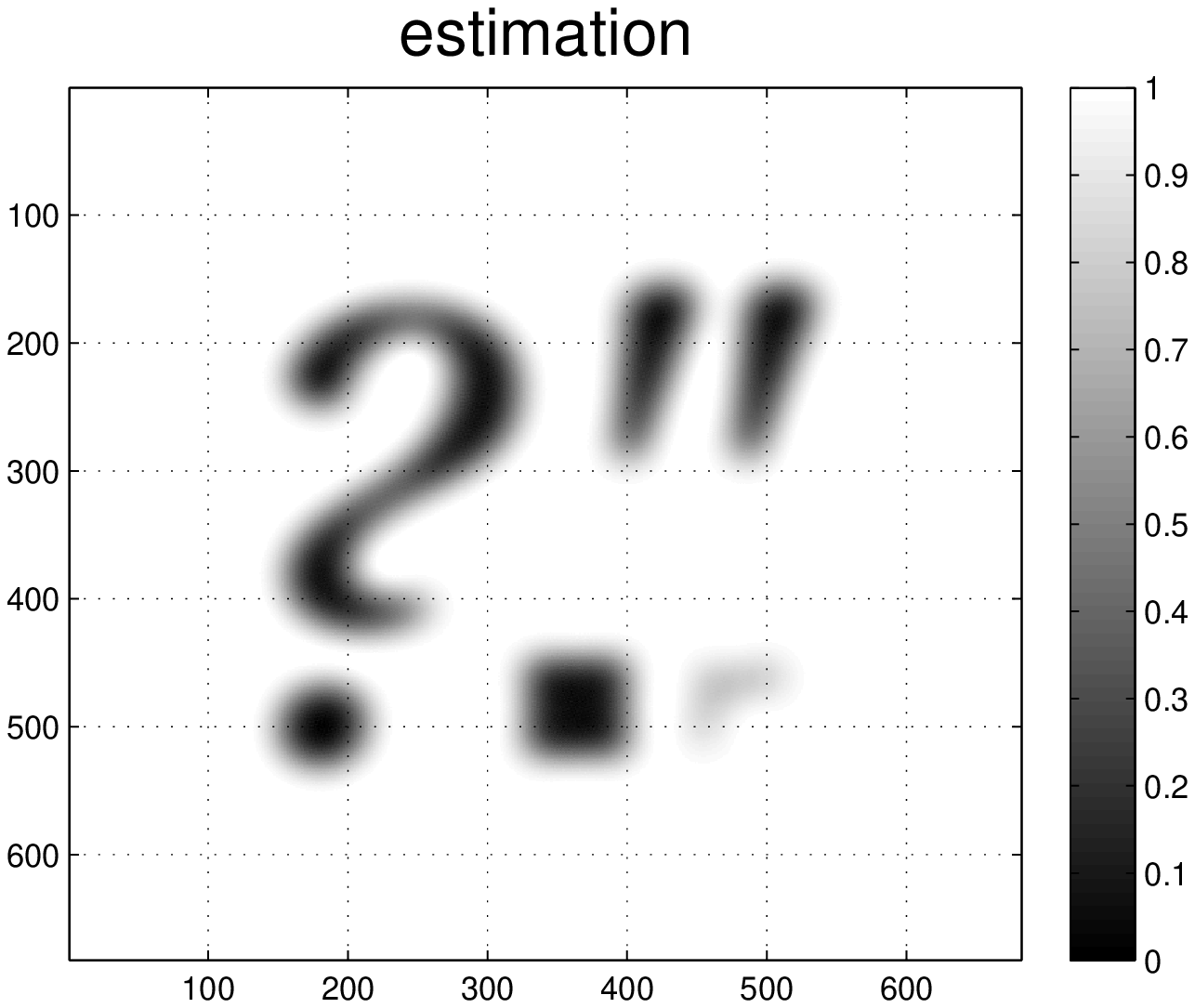}
\end{center}
\caption{Numerical solution of the inverse problem with $0.5\%$ uniformly distributed 
         $L^2$-noise for data acquisition time $T=\tau$ and $T=3\,\tau$. 
         For $T=\tau$ and $T=3\,\tau$, the Discrepancy principle stops optimally 
         for $\eta=9.4$ ($6$ steps) and $\eta=5.9$ ($6$ steps), respectively. 
}
\label{fig:6}
\end{figure}

% --- Inverse pr. 2) ---
\begin{figure}[!ht]
\begin{center}
\includegraphics[height=5cm,angle=0]{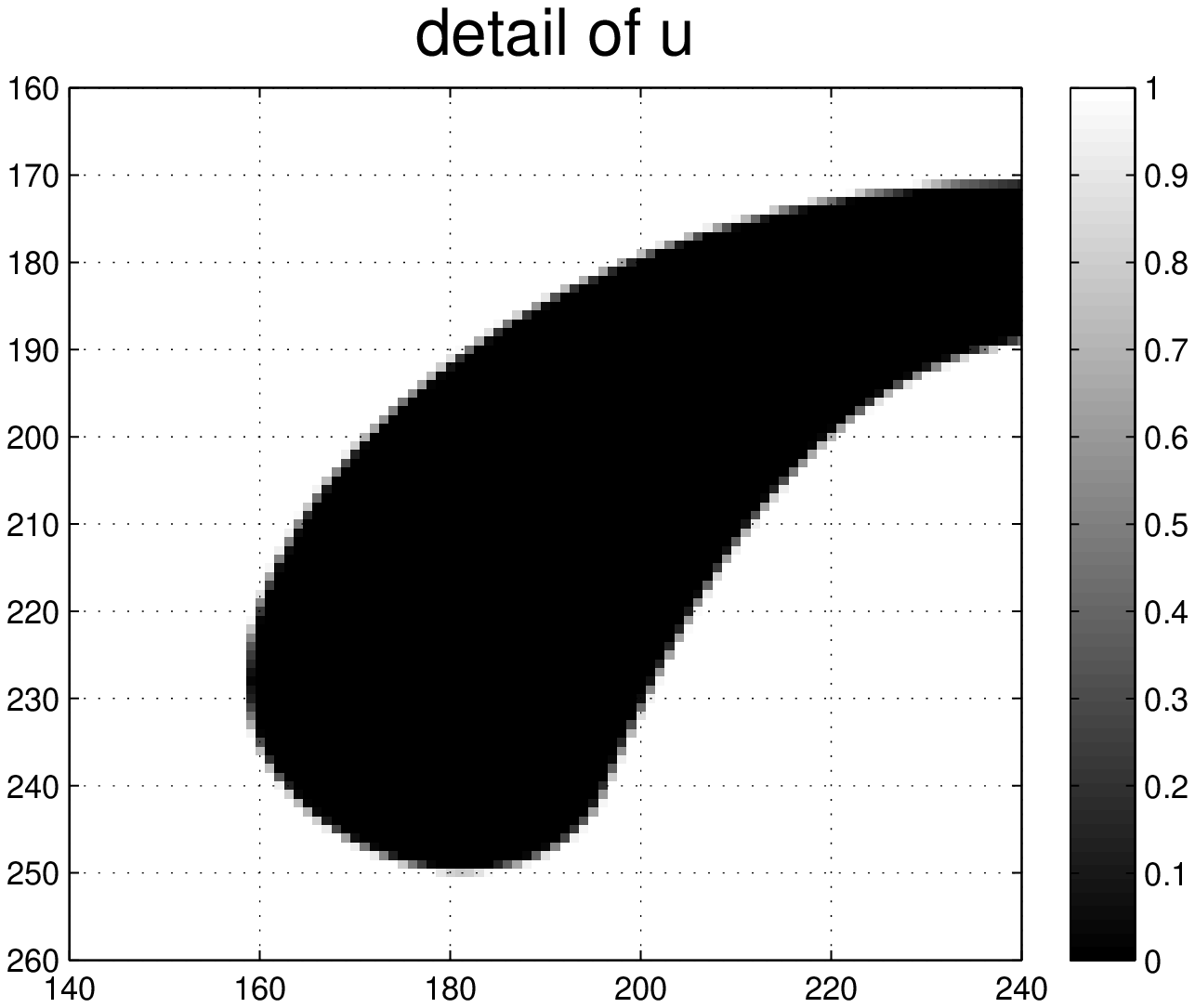}
\includegraphics[height=5cm,angle=0]{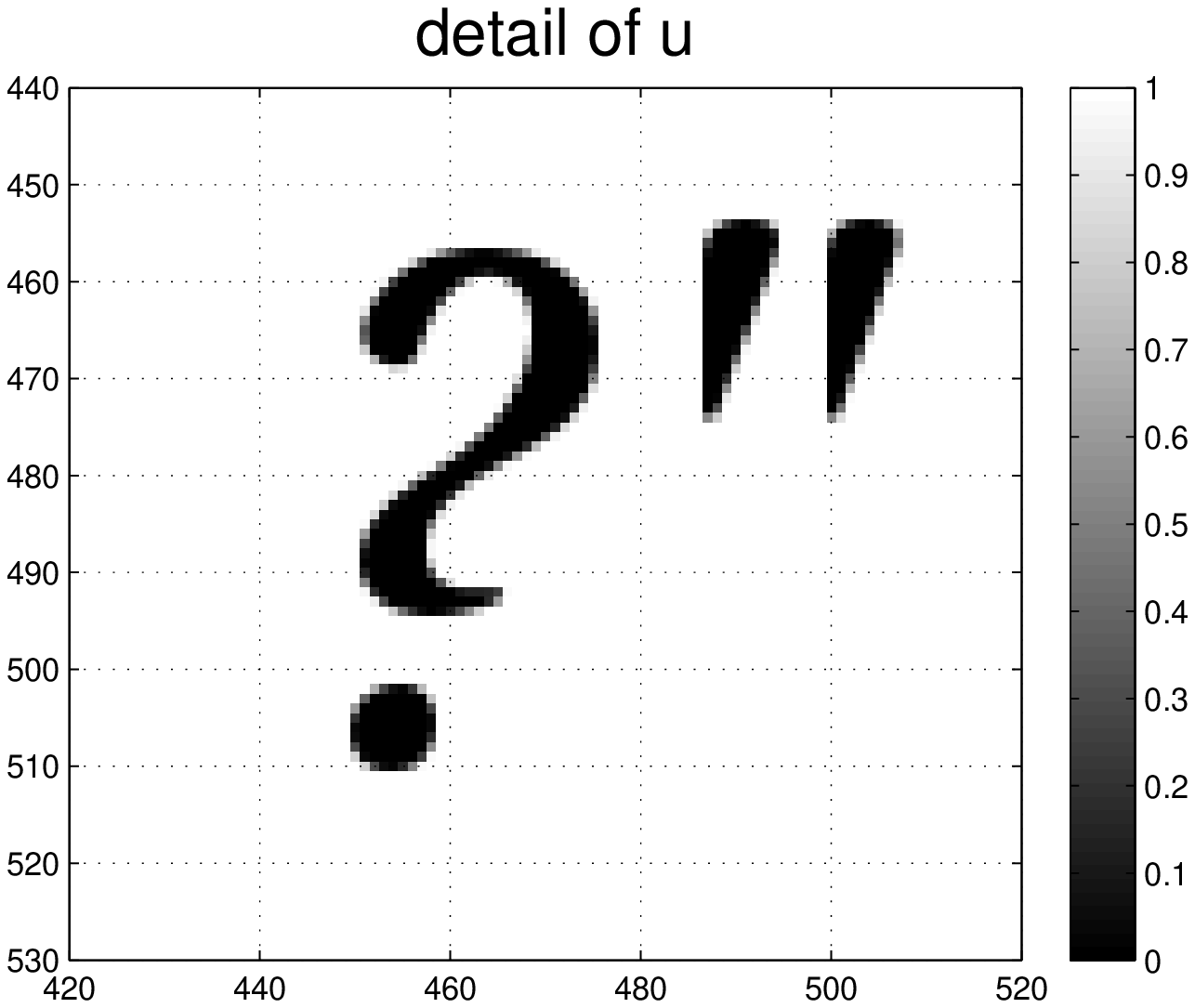}\\
\includegraphics[height=5cm,angle=0]{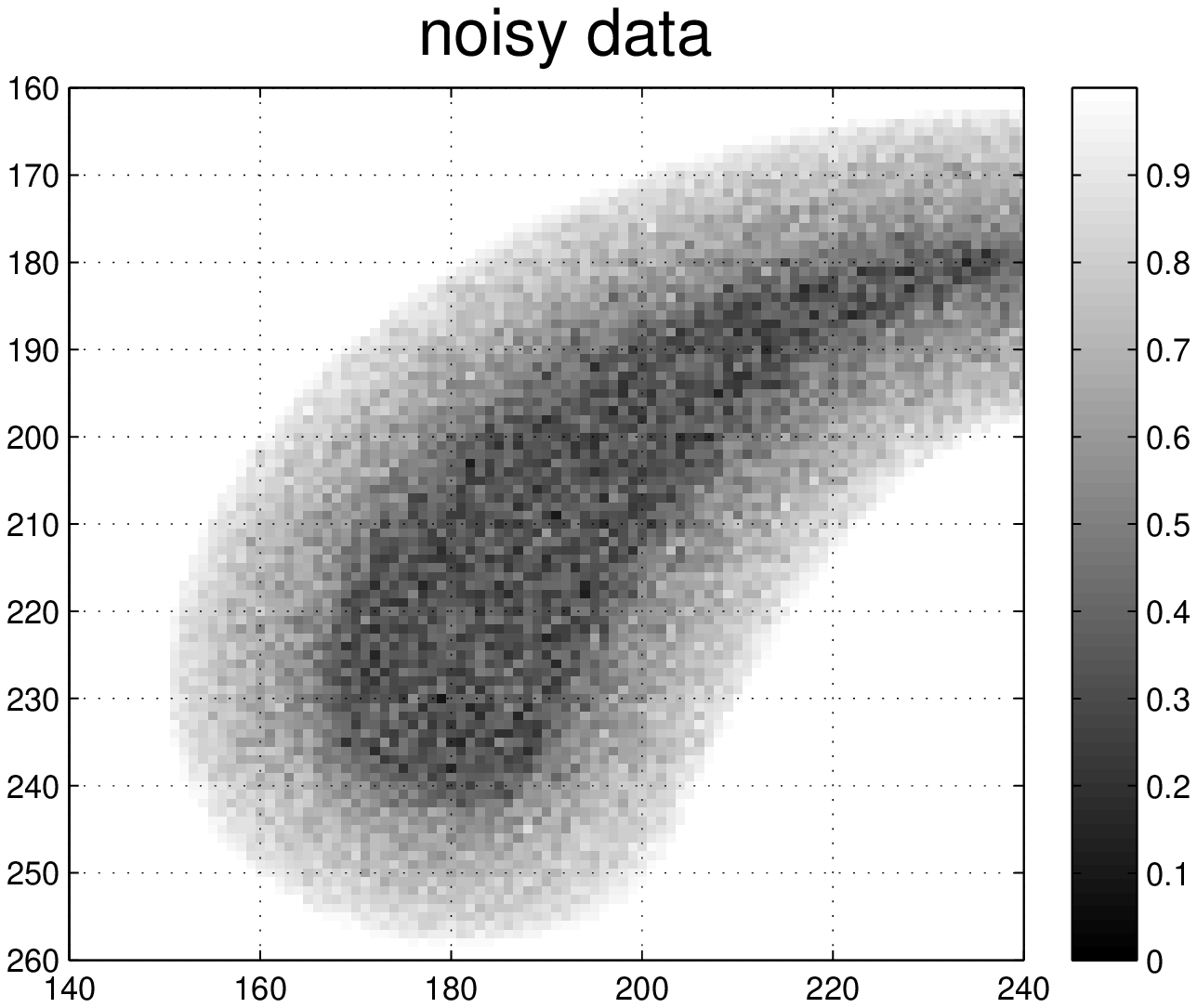}
\includegraphics[height=5cm,angle=0]{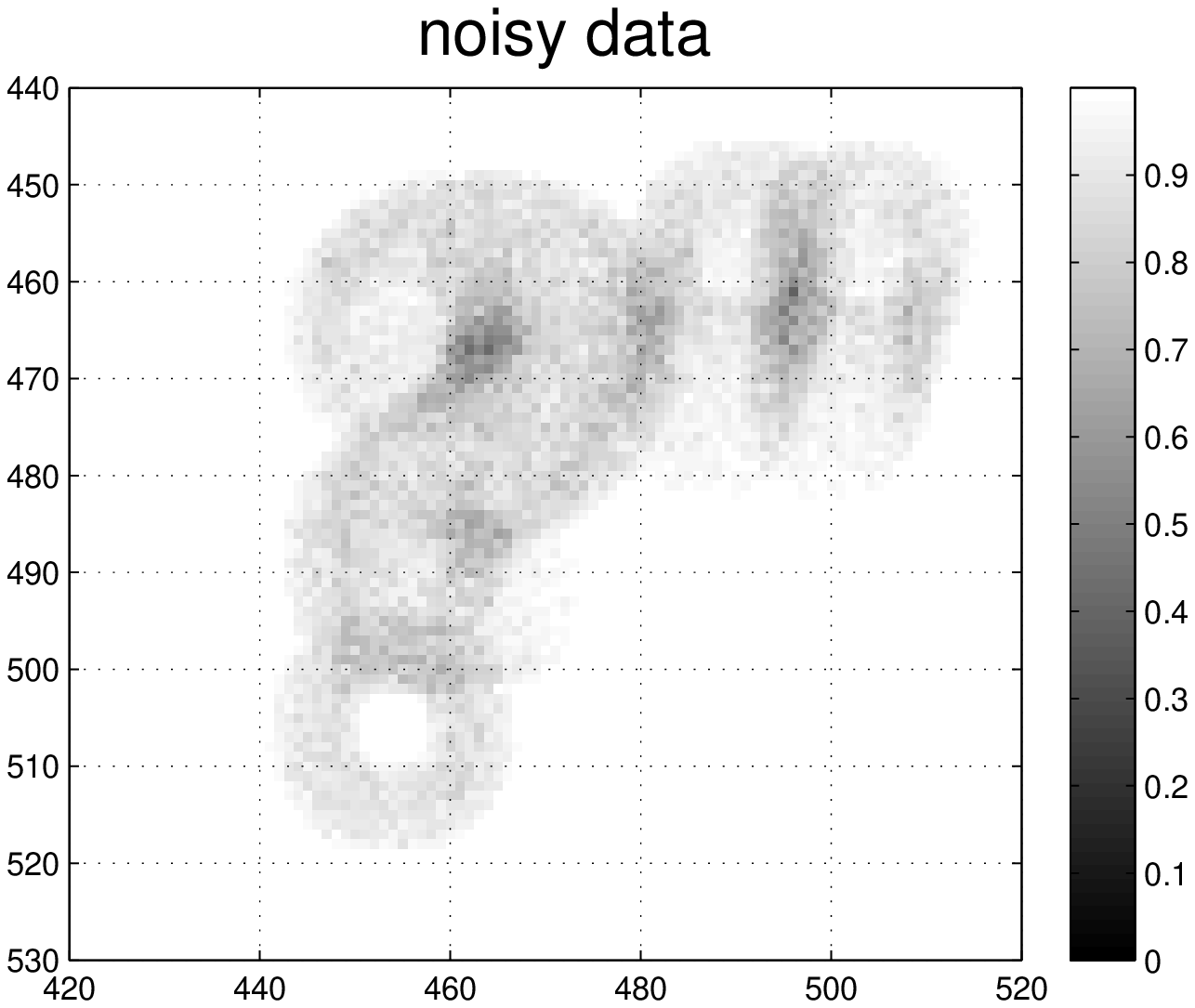}\\
\includegraphics[height=5cm,angle=0]{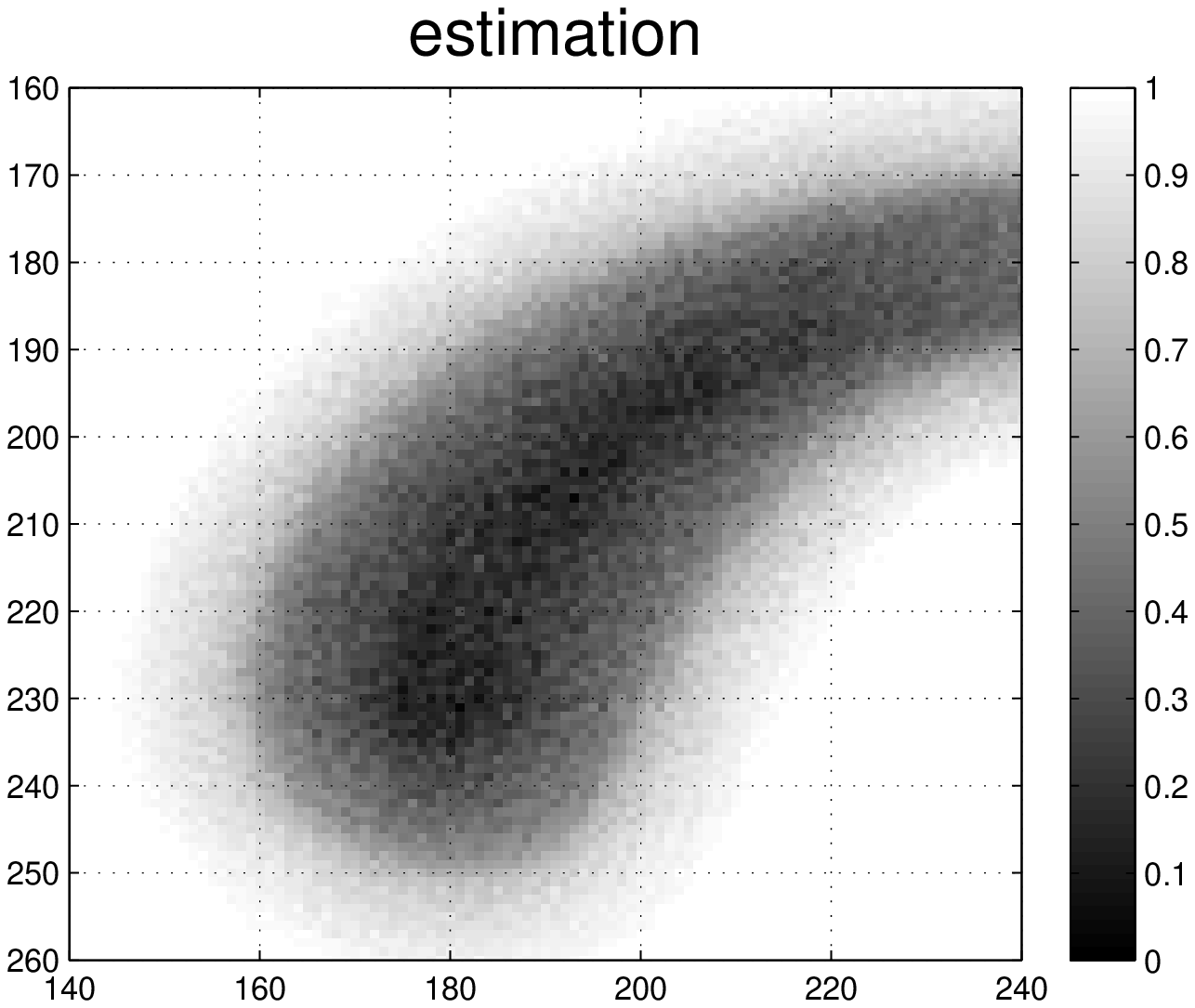}
\includegraphics[height=5cm,angle=0]{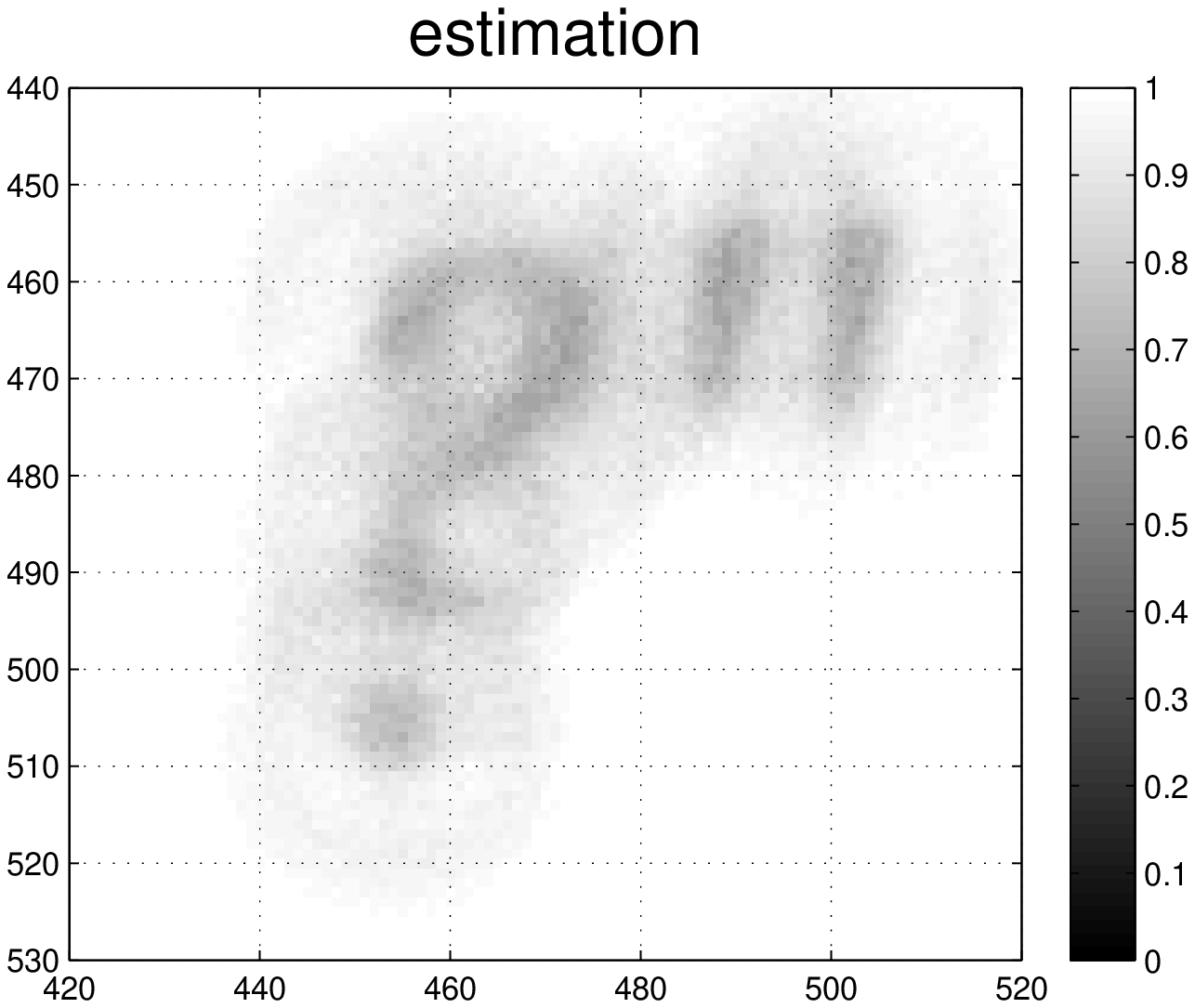} 
\end{center}
\caption{Details to the second row in Fig.~\ref{fig:6} ($T=\tau$). The left 
column shows the left top part of the big question mark and the right column 
shows the small question mark and the small apostrophe.}
\label{fig:7}
\end{figure}

% --- Inverse pr. 4) ---
\begin{figure}[!ht]
\begin{center}
\includegraphics[height=5cm,angle=0]{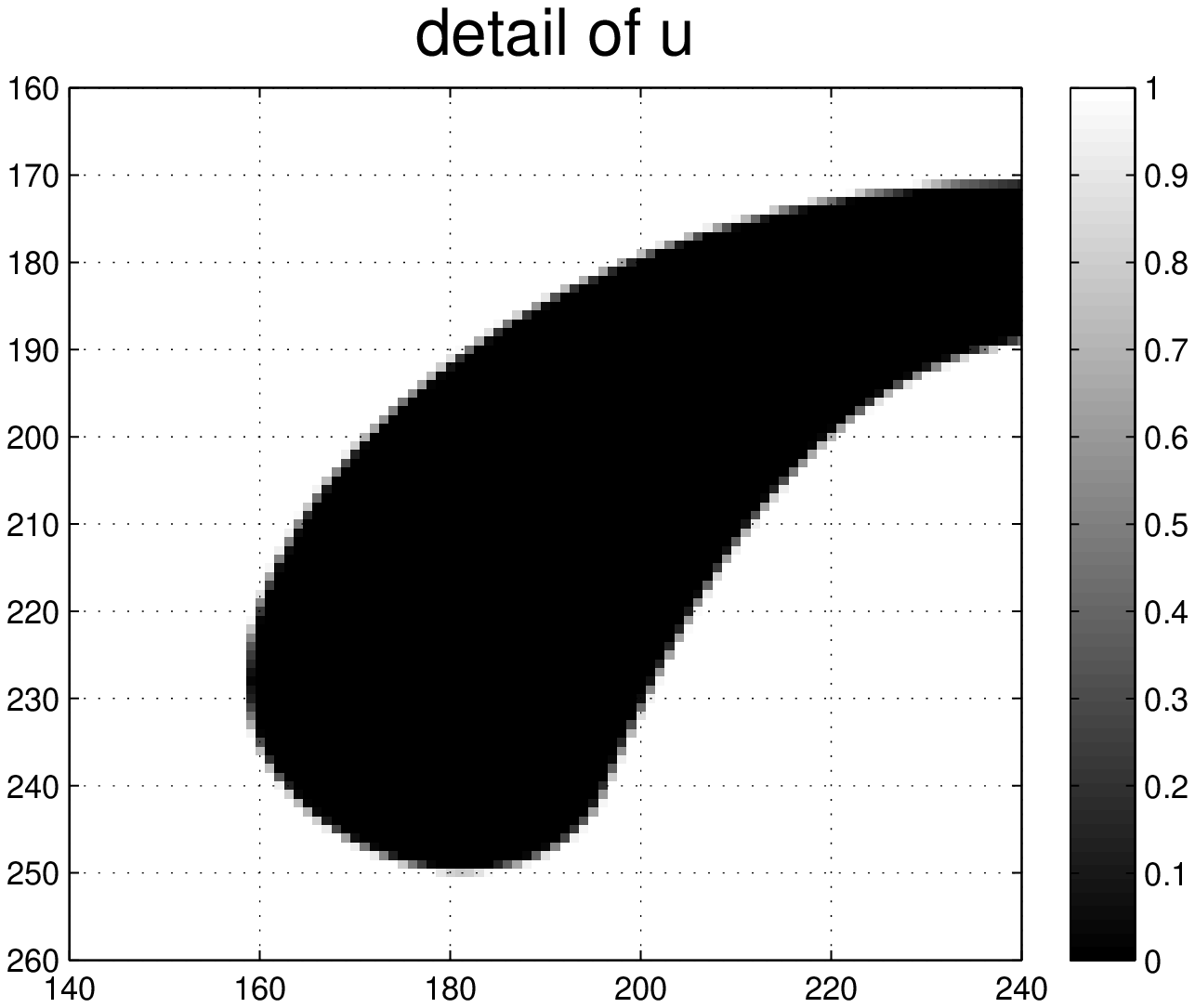}
\includegraphics[height=5cm,angle=0]{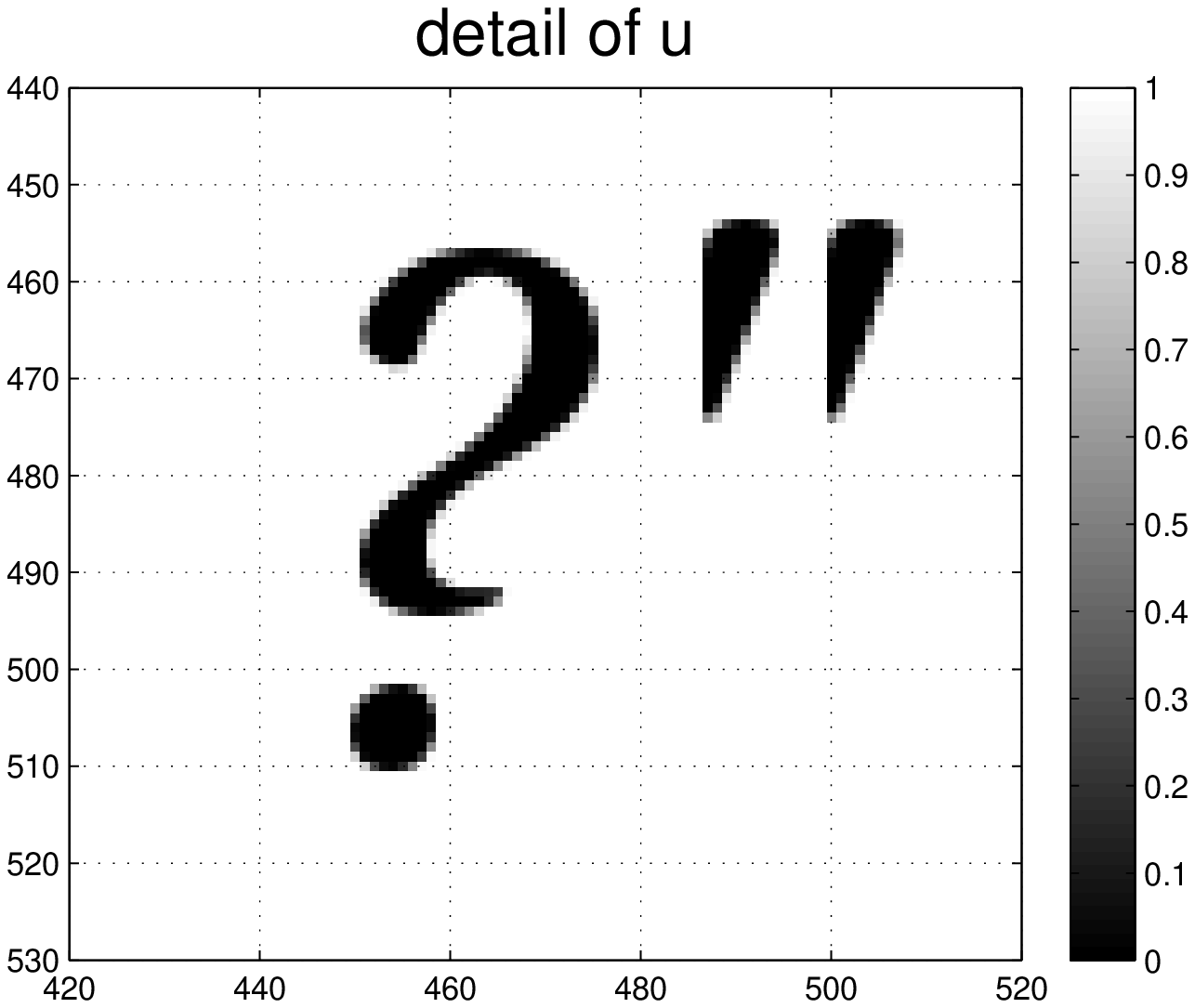}\\
\includegraphics[height=5cm,angle=0]{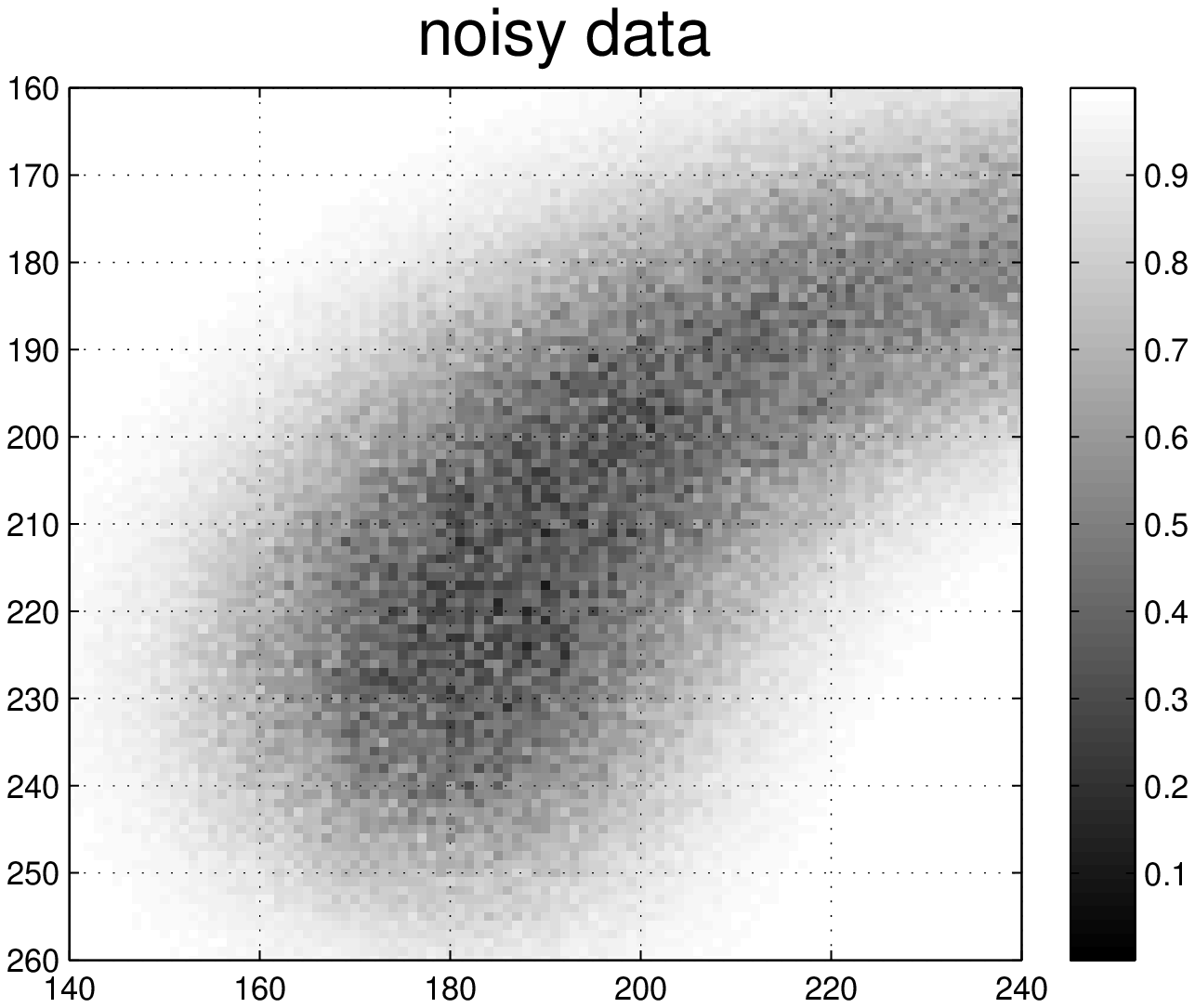}
\includegraphics[height=5cm,angle=0]{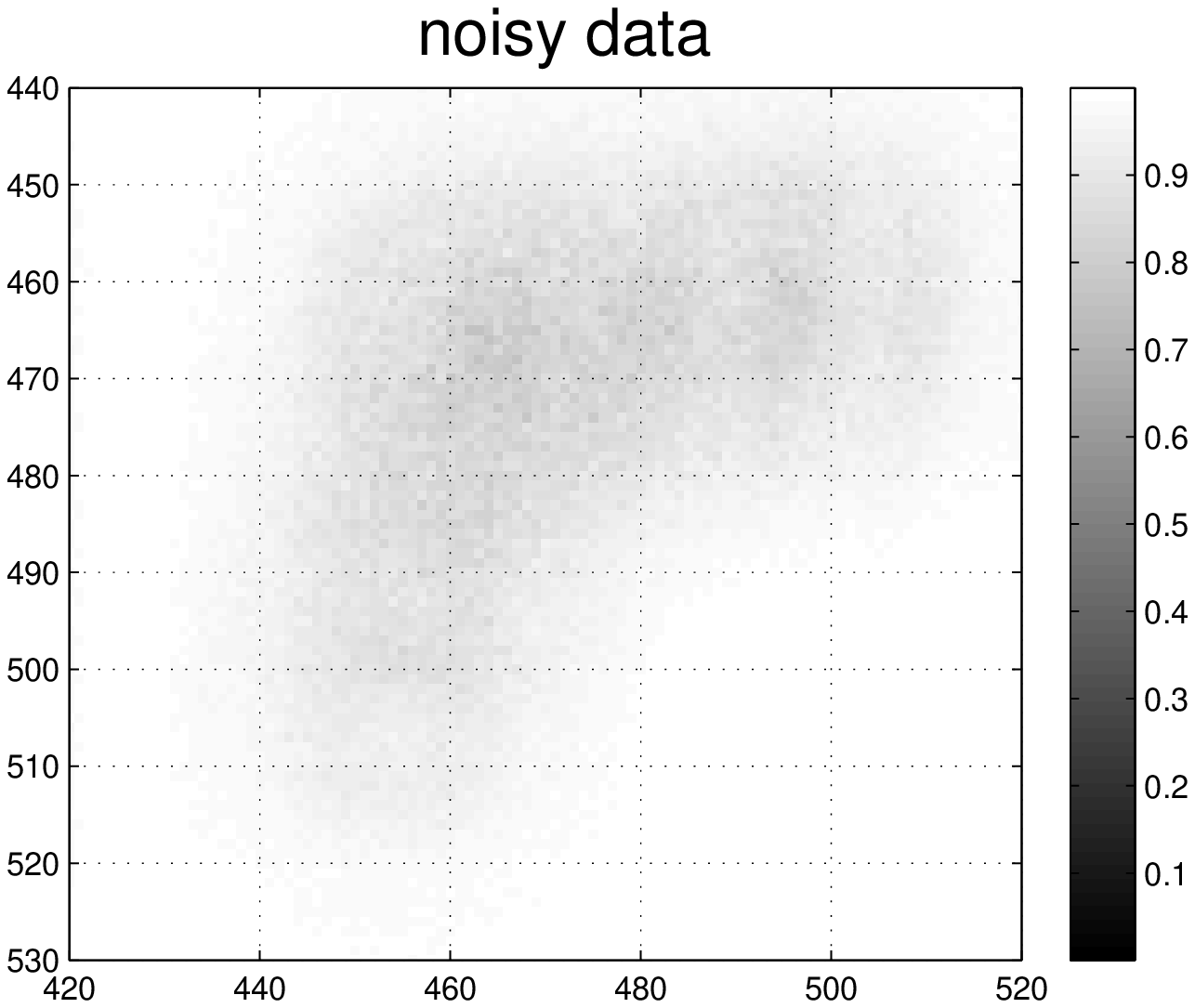}\\
\includegraphics[height=5cm,angle=0]{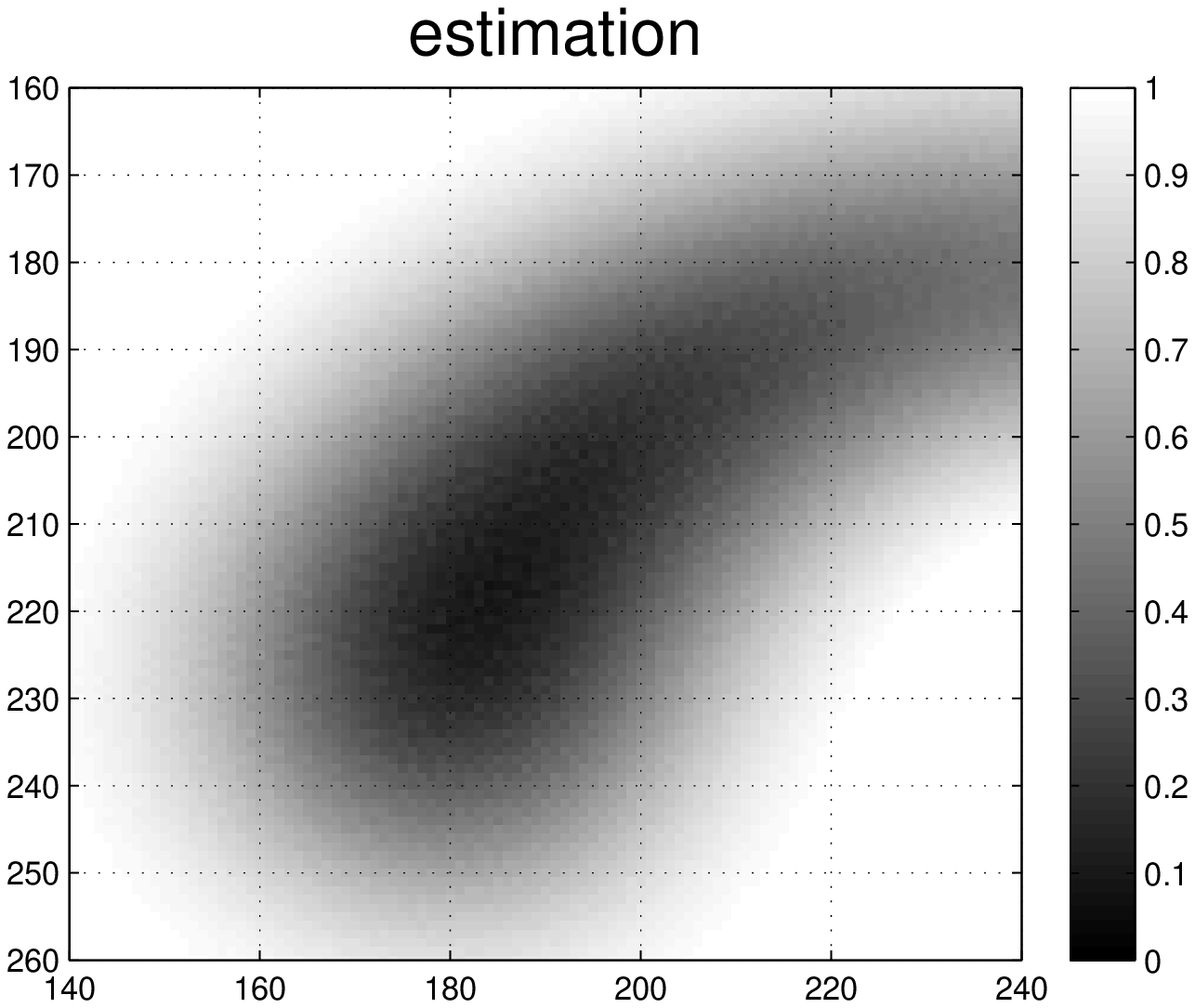}
\includegraphics[height=5cm,angle=0]{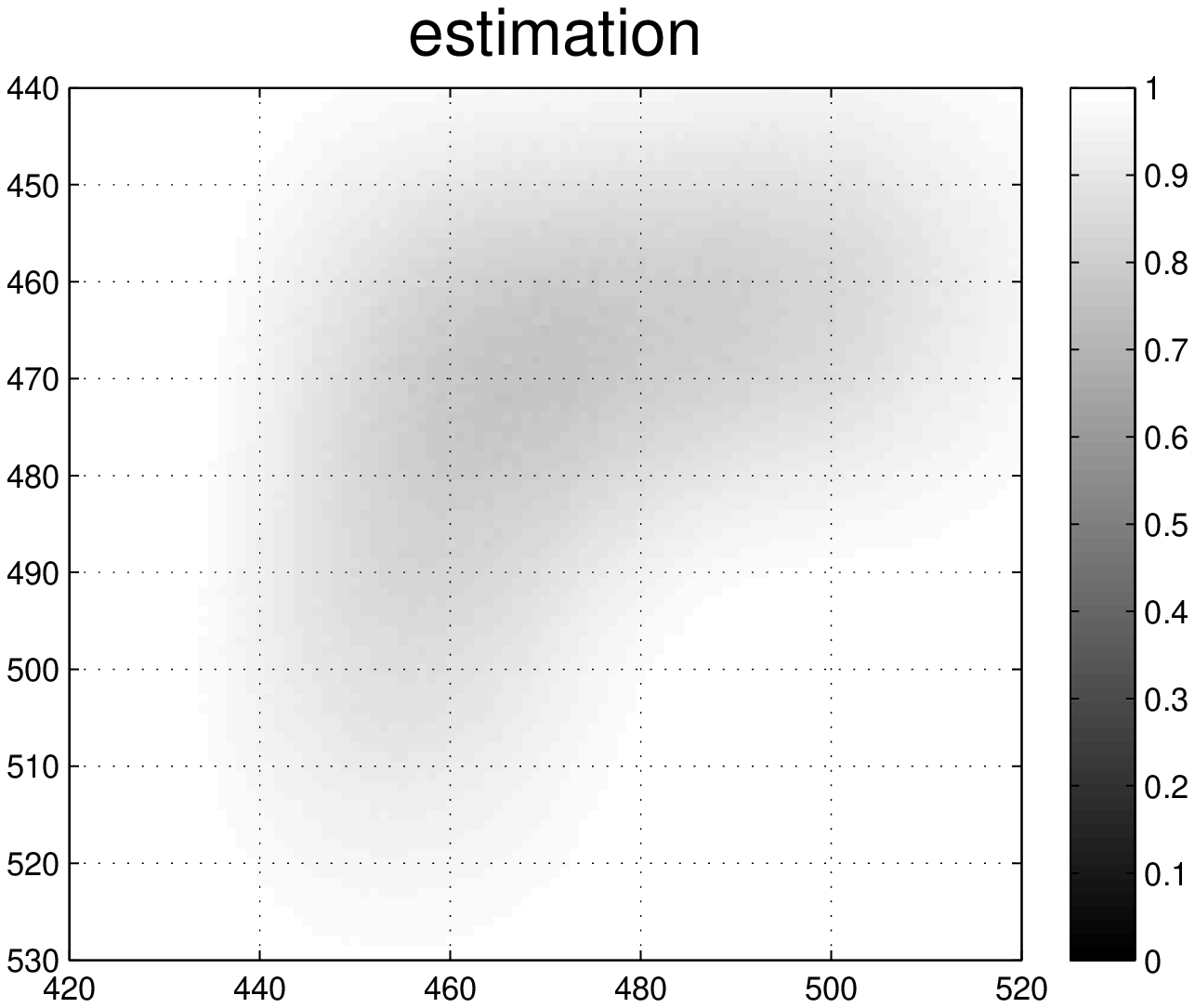} 
\end{center}
\caption{Details to the third row in Fig.~\ref{fig:6} ($T=3\,\tau$). The left 
         column shows the left top part of the big question mark and the right 
         column shows the small question mark and the small apostrophe.}
\label{fig:8}
\end{figure}


\begin{thebibliography}{10}


\bibitem{Add11}
M. Addam: 
\newblock An inverse problem for one-dimensional diffusion transport
equation in optical tomography. 
\newblock {\em preprint}, 2011.

\bibitem{Ahm07} 
Y. Ahmadizadeh:
\newblock Numerical Solution of an Inverse Diffusion Problem
\newblock {\em Applied Mathematical Sciences}, Vol. 1, 2007, no. 18, 863 - 868.

\bibitem{Alt00}
H. W. Alt:
\newblock {\em Lineare Funktionalanalysis.}
\newblock {\em Springer Verlag}, New York, 4.Auflage, 2000.

\bibitem{BroSem79}
I. N. Bronstein and K. A. Semendjajew:
\newblock {\em Taschenbuch der Mathematik.}
\newblock Harri Deutsch Verlag, Thun und Frankfurt/Main, 1979.

\bibitem{CouFriLev67}
R. Courant, K. Friedrichs and H. Lewy:
\newblock On the partial difference equations of mathematical physics.
\newblock {\em IBM J. Res. Develop.}, 11:215--234, 1967.

\bibitem{DauLio92_5}
R. Dautray and J.-L. Lions:
\newblock {\em Mathematical Analysis and Numerical Methods for Science and 
Technology. Volume 5.}
\newblock Springer-Verlag, New York, 1992.

\bibitem{Do98} 
O. Dorn: 
\newblock  A transport-backtransport method for optical tomography. 
\newblock {\em Inverse Problems 14}, 1107-1130, 1998.

\bibitem{ElaIsa97} 
A. Elayyan and V. Isakov: 
\newblock On an inverse diffusion problem. 
\newblock {\em SIAM J. Appl. Math.}, 1997, 57 1737–48.

\bibitem{EngRun95}
H. W. Engl and W. Rundell (eds.):
\newblock {\em Inverse Problems in Diffusion Processes}.
\newblock SIAM, Philadelphia, 1995.

\bibitem{EngHanNeu96}
H. W. Engl, M. Hanke and A. Neubauer: 
\newblock {\em Regularization of Inverse Problems}.
\newblock Kluwer Academic Publishers, Dordrecht, 1996.

\bibitem{Eva99}
L. C. Evans:
\newblock {\em Partial Differential Equations.}
\newblock American Mathematical Society, Providence, Rhode Island, 1999.

\bibitem{Fet80}
A. L. Fetter and J. D. Walecka: 
\newblock {\em Theoretical Mechanics of Particles and Continua.} 
\newblock McGraw-Hill, New York, 1980. 

\bibitem{GrKlLu03}
Y. A. Gryazin, M. V Klibanov and T. R Lucas:
\newblock Imaging the diffusion coefficient in a parabolic inverse problem in 
optical tomography. 
\newblock {\em Inverse Problems 15}, (1999), 373–397.


\bibitem{GuiMorRya04}
F. Guichard, J.-M. Morel and R. Ryan: 
\newblock {\em Contrast invariant image analysis and PDE's"}
\newblock Lecture notes, see http://mw.cmla.ens-cachan.fr/~morel/, 2004.

\bibitem{GasWit99}
C. Gasquet and P. Witomski:
\newblock {\em Fourier Analysis and Applications}.
\newblock Springer Verlag, New York,  1999.

\bibitem{GrKnZuCa04}
M. Griebel, S. Knapek, G. Zumbusch and A. Caglar:  
\newblock {\em Numerische Simulation in der Molek\"uldynamik}. 
\newblock Springer-Verlag, New York, 2004. 

\bibitem{Harris79}
C. J. Harris:
\newblock {\em Mathematical Modelling of Turbulent Diffusion in the Environment}. 
\newblock Academic Press, New York, 1979.

\bibitem{HeHuMc08}
B. M. C. Hetrick, R. Hughes and E. McNabb
\newblock Regularization  of the backwards heat equation via heatlets.
\newblock {\em Electronic Journal of Differential Equations}, Vol. 2008(2008), 
No. 130, pp. 1–8.

\bibitem{Heu91}
H.~Heuser:
\newblock {\em Gew\"ohnliche {D}ifferentialgleichungen}.
\newblock Teubner, Stuttgart, 2.Auflage, 1989.

\bibitem{Hoe03}
L. H\"ormander: 
\newblock {\em The Analysis of Linear Partial Differential Operators I}.
\newblock Springer Verlag, New York, 2nd edition, 2003.

\bibitem{Isa90}
V. Isakov: 
\newblock {\em  Inverse Source Problems.} 
\newblock Math. Surveys and Monographs Series, 34, AMS,Providence, RI, 1990.

\bibitem{Isa98}
V. Isakov: 
\newblock {\em Inverse Problems for Partial Differential Equations}.
\newblock Springer Verlag, New York, 1998.

\bibitem{IsaKin00} 
V. Isakov and S. Kindermann: 
\newblock Identification of the coefficient in a one-dimensional parabolic equation.
\newblock {\em Inverse Problems 16}, (2000), 665-680.

\bibitem{Kir96}
A. Kirsch: 
\newblock {\em An Introduction to the Mathematical Theory of Inverse Problems}.
\newblock Springer Verlag, New York, 1996.

\bibitem{Kow11}
R. Kowar: 
\newblock {On the causality of real-valued semigroups and diffusion.}
\newblock {\em Math. Meth. Appl. Sci. 2012}, 35 207-227, 
          (arXiv:1102.3280v1 [math.AP]).

\bibitem{NatWue01}
F.~Natterer and F.~W{\"u}bbeling:
\newblock {\em Mathematical methods in image reconstruction}.
\newblock Society for Industrial and Applied Mathematics (SIAM), Philadelphia, PA, 2001.

\bibitem{La93}
S. Lang: 
\newblock {\em Real and Functional Analysis}
\newblock Springer-Verlag, New York, 1993.

\bibitem{Lo01}
A. K. Louis:
\newblock {\em Inverse und Schlecht gestellte Probleme.}
\newblock Teubner Verlag, Stuttgart 1994. 

\bibitem{MarSch01}
V. A. Markel and J. C. Schotland:
\newblock Inverse problem in optical diffusion tomography. I. Fourier-Laplace 
inversion formulas.
\newblock {\em J. Opt. Soc. Am. A. Opt. Image Sci. Vis.}, 2001, 18(6):1336-47.

\bibitem{Sa06}
G. Sapiro:
\newblock {\em Geometric Partial Differential Equations and Image Analysis}.
\newblock Cambridge University Press, Cambridge, 2006.

\bibitem{ScGrGrHaLe09}
O. Scherzer, M. Grasmair, H. Grossauer, M. Haltmeier and F. Lenzen:
\newblock {\em  Variational Methods in Imaging}.
\newblock Springer-Verlag, New York, 2009.

\bibitem{StWi93}
U. Storch and H. Wiebe: 
\newblock {\em Lehrbuch der Mathematik. Band III.}
\newblock Wissenschaftsverlag, Mannheim, 1993.

\bibitem{WeChSuLi10}
Hui Wei, Wen Chen, Hongguang Sun and Xicheng Li: 
\newblock A coupled method for inverse source problem of spatial fractional anomalous 
diffusion equations. 
\newblock {\em Inverse Problems in Science and Engineering}, 2010, Vol. 18(7), 945–-956.

\bibitem{Wei98a}
J. Weickert: 
\newblock {\em Anistropic Diffusion in Image Processing}.
\newblock Teubner Stuttgart Verlag, Stuttgart, 1998.

\end{thebibliography}
\end{document}